\documentclass[a4paper,reqno]{amsart}

\usepackage[T1]{fontenc}
\usepackage[utf8]{inputenc}
\usepackage{lmodern}
\usepackage[english]{babel}
\usepackage[margin=2.0cm]{geometry}

\usepackage{amsmath,amsthm,amssymb,mathrsfs,yhmath,stmaryrd}
\usepackage{hhline} %just to make hlines in the table not break vertical lines
\usepackage{float} % to position the diagram precisely

\usepackage{bbold,yfonts}
\usepackage[bbgreekl]{mathbbol}

\usepackage{mathtools}
\usepackage{enumerate}
\usepackage{enumitem}
\usepackage{hyperref}
\usepackage{colonequals}
\usepackage{bbm}
\usepackage{mdwlist}
\DeclareRobustCommand\longtwoheadrightarrow
	{\relbar\joinrel\twoheadrightarrow}
\DeclareRobustCommand\longhookrightarrow
	{\lhook\joinrel\longrightarrow}

\usepackage{tikz}
\usetikzlibrary{cd}

\usepackage{scalerel}
\usepackage{stackengine,wasysym}
\usepackage{bbold}

\usepackage[bbgreekl]{mathbbol}
\usepackage{amsfonts}

\DeclareSymbolFontAlphabet{\mathbb}{AMSb}
\DeclareSymbolFontAlphabet{\mathbbl}{bbold}

\usepackage{dutchcal}

\makeatletter
\newsavebox{\@brx}
\newcommand{\llangle}[1][]{\savebox{\@brx}{\(\m@th{#1\langle}\)}%
  \mathopen{\copy\@brx\kern-0.5\wd\@brx\usebox{\@brx}}}
\newcommand{\rrangle}[1][]{\savebox{\@brx}{\(\m@th{#1\rangle}\)}%
  \mathclose{\copy\@brx\kern-0.5\wd\@brx\usebox{\@brx}}}
\makeatother

\newcommand\reallywidetilde[1]{\ThisStyle{%
  \setbox0=\hbox{$\SavedStyle#1$}%
  \stackengine{-.1\LMpt}{$\SavedStyle#1$}{%
    \stretchto{\scaleto{\SavedStyle\mkern.2mu\AC}{.5150\wd0}}{.6\ht0}%
  }{O}{c}{F}{T}{S}%
}}

\numberwithin{equation}{section}

\allowdisplaybreaks

\makeatletter
\renewcommand{\tocsection}[3]{%
	\indentlabel{\@ifnotempty{#2}{\bfseries\ignorespaces#1 #2\quad}}\bfseries#3}
\renewcommand{\tocsubsection}[3]{%
	\indentlabel{\@ifnotempty{#2}{\ignorespaces#1 #2\quad}}#3}
\newcommand\@dotsep{4.5}
\def\@tocline#1#2#3#4#5#6#7{\relax
	\ifnum #1>\c@tocdepth % then omit
	\else
	\par \addpenalty\@secpenalty\addvspace{#2}%
	\begingroup \hyphenpenalty\@M
	\@ifempty{#4}{%
		\@tempdima\csname r@tocindent\number#1\endcsname\relax
	}{%
		\@tempdima#4\relax
	}%
	\parindent\z@ \leftskip#3\relax \advance\leftskip\@tempdima\relax
	\rightskip\@pnumwidth plus1em \parfillskip-\@pnumwidth
	#5\leavevmode\hskip-\@tempdima{#6}\nobreak
	\leaders\hbox{$\m@th\mkern \@dotsep mu\hbox{.}\mkern \@dotsep mu$}\hfill
	\nobreak
	\hbox to\@pnumwidth{\@tocpagenum{\ifnum#1=1\bfseries\fi#7}}\par% <-- \bfseries for \section page
	\nobreak
	\endgroup
	\fi}
\AtBeginDocument{%
\expandafter\renewcommand\csname r@tocindent0\endcsname{0pt}
}
\def\l@subsection{\@tocline{2}{0pt}{2.5pc}{5pc}{}}
\makeatother

\makeatletter
\@namedef{subjclassname@2020}{\textup{2020} Mathematics Subject Classification}
\makeatother

\newcommand{\N}{\mathbb{N}}
\newcommand{\Z}{\mathbb{Z}}
\newcommand{\Q}{\mathbb{Q}}
\newcommand{\R}{\mathbb{R}}
\newcommand{\C}{\mathbb{C}}
\newcommand{\Hq}{\mathbb{H}}
\newcommand{\bk}{\mathbbm{k}}

\newcommand{\Mod}{\operatorname{Mod}}
\newcommand{\AMod}{{}_A\!\Mod}

\newcommand{\ModA}{\Mod_A}

\newcommand{\AModA}{{}_A\!\Mod_A}
\newcommand{\AModB}{{}_A\!\Mod_B}

\newcommand{\FGP}{\operatorname{FGP}}

\newcommand{\Diff}{\operatorname{Diff}}
\newcommand{\WD}{\operatorname{\check{Diff}}}
\newcommand{\CD}{\operatorname{\breve{Diff}}}
\newcommand{\SDiff}{\mathrm{S}\!\Diff}
\newcommand{\NDiff}{\mathrm{N}\!\Diff}
\newcommand{\Symb}{\operatorname{Symb}}
\newcommand{\WS}{\operatorname{\check{Symb}}}
\newcommand{\symb}{\varsigma}
\newcommand{\ZL}{\operatorname{ZL}}
\newcommand{\VF}{\mathfrak{X}^1_d}
\newcommand{\Ann}{\operatorname{Ann}}

\newcommand{\SJ}{\check{J}}
\newcommand{\Sj}{\check{j}}
\newcommand{\Sphat}{\check{p}}
\newcommand{\Spi}{\check{\pi}}
\newcommand{\Siota}{\check{\iota}}
\newcommand{\Sl}{\check{l}}
\renewcommand{\SS}{\check{S}}
\newcommand{\Sr}{\check{r}}
\newcommand{\Scomp}{\check{c}}
\newcommand{\WAsym}{\check{Asym}}

\DeclareFontFamily{U}{mathx}{}
\DeclareFontShape{U}{mathx}{m}{n}{<-> mathx10}{}
\DeclareSymbolFont{mathx}{U}{mathx}{m}{n}
\DeclareMathAccent{\widecheck}{0}{mathx}{"71} 

\newcommand{\CJ}{\breve{J}}%original one: \textgoth
\newcommand{\Cj}{\breve{j}}
\newcommand{\Cpi}{\breve{\pi}}
\newcommand{\Ciota}{\breve{\iota}}
\newcommand{\Col}{\breve{l}}
\newcommand{\CS}{\breve{S}}

\newcommand{\Hom}{\operatorname{Hom}}
\newcommand{\AHom}{{}_A\!\Hom}

\DeclareMathOperator*{\coker}{co{\ker}}

\newcommand{\End}{\operatorname{End}}

\newcommand{\Tor}{\operatorname{Tor}}

\newcommand{\smooth}[1]{{\mathcal{C}^{\infty}\!{({#1})}}}

\newcommand{\id}{\mathrm{id}}

\newcommand{\im}{\mathrm{Im}}

\newcommand{\op}{\mathrm{op}}
\newcommand{\comp}{\operatorname{c}}
\newcommand{\Asym}{\operatorname{Asym}}

\newcommand{\Cl}{\operatorname{\mathcal{Cl}}}
\newcommand{\vsfd}{-15pt}	%vertical spacing for function definition

\renewcommand{\Im}{\operatorname{Im}}

\theoremstyle{definition}

\newtheorem{defi}{Definition}[section]
\newtheorem{eg}[defi]{Example}

\theoremstyle{plain}

\newtheorem{theo}[defi]{Theorem}
\newtheorem{prop}[defi]{Proposition}
\newtheorem{cor}[defi]{Corollary}
\newtheorem{lemma}[defi]{Lemma}

\newtheorem*{theo*}{Theorem}
\newtheorem*{prop*}{Proposition}
\newtheorem*{cor*}{Corollary}
\newtheorem*{lemma*}{Lemma}

\theoremstyle{remark}
\newtheorem{rmk}[defi]{Remark}

\newcommand{\jetsstorderwrtd}{Proposition~4.6, p.~15}
\newcommand{\jetscorspencerdeltajes}{Corollary~8.31, p.~53}
\newcommand{\jetsproponejses}{Proposition~2.19, p.~10}
\newcommand{\jetspropconnexionsplits}{Proposition~4.10, p.~17}
\newcommand{\jetspropdifferentialoperatorcomposition}{Proposition~10.3, p.~58}
\newcommand{\jetspropfunctorialtwojetseq}{Proposition~7.15, p.~41}
\newcommand{\jetspropoperatorsalsohigherorder}{Proposition~10.2, p.~58}
\newcommand{\jetsproptensorcomparison}{Proposition~8.7, p.~43}
\newcommand{\jetsrmkholprolpi}{Remark~8.15, p.~45}
\newcommand{\jetsrmkiotaholinj}{Remark~8.3, p.~42}
\newcommand{\jetssdifferentialoperators}{§10}
\newcommand{\jetsssSpencer}{§6.3}
\newcommand{\jetssssSplitting}{§2.2.1, p.~8}
\newcommand{\jetssonejetfunctor}{§2}
\newcommand{\jetsssquaternions}{§10.2}
\newcommand{\jetstheohigherwolves}{Theorem~8.30, p.~52}
\newcommand{\jetsdiagtensorcomparison}{(8.13), p.~43}
\newcommand{\jetseqonejetd}{(2.6), p.~7}
\newcommand{\jetseqcommjlmn}{(8.30), p.~46}
\newcommand{\jetseqdefrhon}{(5.8), p.~19}
\newcommand{\jetseqDOfiltration}{(10.3), p.~58}
\newcommand{\jetseqJonedEJonedANdE}{(2.22), p.~9}
\newcommand{\jetseqjpi}{(5.7), p.~19}
\newcommand{\jetseqStwomindd}{(6.29), p.~34}
\newcommand{\jetseqStwomindefNd}{(6.28), p.~34}
\newcommand{\jetsequniquefirstorderlift}{(4.3), p.~15}
\newcommand{\jetsestwojetDt}{(7.7), p.~36}
\newcommand{\jetsexunisymmetricforms}{Example~6.13, p.~31}
\newcommand{\jetsoperatorcomposition}{(10.5), p.~58}
\newcommand{\jetssecDOzero}{§4.2}
\newcommand{\jetssinfinityjets}{§9}
\newcommand{\jetsssimplicittwojet}{§7.2}
\newcommand{\jetssssminfqsf}{§6.2.1}
\newcommand{\jetssssnotation}{§1.2.1}
\newcommand{\jetstheoclassicalnonsemi}{Theorem~5.44, p.~28}
\newcommand{\jetstheoclassicalnjet}{Theorem~8.38, p.~55}
\newcommand{\jetsprophJexact}{Proposition~8.37, p.~55}
\newcommand{\jetsdefdifferentialoperators}{Definition~10.1, p.~58}
\newcommand{\jetslemmasmM}{Lemma~8.17, p.~46}
\newcommand{\jetsdefotherprojections}{Definition~5.9, p.~20}
\newcommand{\jetspropnhjtosh}{Proposition~5.29, p.~23}
\newcommand{\jetspropStensorcomparison}{Proposition~6.15, p.~31}
\newcommand{\jetslemmauniqiota}{Lemma~8.10, p.~44}
\newcommand{\jetspropdefeth}{Proposition~7.6, p.~38}
\newcommand{\jetsproptorsiontwojetses}{Proposition~7.3, p.~36}
\newcommand{\jetspropOmJstab}{Proposition~2.27, p.~11}
\newcommand{\jetspropzeroorder}{Proposition~4.2, p.~15}
\newcommand{\jetspropuniquelift}{Proposition~4.4, p.~15}
\newcommand{\jetsdefiSterminal}{Definition~4.14, p.~17}
\newcommand{\jetsdeftwojet}{Definition~7.7, p.~39}
\newcommand{\jetsproponediffuni}{Proposition~4.3, p.~15}
\newcommand{\jetscorAopzeroorder}{Corollary~10.9, p.~59}
\newcommand{\jetscoralgebrasofDOs}{Corollary~10.8, p.~59}
\newcommand{\jetscordiffcat}{Corollary~10.6, p.~59}
\newcommand{\jetscoruniflat}{Corollary~2.25, p.~11}
\newcommand{\jetsremDOcatareenrichedinMod}{Remark~10.12, p.~59}
\newcommand{\jetsreminfinityjetfunctoronbimodules}{Remark~9.4, p.~56}
\newcommand{\jetsremjetfunctorsonbimodules}{Remark~8.2, p.~42}
\newcommand{\jetsremsemiholjetfunctoronbimodules}{Remark~5.40, p.~27}
\newcommand{\jetsrmkflpj}{Remark~2.12, p.~9}
\newcommand{\jetsrmkgammacomppi}{Remark~8.9, p.~44}
\newcommand{\jetsrmkmaximalextalg}{Remark~6.2, p.~29}
\newcommand{\jetsrmknJbi}{Remark~5.5, p.~18}
\newcommand{\jetsrmknhJexact}{Remark~5.2, p.~18}
\newcommand{\jetsrmkproldop}{Remark~2.20, p.~10}
\newcommand{\jetsssuniversalonejet}{§2.1}
\newcommand{\jetssscman}{§3.1}

\begin{document}

\title{Symbols in Noncommutative Geometry}
\author{Keegan J.~Flood, Mauro Mantegazza, Henrik Winther}
\address{Faculty of Mathematics and Computer Science\\
  UniDistance Suisse\\
  Schinerstrasse 18\\
  3900 Brig\\
  Switzerland}
  \email{keegan.flood@unidistance.ch}
\address{Department of Mathematics and Physics\\
  Charles University\\
  Sokolovsk\'{a} 49/83\\
  186 75 Prague 8\\
  Czech Republic}
  \email{mauro.mantegazza.uni@gmail.com}
\address{Department of Mathematics and Statistics\\
  UiT - The Arctic University of Norway\\
  Hansine Hansens veg 18\\
  N-9019 Tromsø\\
  Norway}
  \email{henrik.winther@uit.no}


%Keywords: jets, jet functors, symmetric forms, Spencer cohomology, partial differential operators, partial differential equations, noncommutative geometry, quantum Riemannian geometry, quantum groups, quantum flag manifolds, differential graded algebras.

\begin{abstract}
	In this paper we prove that the classical Lie bracket of vector fields can be generalized to the noncommutative setting by antisymmetrizing (in a suitable noncommutative sense) their compositions.
	This construction turns out to depend on the representability of linear differential operators, as it relies on the interpretation of vector fields as differential operators.
	In particular we provide necessary and sufficient conditions for (noncommutative) jet modules to be representing objects for differential operators.
	Furthermore, the primary ingredient for guaranteeing the closure of a bracket operation is a treatment of symbols, which classically represent, in an intrinsic way, the highest-order term of a differential operator.
	Thus, we provide an extensive theory of symbols herein.
\end{abstract}

\maketitle

\tableofcontents

\section{Introduction}
Speaking informally, the symbol of a (classical) differential operator is its leading term with respect to partial derivatives \cite{alekseevskij28geometry}.
This object turns out to transform as a (contravariant) tensor under diffeomorphisms, in contrast to the differential operator itself, which transforms in a more complicated way.
This tensoriality means that symbols should be considered as useful geometric objects in their own right.
A striking illustration of this is that the symbol of the Laplace-Beltrami operator $\Delta_g$ of a Riemannian metric $g$ turns out to be precisely the dual metric $g^{-1}$.
Indeed the generalization of the Riemannian example leads to the notion of elliptic operators, defined precisely by the nondegeneracy of their symbol, and by extension, elliptic PDE, for which a vast literature has been produced \cite{vastellipticpde}.
Going further along these lines, a Dirac operator on a (Riemannian) spinor bundle is a differential operator whose symbol is given by the Clifford multiplication \cite{cliffordsymbol}.
There is also a close relationship between connections and symbols.
Connections are equivalent to first order differential operators whose symbols are the identity map (see also Proposition \ref{prop:symbolsofconnections} for the noncommutative analogue of this fact).
Another facet of this interplay is an approach to the idea of quantization.
The algebra of symbols is a commutative one, and thus a section of the symbol map can be seen as a deformation of a commutative algebra into a noncommutative one.
Such a quantization can be facilitated, in the case of symbols of degree $2$, by a choice of a connection and a symmetric covariant tensor of degree $2$.
It turns out that many aspects of differential operators, and in turn the PDE which are associated to them, can be encoded in and computed from the symbol.
One prominent aspect of this is the theory of formal integrability and involutivity of systems of PDE, see for example \cite{GoldschmidtII, Spencer}.

The purpose of the present paper is to generalize as much of this machinery as possible to the setting of noncommutative differential geometry.
The results of this effort will be essential to a self-contained noncommutative theory of differential equations, and in particular, integrability conditions for noncommutative geometric objects.

One can view the Lie bracket as a prototype for this development.
In the setting of differential geometry, the vanishing of (the horizontal part of) the Lie bracket of vector fields, considered as sections of a vector distribution, is the Frobenius integrability of that distribution (a \textit{bona fide} integrability condition).

Previous attempts at generalizing Lie brackets typically depend on extra structure on the algebra.
For example, braided commutators in the case of triangular Hopf algebras, or graded (super)commutators in the case of graded commutative (super)algebras.
In contrast, we utilize a more intrinsic notion of antisymmetrization, which arises naturally from the wedge product (i.e.\ the multiplication map of a DGA).
It is precisely the machinery of symbols, and our associated results about representability of differential operators and symbols, which allow us to show that our notion of antisymmetrization leads to a bracket operation which takes values in vector fields.

In fact all these topics turn out to be closely connected to each other.
We would also like to emphasize that the fact that the duality between jets and differential operators is one of the most important features of classical jet bundles.
Indeed, most of the utility of jet bundles is downstream of this fact.

In our functorial approach to noncommutative differential geometry, this duality amounts to a statement about the representability of the differential operator functors $\Diff^n_d(E, -)$, which associate to each module $F$ the space of linear differential operators of order at most $n$ from $E$ to $F$.

This article is a follow up to \cite{FMW}.
The definitions of jet functors and modules, prolongation maps, differential operators, and the basic properties of and relations between thse objects in the setting of noncommutative differential geometry, can be found there.
\subsection{Some implications}
In the setting of \cite{FMW}, we consider the category of left modules $\AMod$ over a unital associative algebra $A$ to be the noncommutative generalization of the category of vector bundles over a manifold.
This is, indeed, a choice, as one might equally well have started with right modules.
However, selecting right modules at this stage is equivalent to considering the left modules over the opposite algebra $A^{op}$.
Hence, this is a mere choice of convention.

For our constructions, we also equip our algebra $A$ with an \emph{exterior algebra} $\Omega^\bullet_d$, labelled by the \emph{differential} $d\colon \Omega^\bullet_d \rightarrow \Omega^\bullet_d $.
Unlike the matter of right versus left modules, the choice of exterior algebra is essential.

There is a significant body of literature exploring the setting in which these choices have been made.
In this literature, other choices arise which appear largely arbitrary.
First, there are the two notions of right and left vector fields \cite[§2.7]{BeggsMajid}.
\begin{itemize}
\item Which of these should be the generalization of classical vector fields?
\item Second, there is the question of, given a fixed choice of first order calculus, which exterior algebra should one consider?
\end{itemize}
In this work we provide a well-justified answer to the first question, and a partial answer to the second question:
\begin{itemize}
	\item The left vector fields are the ones which naturally appear as differential operators in $\AMod$. 
	\item If one wants holonomic jets and differential operators to have a good correspondence, then the second degree $\Omega^2_d$ of the exterior algebra $\Omega^\bullet_d$ must coincide with the \emph{maximal prolongation} $\Omega^2_{\text{max}, d}$ of $\Omega^1_d$.
\end{itemize}
See §\ref{prop:fields_and_tangent_space} and §\ref{prop:elemental_jet_properties}, respectively, for the full reasoning.
	Crucially, the representability property of the jets of order $2$ is sufficient for giving a more explicit description of the symbol of a holonomic linear differential operator.

\subsection{Overview of results}
Let us consider a unital associative $\bk$-algebra over a unital commutative ring $\bk$, and let $E$ be an object in $\AMod$.
Then we have the functor $\Diff^n_d(E, -)\colon\AMod\to\Mod$, which takes a module $F$ to the space of linear differential operators of order at most $n$ from $E$ to $F$.
We investigate the following question:
\begin{itemize}
\item When is the functor $\Diff^n_d(E, -)$ representable and what is its representing object?
\end{itemize}

To address this question, we construct the natural representing object, which we call the \emph{elemental $n$-jet functor} denoted by $\SJ^n_d$, and we then investigate under what conditions the functors $\SJ^n_d$ and $J^n_d$ coincide. 
The elemental $n$-jet functor is constructed as the image of the natural transformation $\widehat{p}^n_{d}$ (cf.\ \eqref{def:phatn}), defined on an object $E$ as $\widehat{p}^n_{d, E}\colon A \otimes E \rightarrow J^n_d E$, $a\otimes e \mapsto a j^n_{d, E}(e)$.
We show in Proposition \ref{prop:elemental_jets_true_representing} that if $\Diff^n_d(E, -)$ is representable, its representing object is necessarily $\SJ^n_dE$.

When the functor $\Diff^n_d(E, -)$ is representable, we provide a characterization of linear differential operators.
Namely, given a $\bk$-linear map $\Delta\colon E\rightarrow F$ is a linear differential operator of order at most $n$, if and only if the following equation holds (cf.\ Proposition \ref{prop:criterion_WDO}):
\begin{align}
\sum_i a_i \Delta(e_i)=0
&\hfill&
\text{for all }
\sum_i a_i\otimes e_i\in N^n_d(E);
\end{align}
where $N^n_d\colonequals\ker(\widehat{p}^n_d)$ is the kernel functor applied to $\widehat{p}^n_d$, which we call the \emph{functor of differential relations of order $n$} (cf.\ Definition \ref{def:diffrelationsordern}).
This is a generalization to higher order of the condition for first order linear differential operators \cite[\jetsstorderwrtd]{FMW}.

We have a natural transformation $\widehat p^n_d |_{N^{n-1}_d}\colon  N^{n-1}_d\rightarrow S^n_d$.
If $J^{n}_d=\SJ^n_d$, then the functor of symmetric forms $S^n_d$ is generated by prolongations of differential relations (cf.\ Proposition \ref{prop:surjectivephatN} and Definition \ref{def:symgenprolrelations}) in the sense that it is a subfunctor of $\SJ^n_d$, and further satisfies (cf.\ Definition \ref{def:elementalsymforms} and Proposition \ref{prop:elementalexactsequence})
\begin{equation}
	S^n_d = N^{n-1}_d/N^n_d.
\end{equation}

Although our main goal is to further develop noncommutative differential geometry, the generality of our approach means that all constructions apply, in particular, to the classical setting of differential geometry of smooth manifolds.
Due to the classical correspondence results \cite[\jetstheoclassicalnonsemi, \jetstheoclassicalnjet]{FMW}, we may also interpret said constructions as statements about jets in the classical sense.
In this interpretation, our functorial approach is automatically fully coordinate-free, global, and diffeomorphism-invariant.
This perspective allows us to weaken regularity assumptions and generalize some classical results from vector bundles to larger classes of geometric objects, including, but not limited to, quasicoherent sheaves.

Specifially, we extend the classical representability result of differential operators from differential geometry from ${}_{\smooth{M}}\!\FGP$ (vector bundles via Serre-Swan) to the larger category ${}_{\smooth{M}}\!\Mod$, and obtain
\begin{equation}
{}_{\smooth{M}}\!\Hom(J^n_d E,F)
\simeq \Diff^n_d(E,F),
\end{equation}
where $E,F$ and $\Hom$ are interpreted as objects and morphisms in the respective categories (cf.\ Corollary \ref{cor:classicalrepresentability}).

The next line in our noncommutative development is to introduce the symbol maps for linear differential operators.
Symbols can be seen as a natural way to talk about the leading term of a differential operator.
There are two ways to introduce these objects, and we will consider both possibilities and show how they relate to each other in general.
First, we consider $\Symb^n_d(E,F)\colonequals\Diff^n_d(E,F)/\Diff^{n-1}_d(E,F)$ as the fundamental definition (cf.\ Definition \ref{def:symbolquotientdef}), and the symbol of order $n$ of a linear differential operator $\Delta$ of order at most $n$ is just its equivalence class $\symb^n_d(\Delta)=\Delta \mod \Diff^{n-1}_d(E,F)$.
This definition has the useful property that a differential operator of order at most $n$ is actually a differential operator of order $n-1$ if and only if its symbol of order $n$ vanishes.
The second possibility is to consider the map $\Delta \mapsto(\widetilde{\Delta}\circ \iota^n_{d,E}\colon S^n_dE \rightarrow F)$, defined for differential operators $\Delta\colon E \rightarrow F$ of order at most $n$.
This map has the benefit of being readily computable.
If the symmetric forms are generated by prolongations of differential relations, then the two notions are related in the following way:
The map 
\begin{align}
r^n_{d,E,F}\colon\Symb^n_d(E,F)\longrightarrow \AHom(S^n_d(E),F),
&\hfill&
\symb^n_d(\Delta)\longmapsto \widetilde{\Delta}\circ \iota^n_{d,E}
\end{align}
is well-defined and natural in $E$ and $F$ (cf.\ Proposition \ref{prop:symbols_representation_well-defined}).
Moreover, if some further homological conditions are satisfied, $r^n_{d,E,F}$ is a monomorphism and we can identify symbols of a differential operator $\Delta\colon E\to F$ with the restriction to $S^n_d(E)$ of their lift to $J^n_d E$, that is
\begin{equation}
\symb^n_d(\Delta)=\widetilde{\Delta}\circ \iota^n_{d,E}\colon S^n_d(E)\longrightarrow F.
\end{equation}
This result forms the cornerstone of our development of the bracket of vector fields.

Before we consider the bracket, we need to understand what is meant by a vector field in classical differential geometry.
Our point of view is that a vector field is a differential operator of order at most $1$ which annihilates locally constant functions (and in fact the latter forces nonzero vector fields to be of order exactly $1$).
It also happens to hold that the locally constant functions coincide with the kernel of the de Rham differential on the algebra of smooth functions.
Thus, given a first order differential calculus $d\colon A\to\Omega^1_d$, we denote the vector fields (cf.\ Definition \ref{def:vectorfields}) on $A$ by
\begin{equation}
	\VF\colonequals \Ann(\ker(d))\cap \Diff^1_d(A,A),
\end{equation}
where $\operatorname{Ann}$ is the annihilator.
There is also another point of view on classical vector fields, considered as the dual bundle to the differential $1$-forms, cf.\ \cite[§2.7]{BeggsMajid}.
This gives rise to the left and right vector fields mentioned above.
We show that vector fields in our sense are naturally identified with, and isomorphic as bimodules to, the left vector fields (cf.\ Proposition \ref{prop:tangentbundle_is_bundle}).
Once the bimodule structure is established, we show that vector fields (seen as differential operators) satisfy a generalized form of the Leibniz rule (cf.\ Proposition \ref{prop:generalized_Leibniz}),
\begin{equation}
X(ab)= a(X(b)) + (bX)(a).
\end{equation}
Of course, this Leibniz rule reduces to the classical one in the commutative setting (cf.\ Corollary \ref{cor:commutative_generalised_Leibniz_is_Leibniz}).

The vector fields do not, in general, form a subalgebra of the algebra of differential operators on $A$.
However, certain linear combinations of compositions do turn out to always be vector fields.
Our last line of development in this paper is motivated by these considerations, and leads to the development of the brackets of vector fields.

In §\ref{ss:brakets_VF}, we consider the $\bk$-tensor product $\VF \otimes \VF$.
The composition of vector fields can be considered as a map from this space to $\Diff^2_d(A,A)$.
The first step of obtaining a bracket of vector fields is to describe the preimage of $\Diff^1_d(A,A)\subseteq \Diff^2_d(A,A)$ under this map.
With the assumption that $\Omega^2_d = \Omega^2_{d, max}$, i.e.\ the maximal prolongation of $\Omega^1_d$ to second degree, we obtain that this preimage is
\begin{align}
\Asym(\VF)
= \otimes_A^{-1}(\Ann(S^2_d))
\subseteq \VF \otimes \VF,
\end{align}
i.e.\ it coincides with the preimage of the annihilator of the second degree symmetric forms over the natural projection $\otimes_A\colon \VF \otimes \VF \rightarrow \VF \otimes_A \VF$ (cf.\ Definition \ref{def:Asym} and Theorem \ref{theo:asymtodiff}).

This immediately yields the generalization of the Lie bracket in the sense of the map
\begin{equation}
	\Lambda^2 \VF \longrightarrow \VF
\end{equation}
from differential geometry, where our space $\Asym(\VF)$ plays the r\^ole of the classical $\Lambda^2 \VF$.

However, it is desirable to also see the Lie bracket as a binary operation $\VF \times \VF \rightarrow \VF$.
We accomplish this by introducing a map $\wp\colon\VF \otimes \VF \longrightarrow\Asym(\VF)$ (cf.\ Definition \ref{def:quantumliebracket}), which classically corresponds to the alternating sum $X\otimes Y\mapsto X\otimes Y-Y\otimes X$ (cf.\ Proposition \ref{prop:classicalLiebracket}).
This immediately yields a bracket of vector fields.

On the other hand, one is often interested in examples where $\Omega^2_d\ne \Omega^2_{d,max}$, (cf. \cite{woronowicz1989}).
Therefore, in §\ref{ss:elemental_brakets_VF}, we construct an analogous bracket of vector fields using the machinery of elemental differential operators.
Moreover, we provide a characterization of these brackets (cf. \ Theorem \ref{thm:bracketsarebrackets}).
In this case, we only require that $\Omega^1_d$ is flat in $\ModA$.

In §\ref{ss:bracket_examples}, we consider a few classes of examples where the choice of $\wp$ is natural, including Clifford algebras, finite group algebras, and more generally algebras equipped with inner products (cf.\ Theorem \ref{theo:inner_product_brackets} and Corollary \ref{cor:clifford_brackets}).
Of course, the skew-symmetrization map from differential geometry provides an additional example (cf.\ Proposition \ref{prop:classicalLiebracket}).

In addition to this, a few technical results concerning jets for an algebra equipped with the universal exterior algebra will be necessary for our treatment of the results outlined above. 
Since this topic is somewhat outside of the main line of development of this paper, we provide the details in Appendix \ref{section:appendix}.

\subsection{Notation}
We will mostly adopt the notation conventions of \cite[\jetssssnotation]{FMW}.
Compared to \cite{FMW}, in this paper we often treat $(A,B)$-bimodules instead of left $A$-modules, as it is the more general case.
Of course, we can obtain results for left $A$-modules and $A$-bimodules by letting $B$ be $\bk$ and $A$, respectively.
Recall that all results from \cite{FMW} can be formulated for $(A,B)$-modules, c.f.\ \cite[\jetsrmknJbi, \jetsremsemiholjetfunctoronbimodules, \jetsremjetfunctorsonbimodules, \jetsreminfinityjetfunctoronbimodules]{FMW}.

While we present the most general results in $\AModB$, we chose to treat differential operators only on $\AMod$ in order to avoid confusion.
An alternative could be to ask that a $\bk$-linear map between $(A,B)$-modules is a differential operator if it lifts to a jet module as an $(A,B)$-bilinear map.
However, even when we consider $A$-bimodules, such as jet modules or exterior algebras, we are only interested in maps between them being a differential operator only with respect to the left action, as imposing the bilinearity of the lift would severely limit the possibilities.

\subsection*{Acknowledgments}
The authors thank Paolo Aschieri, Andreas \v{C}ap, Rod Gover, Boris Kruglikov, Peter Michor, Jan Slovak, and Thomas Weber for useful discussions.
K.J.F.\ was partially supported by the priority program SPP 2026 \emph{Geometry at Infinity}
of DFG.
M.M.\ was supported by the Charles University PRIMUS grant \emph{Spectral Noncommutative Geometry of Quantum Flag Manifolds} PRIMUS/21/SCI/026.
H.W.\ was partially supported by the UiT Aurora project MASCOT.
This article/publication is based upon work from COST Action CA21109, supported by COST (European Cooperation in Science and Technology).

\section{Representability of the functors of differential operators}
\label{s:Representability}
In this section we analyze a property that characterizes jet bundles in classical differential geometry.
Namely, that of being a representative object for differential operators.
We will thus study when this property extends to the noncommutative case.
\subsection{Jet modules as representing objects}\label{ss:jetmodulesrepresentingobjects}
Let $A$ be a unital associative algebra equipped with an exterior algebra over it.
There exists a natural transformation $\widehat{p}^n_d\colon A\otimes - \to J^n_d$, explicitly defined for any $E$ in $\AModB$ as follows
\begin{align}
\label{def:phatn}
\widehat{p}^n_{d,E}\colon A\otimes E\longrightarrow J^n_d E,
&\hfill&
a\otimes e\longmapsto a j^n_{d,E} (e).
\end{align}
By construction, this map is $(A,B)$-bilinear.
\begin{defi}\label{def:diffrelationsordern}
We denote the kernel of $\widehat{p}^n_d$ by $N^n_d$ and we refer to it as the \emph{functor of differential relations of order $n$ with respect to $d$}.
\end{defi}
The relations between jet prolongations and jet projections (cf.\ \cite[\jetsrmkholprolpi]{FMW}) give us the following commutative diagram for all $m\le n$.
\begin{equation}\label{diag:phat_compatible_with_pi}
\begin{tikzcd}
A\otimes-\ar[d,"\widehat{p}^n_d"']\ar[dr,"\widehat{p}^m_d"]&[30pt]\\
J^n_d\ar[r,"\pi^{n,m}_d"]&J^m_d
\end{tikzcd}
\end{equation}

\begin{rmk}\label{rmk:prosphat}\
\begin{enumerate}
\item In low degrees, we have $N^0_d=\Omega^1_u$ and $N^1_d=N_d$, cf.\ \cite[\jetseqJonedEJonedANdE]{FMW}.
\item Commutativity of \eqref{diag:phat_compatible_with_pi} implies that for all $m\le n$ we have $N^n_d\subseteq N^m_d$.
Thus, giving the following chain of inclusions
\begin{equation}
\Omega^1_u\supseteq N_d\supseteq N^2_d\supseteq \cdots\supseteq N^n_d\supseteq \cdots.
\end{equation}
\item For $n\ge 1$, the functor $A\otimes -$ is the $n$-jet functor for the universal exterior algebra (cf.\ Theorem \ref{theo:universal_higher_jets} in Appendix \ref{section:appendix}), so we can interpret $\widehat{p}^n_d$ as a map $J^n_u\to J^n_d$.
\item For the universal exterior algebra, $N^0_u=\Omega^1_u$ and $N^n_u=0$ for all $n>0$.
\end{enumerate}
\end{rmk}

We can consider the influence of this map on lifts of differential operators.
\begin{theo}\label{theo:do_lift_uniquely}
	Let $\Omega^{\bullet}$ be an exterior algebra over $A$ and $E$ in $\AMod$.
	The following are equivalent.
	\begin{enumerate}
		\item\label{theo:do_lift_uniquely:i} The map $\widehat{p}^n_{d,E}$ is surjective;
		\item \label{theo:do_lift_uniquely:ii} The map $\widehat{p}^m_{d,E}$ is surjective for all $m\le n$ such that $\pi^{n,m}_{d,E}$ is surjective;
		\item\label{theo:do_lift_uniquely:iii} For any $F$ in $\AMod$ and $\Delta \in \Diff^n_d(E,F)$, the $n$-jet lift $\widetilde{\Delta}\in \AHom(J^n_d E,F)$ is unique;
		\item\label{theo:do_lift_uniquely:iv} For any $F$ in $\AMod$ and $\Delta \in \Diff^m_d(E,F)$ with $m \le n$ such that $\pi^{n,m}_{d,E}$ is surjective, the $m$-jet lift $\widetilde{\Delta}\in \AHom(J^m_d E,F)$ is unique.
	\end{enumerate}
\end{theo}
\begin{proof}
For \eqref{theo:do_lift_uniquely:i}$\implies$\eqref{theo:do_lift_uniquely:iii}, let $\Delta\in\Diff^n_d(E,F)$, and let $\widetilde{\Delta}_1,\widetilde{\Delta}_2\colon J^n_d E\to F$ be two lifts of $\Delta$.
By $A$-linearity, $\widetilde{\Delta}_1\circ \widehat{p}^n_{d,E}(a\otimes e)=a\Delta(e)=\widetilde{\Delta}_2\circ \widehat{p}^n_{d,E}(a\otimes e)$.
Surjectivity of $\widehat{p}^{n}_{d,E}$ implies $\widetilde{\Delta}_1=\widetilde{\Delta}_2$, whence the uniqueness of the lift.
For \eqref{theo:do_lift_uniquely:iii}$\implies$\eqref{theo:do_lift_uniquely:iv}, let $\Delta\in\Diff^m_d(E,F)$, and let $\widetilde{\Delta}_1,\widetilde{\Delta}_2\colon J^m_d E\to F$ be two lifts of $\Delta$.
By precomposing them with $\pi^{n,m}_{d,E}$, we obtain lifts of $\Delta$ to $J^n_d E$, cf.\ \cite[\jetspropoperatorsalsohigherorder]{FMW}.
By uniqueness of the lift at $n$, we get $\widetilde{\Delta}_1\circ \pi^{n,m}_{d,E}=\widetilde{\Delta}_2\circ \pi^{n,m}_{d,E}$.
Surjectivity of $\pi^{n,m}_{d,E}$ implies $\widetilde{\Delta}_1=\widetilde{\Delta}_2$, and hence uniqueness of the lift.
For \eqref{theo:do_lift_uniquely:iv}$\implies$\eqref{theo:do_lift_uniquely:ii}, consider the cokernel projection of the map $\widehat{p}^m_{d,E}$.
This is an $A$-linear map $J^m_d E\to \coker(\widehat{p}^m_{d,E})$ which vanishes on the image of $\widehat{p}^m_{d,E}$, that is $Aj^m_{d,E}(E)$.
In particular, it vanishes when precomposed with $j^m_{d,E}$.
Thus, the cokernel projection lifts $0\colon E\to F$, as does the zero map.
By \eqref{theo:do_lift_uniquely:iv}, the cokernel projection, and hence the cokernel, vanishes.
Thus, $\widehat{p}^m_{d,E}$ is surjective.
The final implication \eqref{theo:do_lift_uniquely:ii}$\implies$\eqref{theo:do_lift_uniquely:i} follows as $\pi^{n,n}_{d,E}=\id_E$ is an epi.
\end{proof}
\begin{rmk}
Points \eqref{theo:do_lift_uniquely:ii} and \eqref{theo:do_lift_uniquely:iv} hold for all $m\le n$ if in particular all $m$-jet exact sequences are exact for $m\le n$.
\end{rmk}
The composition with $j^n_{d,E}$ realizes a map $\AHom(J^n_d E,F)\to \Diff^n_d(E,F)$ which is natural in both $E$ and $F$, so in particular it induces a natural transformation of functors $\AMod\to \Mod$
\begin{equation}\label{eq:epi_defining_Diff}
\AHom(J^n_d E,-)\longtwoheadrightarrow \Diff^n_d(E,-).
\end{equation} 
This map is surjective by definition, and Theorem \ref{theo:do_lift_uniquely} provides a criterion for injectivity, making it a natural isomorphism.
In other words, under these hypotheses, the functor $\Diff^n_d(E,-)$ is representable, giving us the following.
\begin{cor}\label{cor:representability_diff_op}
	The jet module $J^n_d E$ is a representing object for linear differential operators of order at most $n$ on $E$, i.e.\ for the functor $\Diff^n_d(E,-)$, if and only if the equivalent conditions from Theorem \ref{theo:do_lift_uniquely} hold.
\end{cor}
\begin{cor}\label{cor:universal_representability}
The universal jet module $J^n_u E$ is a representing object for linear differential operators of order at most $n$ on $E$, with respect to $\Omega^1_u$.
\end{cor}
\begin{proof}
The map $\widehat{p}^n_{u,E}$ is always surjective, being $\id_{A\otimes E}$ for $n>0$, and the multiplication map for $n=0$, (cf.\ Theorem \ref{theo:universal_higher_jets} in Appendix \ref{section:appendix}).
\end{proof}

We will now assume the $n$-jet exact sequence to be left exact.
From \eqref{diag:phat_compatible_with_pi} we obtain the following diagram, where $\widehat{\iota}_{N^n_d(E)}$ is the kernel inclusion for $\widehat{p}^n_d$ and the left vertical map is obtained by the kernel universal property.
\begin{equation}\label{diag:unitogeneral}
	\begin{tikzcd}
		 0\ar[r]& N^{n-1}_d \ar[r,hook,"\widehat{\iota}_{N^{n-1}_d}"]\ar[d,"\widehat{p}^{n}_d|_{N^{n-1}_d}"'] &[30pt] A\otimes - \ar[r,"\widehat{p}^{n-1}_d"] \ar[d,"\widehat{p}^{n}_d"] &[30pt] J^{n-1}_d\ar[r,two heads]\ar[d,equals]& \coker(\widehat{p}^{n-1}_d)\\
		 0\ar[r]& S^{n}_d \ar[r,hook,"\iota^{n}_d"] & J^{n}_d \ar[r, "\pi^{n,n-1}_d"] & J^{n-1}_d
	\end{tikzcd}
\end{equation}

\begin{prop}\label{prop:surjectivephatN}
	Let the  $n$-jet sequence be left exact at $E$ in $\AModB$, then the following are equivalent:
	\begin{enumerate}
		\item\label{prop:surjectivephatN:1} $(\widehat{p}^{n}_d|_{N^{n-1}_d})_E$ is surjective;
		\item\label{prop:surjectivephatN:2} $\iota^n_{d,E}(S^{n}_d(E)) \subseteq Aj^{n}_{d,E}(E)$.
	\end{enumerate}
\end{prop}
\begin{proof}
	We evaluate \eqref{diag:unitogeneral} at $E$.
	The implication \eqref{prop:surjectivephatN:1}$\implies$ \eqref{prop:surjectivephatN:2} follows from the commutativity of the left square.
	For \eqref{prop:surjectivephatN:2}$\implies$ \eqref{prop:surjectivephatN:1}, let $\xi\in S^{n}_d(E)$.
	Since $\iota^n_{d,E}(S^{n}_d(E))\subseteq Aj^{n}_{d,E}(E)$, we can write $\xi=\sum_i a_i j^{n}_{d,E}(e_i)$ for some $a_i\in A$ and $e_i\in E$.
	Furthermore, $0=\pi^{n,n-1}_{d,E}(\xi)=\pi^{n,n-1}_{d,E}\circ \widehat{p}^{n}_{d,E}(\sum_i a_i\otimes e_i)=\widehat{p}^{n-1}_{d,E}(\sum_i a_i\otimes e_i)$.
	It follows that $\sum_i a_i\otimes e_i\in N^{n-1}_d(E)$, and thus $(\widehat{p}^{n}_d|_{N^{n-1}_d})_E$ is surjective.
\end{proof}
\begin{defi}\label{def:symgenprolrelations}
	If either of the equivalent properties of Proposition \ref{prop:surjectivephatN} holds, we will say that $E$-valued symmetric forms of degree $n$ are \emph{generated by prolongations of differential relations of order $n-1$}.
\end{defi}

\begin{cor}\label{cor:ker_and_cok}
Let the $n$-jet sequence for $E$ be left exact, and suppose that the jet module $J^{n-1}_d E$ is a representing objects for differential operators $\Diff^{n-1}_d(E,-)$.
Then the following hold:
	\begin{enumerate}
		\item $\ker\left(\widehat{p}^{n}_d|_{N^{n-1}_d}\right)_E = N^{n}_d(E)$;
		\item $\coker\left(\widehat{p}^{n}_d|_{N^{n-1}_d}\right)_E \cong \coker(\widehat{p}^{n}_{d,E})$.
	\end{enumerate}
\end{cor}
\begin{proof}
Consider the diagram \eqref{diag:unitogeneral} in degree $n$ evaluated at $E$, and compute all kernels and cokernels.
\begin{equation}\label{diag:cokfest}
	\begin{tikzcd}
	&\ker\left(\widehat{p}^{n}_d|_{N^{n-1}_d}\right)_E\ar[d,hook]&N^{n}_d(E)\ar[d,hook,"\widehat{\iota}_{N^n_d(E)}"]&[30pt]0\ar[d]\\
		 0\ar[r]& N^{n-1}_d(E) \ar[r,hook,"\widehat{\iota}_{N^{n-1}_d(E)}"]\ar[d,"\left(\widehat{p}^{n}_d|_{N^{n-1}_d}\right)_E"'] & A\otimes E \ar[r,"\widehat{p}^{n-1}_{d,E}"] \ar[d,"\widehat{p}^{n}_{d,E}"] & J^{n-1}_dE\ar[r,two heads]\ar[d,equals]& \coker(\widehat{p}^{n-1}_{d,E})\\
		 0\ar[r]& S^{n}_d(E) \ar[r,hook,"\iota^{n}_{d,E}"]\ar[d,two heads] & J^{n}_d E \ar[r, "\pi^{n,n-1}_{d,E}"]\ar[d,two heads] & J^{n-1}_d E \ar[d]\\
		 &\coker\left(\widehat{p}^{n}_d|_{N^{n-1}_d}\right)_E&\coker(\widehat{p}^{n}_{d,E})&0
	\end{tikzcd}
\end{equation}
By assumption, together with the equivalence in Theorem \ref{theo:do_lift_uniquely}, it follows that $\coker(\widehat{p}^{n-1}_{d,E})$ vanishes.
Applying the snake lemma yields the result.
\end{proof}

\begin{cor}\label{cor:symm_forms_N/N}
Given the assumptions of Corollary \ref{cor:ker_and_cok}, $J^{n}_d E$ is a representing object for $\Diff^{n}_d(E,-)$ if and only if $E$-valued symmetric forms of degree $n$ are generated by prolongations of differential relations of order ${n-1}$.
If this happens, then $S^n_d(E)\simeq (N^{n-1}_d/N^n_d)(E)$.
\end{cor}

\begin{rmk}\label{rmk:representability_low_degree}
We have that $\coker(\widehat{p}^{n}_d)=0$ for $n=0$ and $n=1$.
Thus $J^0_d$ and $J^1_d$ represent differential operators of order at most $0$ and $1$, respectively.
\end{rmk}
In degree $2$ the situation is already slightly more delicate.
In fact, we have the following results involving the maximal exterior algebra , cf.\ \cite[Lemma 1.32, p.~23]{BeggsMajid} or \cite[\jetsrmkmaximalextalg]{FMW}, and the associated minimal functor of symmetric forms, cf.\ \cite[\jetssssminfqsf]{FMW}.
\begin{prop}\label{prop:rep_for_2jet}
Let $\Omega^\bullet_d$ be an exterior algebra over $A$, then
\begin{enumerate}
\item\label{prop:rep_for_2jet:1} $\widehat{p}^2_{d}(N_d)=\iota^2_d(S^2_{d,\text{min}})\subseteq J^2_d$.
\suspend{enumerate}
Moreover, if $E$ in $\AModB$ is such that $\Tor_1^A(\Omega^1_d,E)=0$, then
\resume{enumerate}
\item\label{prop:rep_for_2jet:2} $S^2_{d}(E)$ is generated by prolongations of differential relations of order $1$ if and only if $\Omega^2_d=\Omega^2_{d,\text{max}}$;
\item\label{prop:rep_for_2jet:3} The map $\widehat{p}^2_{d,E}$ is surjective if and only if $\Omega^2_d=\Omega^2_{d,\text{max}}$;
\item\label{prop:rep_for_2jet:4} $N^2_{d}(E)=\{\sum_{i,j} a_{i,j}\otimes b_{i,j}e_j\in N_d(E)|da_{i,j}\otimes_A db_{i,j}\otimes_A e_j=0\}$, and $S^2_{d,\text{min}}(E)=N^1_d(E)/N^2_d(E)$.
\end{enumerate}
\end{prop}
\begin{proof}\
\begin{enumerate}
\item By \cite[\jetsrmkflpj]{FMW}, for all $E$ in $\AModB$, we can write a generic element of $N_d(E)$ as $\sum_{i,j} a_{i,j}\otimes b_{i,j} e_j$, where $\sum_i a_{i,j}\otimes b_{i,j}\in N_d$ and $e_j\in E$ for all $j$.
By algebraic manipulation, we have
\begin{equation}\label{eq:imageN1max}
\begin{split}
\widehat{p}^2_{d,E}\left(\sum_{i,j} a_{i,j}\otimes b_{i,j}e_j\right)
&=\sum_{i,j} [a_{i,j}\otimes 1]\otimes_A[1\otimes b_{i,j}e_j]\\
&=-\sum_{i,j} \iota^1_{d,E}(da_{i,j})\otimes_A[1\otimes b_{i,j}e_j]+\sum_{i,j} [1\otimes 1]\otimes_A[a_{i,j}\otimes b_{i,j}e_j]\\
&=-\sum_{i,j} \iota^1_{d,E}(da_{i,j}b_{i,j})\otimes_A[1\otimes e_j]+\sum_{i,j} \iota^1_{d,E}(da_{i,j})\otimes_A([b_{i,j}\otimes e_j]-[1\otimes b_{i,j}e_j])+0\\
&=0+\sum_{i,j} \iota^1_{d,E}(da_{i,j})\otimes_A\iota^1_{d,E}\rho^1_{d,E}([b_{i,j}\otimes e_j])\\
&=-\iota^2_{d,E}\left(\sum_{i,j} da_{i,j}\otimes_A db_{i,j}\otimes_A e_j\right).
\end{split}
\end{equation}
Where $\rho^1_{d,E}$ is defined at \cite[\jetseqdefrhon]{FMW}, and we can write the final equation using $\iota^2_d$, since $S^2_{d,\text{min}}\subseteq S^2_d$.
The elements thus obtained constitute $\iota^2_{d,E}(S^2_{d,\text{min}}(E))$, cf.\ \cite[\jetseqStwomindd]{FMW}.
\item $\Tor^A_1(\Omega^1_d,E)=0$ implies that the $2$-jet sequence is short exact, cf.\ \cite[\jetsssimplicittwojet]{FMW}, and thus we can consider the map $\widehat{p}^2_{d,E}|_{N_d(E)}$ as a restriction.
Under these conditions, \eqref{prop:rep_for_2jet:1} says that $\Im(\widehat{p}^2_{d,E}|_{N_d(E)})=S^2_{d,\text{min}}(E)$.
This implies that $\widehat{p}^2_{d,E}|_{N_d(E)}$ is surjective if and only if $S^2_d(E)=S^2_{d,\text{min}}(E)$, or equivalently if and only if $\Omega^2_d=\Omega^2_{d,\text{max}}$.
\item Since $\Tor_1^A(\Omega^1_d,E)=0$, the $2$-jet short exact sequence holds.
Moreover, the $1$-jet always represents differential operators, so we can apply Corollary \ref{cor:ker_and_cok}, which in particular tells us that
\begin{equation}
\coker(\widehat{p}^2_{d,E})
\cong\coker(\widehat{p}^2_{d,E}|_{N_d(E)})
=S^2_d(E)/S^2_{d,\text{min}}(E).
\end{equation}
And thus surjectivity holds if and only if $\Omega^2_d$ is maximal.
\item The first part follows from Corollary \ref{cor:ker_and_cok} together with the computation \eqref{eq:imageN1max}.
The second part is a consequence of Corollary \ref{cor:symm_forms_N/N}.\qedhere
\end{enumerate}
\end{proof}

\subsection{Classical case}
In classical differential geometry, the $n$-jet vector bundle $J^n E$ of a vector bundle $E\to M$ represents linear differential operators of order at most $n$, i.e.\ $\Diff^n(E,-)$.
However, this representability result is only assured on the category of vector bundles over $M$.
By the Serre-Swann theorem, this category is equivalent to the category of finitely generated projective $\smooth{M}$-modules.
In particular, this ensures that the $n$-jet module $J^n_{d} \Gamma(M,E)$, with respect to the de Rham exterior algebra, represents the functor of differential operators $\Diff^n_{d}(\Gamma(M,E),-)$ on the category ${}_{\smooth{M}}\!\FGP$.

In this section, we extend this representability result to the whole category ${}_{\smooth{M}}\!\Mod$.
In order to do so, we apply Corollary \ref{cor:representability_diff_op}, and in particular we show that $\widehat{p}^n_{d}$ is surjective.
Note that, in the following result, second countability is assumed in the definition of a smooth manifold.
\begin{lemma}\label{lemma:classical_jets_generated_by_prolongations}
Given a smooth manifold $M$, global sections of $J^n_d (M\times \R)$ are (finite) $\smooth{M}$-linear combinations of prolongations.
That is, the map $\widehat{p}^n_{d,\smooth{M}}\colon \smooth{M}\otimes\smooth{M}\to J^n_d \smooth{M}$ is surjective.
\end{lemma}
\begin{proof}
A manifold $M$ of dimension $m$ can always be covered by a finite atlas $\{(U_i,x_i)|i\in I\}$ of charts, cf.\ \cite[Corollary, p.~20]{greub1972connections}.
Associated to this atlas, we can always produce a partition of unity $\{\rho_i|i\in I\}$.
Let $\sigma\in\Gamma(M,J^n(M\times \R))$, and consider for each $i\in I$ the restriction $\sigma|_{U_i}\in\Gamma(U_i,J^n (M\times \R))$.
The local chart $x_i=(x_{i,1},\dots,x_{i,m})\colon U_i\to \R^m$ induces an isomorphism $\Gamma(U_i,J^n(M\times \R))\cong\Gamma(U_i,J^n(U_i\times \R))\cong \Gamma(x_i(U_i),J^n(x_i(U_i)\times\R))$, and via the latter isomorphism, we can use the so-called jet coordinates to write every prolongation of a function $f\in\smooth{U_i}$ as
\begin{align}\label{eq:jet_coordinates}
j^n(f)
=\left(f_{k_1,\dots,k_m}|(k_1,\dots,k_m)\in K_n\right).
&\hfill&
\text{where }
K_h\colonequals \left\{(k_1,\dots,k_m)\in\N^m\middle| \sum_{j=1}^m k_j\le h\right\}\text{ for }
h\ge 0.
\end{align}
Here, $f_{k_1,\dots,k_m}=\partial^{k_1}_{x_{i,1}}\cdots\partial^{k_m}_{x_{i,m}}f$, where in particular, the case where $k_h=0$ denotes the absence of the derivation $\partial_{x_{i,h}}$.
For a generic section of the $n$-jet bundle, we can represent it as a ${n+m}\choose{m}$-tuple of functions as in \eqref{eq:jet_coordinates}, where the functions $f_{k_1,\dots,k_m}\in \smooth{U_i}$ are potentially unrelated.
This means that the set of tuples of the form \eqref{eq:jet_coordinates} where only one $f_{k_1,\dots,k_m}$ is the constant $1$ and the others are $0$, is a basis for $\Gamma\left(U_i,J^n(U_i\times\R)\right)$.
We denote this particular tuple by $e_{k_1,\dots, k_m}$.
We now prove that each element of this basis can be written as a finite $\smooth{U_i}$-linear combination of prolongations, proving that the whole $\Gamma\left(U_i,J^n(U_i\times\R)\right)$ can be generated via linear combination of prolongations of a finite set of functions.

We proceed by induction on $k=\sum_{h=1}^m k_h$, called the degree, to prove that the family of monomials of coordinate functions $\{\frac{1}{k_1!\cdots k_m!}x_{i,1}^{k_1}\cdots x_{i,m}^{k_m}|(k_1,\dots,k_m)\in K_k\}$ are the functions whose prolongations generate the first ${m+k}\choose{k}$ elements of our basis.
For $k=0$, the only element of $K_0$ is $(0,\dots,0)$, so it is enough to prove the result for $e_{0,\dots,0}=(1,0,\dots,0)$, and in fact we can write it as $j^n(1)$.
Suppose the statement is true for $k-1$, we will now generate each element $e_{k_1\dots k_n}$ of degree $k$.
We have that the only nonzero component in degree $k$ of
\begin{equation}
j^n\left(\frac{1}{k_1!\cdots k_m!} x_{i,1}^{k_1}\cdots x_{i,m}^{k_m}\right)
\end{equation}
is in the component $(k_1\dots k_m)$, where it is $1$.
Every component in degree higher than $k$ vanishes.
It follows that $e_{k_1\dots k_m}-j^n(\frac{1}{k_1!\cdots k_m!} x_{i,1}^{k_1}\cdots x_{i,m}^{k_m})$ is a $\smooth{x_i(U_i)}$-linear combination of elements $e_{k_1\dots k_h}$ for $h<k$.
By inductive hypothesis, we know there exists a family of $a_{h_1\cdots h_m}\in \smooth{U_i}$, for $(h_1,\cdots, h_m)\in K_{n-1}$, such that
\begin{align}
e_{k_1\dots k_m}-j^n\left(\frac{1}{k_1!\cdots k_m!} x_{i,1}^{k_1}\cdots x_{i,m}^{k_m}\right)
=\sum_{(h_1,\cdots, h_m)\in K_{n-1}} a_{h_1\cdots h_m} j^n\left(\frac{1}{h_1!\cdots h_m!}x_{i,1}^{h_1}\cdots x_{i,m}^{h_m}\right).
\end{align}
This implies that
\begin{align}
e_{k_1\dots k_m}
=j^n\left(\frac{1}{k_1!\cdots k_m!} x_{i,1}^{k_1}\cdots x_{i,m}^{k_m}\right)
+\sum_{(h_1,\cdots, h_m)\in K_{n-1}} a_{h_1\cdots h_m} j^n\left(\frac{1}{h_1!\cdots h_m!}x_{i,1}^{h_1}\cdots x_{i,m}^{h_m}\right),
\end{align}
proving the inductive step.

We have thus shown that we can write
\begin{align}
\sigma|_{U_i}
=\sum_{(k_1,\cdots, k_m)\in K_n} a_{k_1\cdots k_m} j^n\left(\frac{1}{k_1!\cdots k_m!}x_{i,1}^{k_1}\cdots x_{i,m}^{k_m}\right),
&\hfill&
\text{for }
a_{k_1\cdots k_m}\in \smooth{U_i}.
\end{align}
Since $\rho_i$ is the constant function zero on a neighborhood of the boundary of $U_i$, we can extend each $a_{k_1\cdots k_m}$ to a function $\widetilde{a}_{k_1\cdots k_m}$ on $M$ that coincides with $a_{k_1,\dots,k_m}$ on the support of $\rho_i$ and is zero elsewhere.
Similarly, every coordinate function $x_i\colon U_i\to \R^m$ can be extended to a global function $\widetilde{x}_i$ on $M$ so that it coincides with $x_i$ on the support of $\rho_i$.
In particular, on the support of $\rho_i$ we will also obtain $j(x_{i,1}^{k_1}\cdots x_{i,m}^{k_m})=j(\widetilde{x}_{i,1}^{k_1}\cdots \widetilde{x}_{i,m}^{k_m})$, as we can see from the local description using jet coordinates.
Thus, we have
\begin{equation}
\begin{split}
\sigma
&=\sum_{i\in I}
\rho_i\sigma\\
&=\sum_{(k_1,\cdots, k_m)\in K_n} \rho_i a_{k_1\cdots k_m} j^n\left(\frac{1}{k_1!\cdots k_m!}x_{i,1}^{k_1}\cdots x_{i,m}^{k_m}\right)\\
&=\sum_{(k_1,\cdots, k_m)\in K_n} \rho_i \widetilde{a}_{k_1\cdots k_m} j^n\left(\frac{1}{k_1!\cdots k_m!}\widetilde{x}_{i,1}^{k_1}\cdots \widetilde{x}_{i,m}^{k_m}\right).
\end{split}
\end{equation}
In other words, the prolongations of functions $\widetilde{x}_{i,1}^{k_1}\cdots\widetilde{x}_{i,m}^{k_m}$ with $(k_1,\dots,k_m)\in K_n$ generate the space of sections $\Gamma(M,J^n(M\times\R))$ as a $\smooth{M}$-module.
Since the collection of these functions is finite, this completes the proof.
\end{proof}
\begin{rmk}
Lemma \ref{lemma:classical_jets_generated_by_prolongations} offers an alternative proof of the fact that the module of differential forms is generated by $\smooth{M}$ and the de Rham differential $d$.
In fact, given $\omega\in\Omega^1_{dR}\subset J^1_d(M\times \R)$, there exist $a_i,b_i\in\smooth{M}$, such that $\omega=\sum_i a_i j^1_d(b_i)$.
As noted in \cite[p.~945]{Crainic}, cf.\ \cite[\jetsrmkproldop]{FMW}, we have
\begin{equation}
\omega
=\sum_i a_i j^1_d(b_i)
=\sum_i a_i db_i +\sum_i a_i b_ij^1_d(1).
\end{equation}
Applying $\pi^{1,0}_d$ to both sides of this equation we get $\sum_i a_i b_i=0$, and thus $\omega=\sum_i a_i db_i$.
\end{rmk}
Before continuing, we will refine the result \cite[\jetsproptensorcomparison]{FMW} in the situations where the jet functors preserves epimorphisms.
E.g.\ if $\Omega^1_d$, $\Omega^2_d$, $\Omega^3_d$ are flat in $\Mod_A$, $\dh^{m-1}_d=0$ and $S^m_d$ is exact for all $2\le m\le n$, cf.\ \cite[\jetsprophJexact]{FMW}.
\begin{lemma}
\label{lemma:tensorial_comparison}
Let $\Omega^\bullet_d$ be an exterior algebra over $A$ such that the functor $J^n_d$ preserves epimorphisms.
Then, the natural transformation
\begin{equation}
\gamma^n_d\colon J^n_d A\otimes_A -\longrightarrow J^n_d
\end{equation}
is a natural epimorphism of endofunctors of $\AModB$.
\end{lemma}
\begin{proof}
By \cite[\jetsproptensorcomparison]{FMW}, which cn be extended to $(A,B)$-bimodules, we know that the map $\gamma^n_{d,P}$ is an isomorphism when $P$ is flat in $\AModB$.

Module categories have enough projectives, and thus for all $E$ in $\AModB$ there exists a surjection $\varphi\colon P\twoheadrightarrow E$, where $P$ is projective in $\AModB$.
By naturality of $\gamma^n_d$ with respect to $\varphi$, the following diagram commutes
\begin{equation}
\begin{tikzcd}
J^n_d A\otimes P\ar[r,"\gamma^n_{d,P}","\sim"']\ar[d,two heads,"\id_{J^n_d A} \otimes_A\varphi"']&J^n_d P\ar[d,two heads,"J^n_d(\varphi)"]\\
J^n_d A\otimes_A E\ar[r,"\gamma^n_{d,E}"]&J^n_d E
\end{tikzcd}
\end{equation}
Since the map $J^n_d(\varphi)$ is surjective, by the properties of epimorphisms, it follows that also $\gamma^n_{d,E}$ is surjective.
\end{proof}
\begin{prop}\label{prop:classical_jets_elemental_also_over_nonfgp_bundles}
For the de Rham exterior algebra, the natural transformation $\widehat{p}^n_d\colon \smooth{M}\otimes- \to J^n_d$ is a natural epimorphism.
\end{prop}
\begin{proof}
By Lemma \ref{lemma:classical_jets_generated_by_prolongations}, we know that the map $\widehat{p}^n_{d,\smooth{M}}\colon \smooth{M}\otimes \smooth{M}\to J^n_d \smooth{M}$ is surjective.

By \cite[\jetsprophJexact]{FMW}, since the classical case satisfies all the required assumptions, the functor $J^n_d$ is in particular exact, and thus in particular preserves epimorphisms.
We can thus apply Lemma \ref{lemma:tensorial_comparison}, telling us that there is a natural epimorphism $\gamma^n_d\colon J^n_d\smooth{M} \otimes_{\smooth{M}}-\to J^n_d$ of endofunctors of ${}_\smooth{M}\!\Mod_{\R}\cong {}_\smooth{M}\!\Mod$.
Since $\gamma^n_d$ is compatible with $l^n_{d}\colon J^n_d\hookrightarrow J^1_d\circ J^{n-1}_d$, we can prove by induction that $\gamma^n_d$ is compatible with the jet prolongation $j^n_d$, i.e.\ $j^n_{d,E}=\gamma^n_d\circ j^n_{d,\smooth{M}}\otimes_\smooth{M} id_E$.
It follows that $\gamma^n_d$ is also compatible with $\widehat{p}^n_d$, giving us the following commutative diagram.
\begin{equation}
\begin{tikzcd}[column sep=90pt]
\smooth{M}\otimes \smooth{M}\otimes_\smooth{M} E\ar[r,two heads,"\widehat{p}^n_{d,\smooth{M}}\otimes_\smooth{M} \id_E"]\ar[d,"\wr"']&J^n_d \smooth{M}\otimes_\smooth{M} E \ar[d,two heads,"\gamma^n_d"]\\
\smooth{M}\otimes E\ar[r,"\widehat{p}^n_{d,E}"']&J^n_d E
\end{tikzcd}
\end{equation}
By the properties of epimorphisms, it follows that also $\widehat{p}^n_{d,E}$ is surjective.
\end{proof}
\begin{cor}\label{cor:classicalrepresentability}
The $n$-jet functor for the de Rham exterior algebra represents the functor of differential operators of order at most $n$.
I.e., we have the following isomorphism, which is natural in $E$ and $F$ in ${}_{\smooth{M}}\Mod$
\begin{equation}
{}_{\smooth{M}}\!\Hom(J^n_d E,F)
\simeq \Diff^n_d(E,F)
\end{equation}
\end{cor}
\subsection{Lifts of the zero map}\label{ss:Lifts_of_the_zero_map}
We can be more precise in describing the obstruction to representability of differential operators.
\begin{defi}
Let $\ZL^n(E,F)$ be the kernel of the map $-\circ j^n_{d,E}\colon \AHom(J^n_d E,F)\to \Diff^n_d(E,F)$, called the \emph{module of zero lifts}.
\end{defi}
The elements of $\ZL^n(E,F)$ are $A$-linear maps $J^n_d E\to F$ that lift the zero map through $j^n_{d,E}$, whence the name.
The construction $\ZL^n$ is bifunctorial in $E$ and $F$.
In particular, $\ZL^n(E,-)$ is the kernel of the natural epimorphism \eqref{eq:epi_defining_Diff}, and therefore it is the obstruction to the representability of differential operators by $J^n_d E$.
\begin{prop}\label{prop:zero_lift_representable}
The lifts of the zero map are represented by $\coker(\widehat{p}^n_{d})$, i.e.\ there is a natural isomorphism
\begin{equation}
\ZL^n(E,-)\simeq \AHom(\coker(\widehat{p}^n_{d,E}),-).
\end{equation}
\end{prop}
\begin{proof}
Consider the following right exact sequence in $\AMod$.
\begin{equation}
\begin{tikzcd}
A\otimes E \ar[r,"\widehat{p}^{n}_{d,E}"]& J^{n}_d E\ar[r,two heads]& \coker(\widehat{p}^{n}_{d,E})\ar[r]&0
\end{tikzcd}
\end{equation}
We now apply the functor $E\mapsto \AHom(E,-)$ to this sequence.
This functor is contravariant and maps colimits into limits, as the functor $\AHom(-,F)\colon \AMod^\op\to \Mod$ is left exact for all $F$.
We thus have the following left exact sequence.
\begin{equation}
\begin{tikzcd}
0\ar[r]& \AHom(\coker(\widehat{p}^n_{d,E}),-)\ar[r,hook]& \AHom(J^{n}_dE,-)\ar[r,"-\circ\widehat{p}^{n}_{d,E}"]&[30pt] \AHom(A\otimes E,-)
\end{tikzcd}
\end{equation}
By the adjunction between extension and restriction of scalars, we have the natural isomorphism $\AHom(A\otimes -,-)\simeq \Hom(-,-)$.
Furthermore, by applying this isomorphism to $-\circ \widehat{p}^n_{d,E}$, we obtain the map $-\circ j^n_{d,E}$, whose image is $\Diff^n_d(E,-)$ by definition.
\begin{equation}
\begin{tikzcd}
0\ar[r]& \AHom(\coker(\widehat{p}^n_{d,E}),-)\ar[r,hook]& \AHom(J^{n}_dE,-)\ar[rr,"-\circ j^{n}_{d,E}"]\ar[dr,two heads]& & \Hom(E,-)\\[-10pt]
&&&\Diff^n_d(E,-)\ar[ur,hook]
\end{tikzcd}
\end{equation}
since $\AHom(\coker(\widehat{p}^n_{d,E}),-)$ is kernel of the map $-\circ j^n_{d,E}\colon \AHom(J^n_d E,F)\to \Diff^n_d(E,F)$, we get the desired natural isomorphism.
\end{proof}
Proposition \ref{prop:zero_lift_representable} is an alternative way of showing the relation between the representability of differential operators and the surjectivity of $\widehat{p}^n_{d}$, cf.\ Theorem \ref{theo:do_lift_uniquely}.
\section{Elemental jet functors and elemental linear differential operators}
\subsection{Elemental jets}\label{ss:elementaljets}
So far we have shown several criteria that establish when a jet module is a representing object for differential operators, cf.\ Corollary \ref{cor:representability_diff_op} and Proposition \ref{prop:zero_lift_representable}.
These results involve the surjectivity of the map $\widehat{p}^n_d$.
Therefore, the image of this map is itself worth studying.
\begin{defi}\label{def:elementalnjetfunctor}
Given an exterior algebra $\Omega^\bullet_d$ over an algebra $A$, we call \emph{elemental $n$-jet functor} the image of the natural transformation $\widehat{p}^n_d\colon A\otimes - \to J^n_d$ and we denote it by $\SJ^n_d$.
Given an $(A,B)$-bimodule $E$, we call the elements of $\SJ^n_d E$ \emph{elemental $n$-jets of $E$}.
\end{defi}
For all $E$, we have $\SJ^n_d E=Aj^n_d(E)$ and thus elemental jets of $E$ are $A$-linear combinations of prolongations of elements of $E$.

In the following proposition we use the properties shown in the previous §\ref{s:Representability}, translating them in terms of the elemental jet functors.
\begin{prop}\label{prop:elemental_jet_properties}\
\begin{enumerate}
\item\label{prop:elemental_jet_properties:1} $\SJ^i_d=J^i_d$ for $i=0,1$.
\item\label{prop:elemental_jet_properties:2} given $E$ in $\AMod$ such that $\Tor^A_1(\Omega^1_d,E)=0$, $\SJ^2_d E=J^2_d E$ if and only if $\Omega^2_d=\Omega^2_{d,\text{max}}$.
\item\label{prop:elemental_jet_properties:3} $\SJ^n_d(E)=A\otimes E/N^n_d(E)$ and thus for $n\ge 1$ we have $\SJ^n_d=J^n_u/N^n_d$.
\item\label{prop:elemental_jet_properties:4} Differential operators  on $E$ are representable with representing object $J^n_d E$ if and only if all jets are elemental, i.e.\ when $\SJ^n_dE=J^n_d E$.
\item\label{prop:elemental_jet_properties:5} $\SJ^n_u=J^n_u$.
\end{enumerate}
\end{prop}
\begin{proof}\
\begin{enumerate}
\item Cf.\ Remark \ref{rmk:representability_low_degree}.
\item Cf.\ Proposition \ref{prop:rep_for_2jet}.\eqref{prop:rep_for_2jet:3}.
\item It follows by definition of $\SJ^n_d E$, as $\widehat{p}^n_{d,E}$ surjects onto $\SJ^n_d E$ and has kernel $N^n_d(E)$.
\item Cf.\ Corollary \ref{cor:representability_diff_op}.
\item It follows from Corollary \ref{cor:universal_representability} and \eqref{prop:elemental_jet_properties:4}.\qedhere
\end{enumerate}
\end{proof}
\begin{rmk}
	Proposition \ref{prop:elemental_jet_properties}.\eqref{prop:elemental_jet_properties:2} indicates that the representability aspect of the jet functors behaves as in the classical setting, only in the case of the maximal exterior algebra rather than a generic exterior algebra (at least to degree $2$).
\end{rmk}

The jet prolongations $j^n_d\colon \id_{\AModB}\to J^n_d$ factor through the elemental jet functors, as for all $e\in E$ we have $j^n_{d,E}(e)=\widehat{p}^n_{d,E}(1\otimes e)$.
We will denote the factorizations by $\Sj^n_d\colon \id_{\AModB}\to \SJ^n_d$, and in general, our elemental jet notation will be the analogue of the holonomic jet notation but with a $\check{\phantom{A}}$ on top.
We denote the components of the canonical epi-mono factorization of the transformation $\widehat{p}^n_d$ by
\begin{equation}\label{eq:phatnepimonocomponents}
\begin{tikzcd}
A\otimes - \ar[r,two heads,"\Sphat^n_d"]& \SJ^n_d\ar[r,hook,"\iota_{\SJ^n_d}"]& J^n_d.
\end{tikzcd}
\end{equation}
Similarly, the jet projection of an elemental jet is an elemental jet.
Explicitly, given an elemental jet $\sum_i a_ij^n_{d,E} (e_i)\in \SJ^n_d(E)$, we have $\pi^{n,m}_{d,E}(\sum_i a_ij^n_{d,E} (e_i))=\sum_i a_ij^m_{d,E} (e_i)$.
Therefore, the restrictions of the holonomic jet projections $\pi^{n,m}_d\colon J^n_d\to J^m_d$, for $n\ge m$, to elemental jets,  factor through the inclusion $\iota_{\SJ^{m}_d}\colon \SJ^m_d\hookrightarrow J^m_d$.
We thus denote the resulting factorizations by $\Spi^{n,m}_d\colon \SJ^n_d\to \SJ^m_d$.
\begin{defi}\label{def:elementalsymforms}
We term $\SS^n_d\colonequals N^{n-1}_d/N^{n}_d$ the functor of \emph{elemental symmetric forms}.
For $n=0$ we set $\SS^0_d=\id_{\AModB}$.
\end{defi}
Notice that if we define $N^{-1}_d(E)=A\otimes E$, then the definition of $\SS^n_d$ is compatible with that of $\SS^0_d$.
\begin{prop}\label{prop:elementalexactsequence}
The jet projections between elemental jet functors are surjective.

Moreover, $\ker(\Spi^{n,n-1}_d\colon\SJ^n_d\to \SJ^{n-1}_d)=\SS^n_d$, i.e.\ the following is a short exact sequence
\begin{equation}\label{diag:elemental_jet_ses}
\begin{tikzcd}
0\ar[r]&\SS^n_d\ar[r,hook,"\Siota^n_d"]&\SJ^n_d\ar[r,two heads,"\Spi^{n,n-1}_d"]&\SJ^{n-1}_d\ar[r]&0.
\end{tikzcd}
\end{equation}
Furthermore, if $J^{n-1}_d=\SJ^{n-1}_d$, the $n$-jet sequence is left exact, and symmetric forms of degree $n$ are generated by prolongations, then $S^n_d\simeq \SS^n_d$, and thus elemental jets satisfy the jet short exact sequences.
\end{prop}
\begin{proof}
Commutativity of \eqref{diag:phat_compatible_with_pi} implies a factorization of the projection through the elemental jet functors, and thus the following commutative diagram
\begin{equation}\label{diag:elemental_jet_projection}
\begin{tikzcd}
A\otimes -\ar[r,two heads,"\Sphat^n_d"']\ar[rr,bend left,"\widehat{p}^n_d"]\ar[d,equals]&\SJ^n_d\ar[d,"\Spi^{n,n-1}_d"]\ar[r,hook,"\iota_{\SJ^n_d}"']&J^n_d\ar[d,"\pi^{n,n-1}_d"]\\
A\otimes -\ar[r,two heads,"\Sphat^{n-1}_d"]\ar[rr,bend right,"\widehat{p}^{n-1}_d"]&\SJ^{n-1}_d\ar[r,hook,"\iota_{\SJ^{n-1}_d}"]&J^{n-1}_d
\end{tikzcd}
\end{equation}
In particular, commutativity of the left square in \eqref{diag:elemental_jet_projection} tells us that its top-right composition is an epimorphism, since the bottom-left composition is.
It follows that also $\Spi^{n,n-1}_d\colon \SJ^n_d\to \SJ^{n-1}_d$ is an epi, as it is the last term in an epimorphic composition.

We now have a diagram equivalent to \eqref{diag:unitogeneral} but for elemental jet functors.
We compute kernels and cokernels of the vertical maps, and of course $\widehat{p}^n_d$ and $\Sphat^n_d$ have the same kernel, i.e.\ $N^n_d$.
Thus we obtain the following.
\begin{equation}\label{diag:elementalcokfest}
	\begin{tikzcd}
	&\ker\left(\Sphat^{n}_d|_{N^{n-1}_d}\right)\ar[d,hook]&N^{n}_d \ar[d,hook]&[30pt]0\ar[d]\\
		 0\ar[r]& N^{n-1}_d \ar[r,hook]\ar[d,"\Sphat^{n}_d|_{N^{n-1}_d}"'] & A\otimes - \ar[r,two heads,"\Sphat^{n-1}_{d,E}"] \ar[d,two heads,"\Sphat^{n}_{d}"] & \SJ^{n-1}_d\ar[r]\ar[d,equals]& 0\\
		 0\ar[r]& \ker(\Spi^{n,n-1}_d) \ar[r,hook]\ar[d,two heads] & \SJ^{n}_d \ar[r,two heads, "\Spi^{n,n-1}_{d}"]\ar[d] & \SJ^{n-1}_d \ar[d]\ar[r]&0\\
		 &\coker\left(\Sphat^{n}_d|_{N^{n-1}_d}\right)&0&0
	\end{tikzcd}
\end{equation}
By the snake lemma, $\coker\left(\Sphat^{n}_d|_{N^{n-1}_d}\right)=0$ and $\ker\left(\Sphat^{n}_d|_{N^{n-1}_d}\right)=N^n_d$.
It follows that $\ker(\Spi^{n,n-1}_d)\cong N^{n-1}_d/N^n_d=\SS^n_d$, and the bottom row of \eqref{diag:elementalcokfest} gives \eqref{diag:elemental_jet_ses}.

The final statements follow from Corollary \ref{cor:symm_forms_N/N}.
\end{proof}

Before moving to a different topic, we observe that under reasonable regularity, we can see elemental symmetric forms as symmetric forms.
\begin{lemma}\label{lemma:elemental_forms_are_forms}
If $\Omega^\bullet_d$ is an exterior algebra over $A$ such that the $n$-jet sequence is left exact, then there exists a unique monomorphism $\iota_{\SS^n_d}\colon \SS^n_d\to S^n_d$ such that the following diagram commutes
\begin{equation}
\begin{tikzcd}
	{\SS^n_d} & {\SJ^n_d}\\
	{S^n_d} & {J^n_d}
	\arrow["{\Siota^n_d}", hook, from=1-1, to=1-2]
	\arrow["{\iota^n_d}"', hook, from=2-1, to=2-2]
	\arrow["{\iota_{\SS^n_d}}"', hook, from=1-1, to=2-1]
	\arrow["{\iota_{\SJ^n_d}}", hook, from=1-2, to=2-2]
\end{tikzcd}
\end{equation}
\end{lemma}
\begin{proof}
Consider the diagram
\begin{equation}
\label{diag:elementaljetinjection}
\begin{tikzcd}
	0 & {\SS^n_d} & {\SJ^n_d} & {\SJ^{n-1}_d}&0\\
	0 & {S^n_d} & {J^n_d} & {J^{n-1}_d}
	\arrow["{\Siota^n_d}", hook, from=1-2, to=1-3]
	\arrow["{\Spi^{n,n-1}_d}", from=1-3, to=1-4]
	\arrow[from=1-4, to=1-5]
	\arrow[from=2-1, to=2-2]
	\arrow["{\iota^n_d}"', hook, from=2-2, to=2-3]
	\arrow["{\pi^{n,n-1}_d}"', from=2-3, to=2-4]
	\arrow["{\iota_{\SS^n_d}}", dashed, hook, from=1-2, to=2-2]
	\arrow["{\iota_{\SJ^n_d}}", hook, from=1-3, to=2-3]
	\arrow["\iota_{\SJ^{n-1}_d}", hook, from=1-4, to=2-4]
	\arrow[from=1-1, to=1-2]
\end{tikzcd}
\end{equation}
The right square commutes by the definition of $\Spi^{n,n-1}_d$.
The top horizontal sequence is exact by Proposition \ref{prop:elementalexactsequence} and the bottom one is left exact by hypothesis.
The existence and uniqueness of the left vertical map $\iota_{\SS^n_d}$, making the left square commute, follows by the universal property of the kernel.

Further, since both $\iota^n_d$ and $\iota_{\SJ^n_d}$ are monos, the commutativity of the left square in \eqref{diag:elementaljetinjection} implies that $\iota_{\SJ^2_d}$ is a mono.
\end{proof}
\begin{rmk}\label{rmk:conditions_lemma elemental_symm_forms}
The conditions of Lemma \ref{lemma:elemental_forms_are_forms} are satisfied for $n=2$ if $\Omega^1_d$ is flat in $\ModA$, cf.\ \cite[\jetspropfunctorialtwojetseq]{FMW}.
For generic $n$ see \cite[\jetstheohigherwolves]{FMW}.
\end{rmk}

\subsection{Elemental linear differential operators}
Analogously to \cite[\jetsdefdifferentialoperators]{FMW}, we give a notion of differential operator corresponding to elemental jet functors.
\begin{defi}
\label{def:differentialoperators}	
Let $E,F\in \AMod$.
A $\bk$-linear map $\Delta\colon E \rightarrow F$ is called a \textit{elemental linear differential operator} of order at most $n$ with respect to the exterior algebra $\Omega^\bullet_d$, if it factors through the prolongation operator $\Sj^{n}_d\colon \id_{\AMod}\to \SJ^n_d$, i.e.\ there exists an $A$-module map $\widetilde \Delta \in \AHom(\SJ^n_d E,F)$ such that the following diagram commutes:
	\begin{equation}
	\begin{tikzcd}\label{diag:universal_prop_elemental_DOs}
		\SJ_d^nE \arrow[dr, "\widetilde\Delta"] & \\
		E \arrow[r,"\Delta"] \arrow[u,"\Sj^n_{d,E}"] & F 
	\end{tikzcd}
	\end{equation}
	If $n$ is minimal, we say that $\Delta$ is an \emph{elemental linear differential operator of order $n$} with respect to the exterior algebra $\Omega^\bullet_d$.

	We denote the set of elemental linear differential operators $E\to F$ of order at most $n$ by $\WD^n_d(E,F)$, and the set of all elemental linear differential operators of finite order by $\WD_d(E,F)$.
\end{defi}

We now prove properties analogous to those holding for differential operators in the following.
\begin{prop}\label{prop:properties_elemental_DOs}
Let $\Omega^\bullet_d$ be an exterior algebra over $A$, then
\begin{enumerate}
\item\label{prop:properties_elemental_DOs:1} $\WD^n_d(E,F)$ is the image of $-\circ \Sj^n_d\colon \AMod(\SJ^n_d E,F)\to \Hom(E,F)$, and thus $\WD^n_d(E,F)$ and $\WD_d(E,F)$ are $\bk$-submodules of $\Hom(E,F)$.
\item\label{prop:properties_elemental_DOs:2} The lift of an elemental linear differential operator of order at most $n$ to the corresponding elemental $n$-jet module is unique.
\item\label{prop:properties_elemental_DOs:3} $\Diff^n_d(E,F)\subseteq \WD^n_d(E,F)$ and equality holds if $\SJ^n_d E=J^n_d E$.
The lift to the elemental jet module of a holonomic differential operator is the restriction of any lift to the corresponding holonomic jet module.
\item\label{prop:properties_elemental_DOs:4} $\WD^0_d(E,F)=\AHom(E,F)$ and $\WD^1_d(E,F)=\Diff^1_d(E,F)$.
\item\label{prop:properties_elemental_DOs:5} $\WD^m_d(E,F)\subseteq \WD^n_d(E,F)$ for all $m\le n$, thus giving a $\bk$-module filtration
\begin{equation}
\AHom(E,F)\subseteq \WD^1_d(E,F)\subseteq \cdots\subseteq \WD^n_d(E,F)\subseteq \cdots\subseteq \WD_d(E,F)\subseteq \Hom(E,F).
\end{equation}
\item\label{prop:properties_elemental_DOs:6} $\WD^0_u(E,F)=\AHom(E,F)$ and $\WD^n_u(E,F)=\Hom(E,F)$ for $n>0$.
\end{enumerate}
\end{prop}
\begin{proof}\ 
\begin{enumerate}
\item The first part follows by definition, and the second follows from the fact that $j^n_d$, and hence $\Sj^n_d$, are $\bk$-linear, and $\WD_d(E,F)=\bigcup_{n\in\N} \WD^n_d(E,F)$.
\item All elements of $\SJ^n_d E$ are of the form $\sum_i a_i \Sj^n_{d,E}(e_i)$ for some $a_i\in A$ and $e_i\in E$.
Thus, for all lifts $\widetilde{\Delta}$ of $\Delta\colon E\to F$, by $A$-linearity we have the following
\begin{equation}\label{eq:uniqueness_lift_WDO}
\widetilde{\Delta}\left(\sum_i a_i \Sj^n_d(e_i)\right)
=\sum_i a_i \widetilde{\Delta}\circ \Sj^n_{d,E}(e_i)
=\sum_i a_i \Delta(e_i).
\end{equation}
Thus the value of $\widetilde{\Delta}$ is uniquely determined by $\Delta$.
\item Let $\Delta\colon E\to F$ be a differential operator and $\widetilde{\Delta}\colon J^n_d E\to F$ be one of its lifts.
The restriction $\widetilde{\Delta}\circ \iota_{\SJ^n_d E}\colon \SJ^n_d E\to F$ lifts $\Delta$, as we have
\begin{equation}
\widetilde{\Delta}\circ \iota_{\SJ^n_d E}\circ \Sj^n_{d,E}
=\widetilde{\Delta}\circ j^n_{d,E}
=\Delta.
\end{equation}
Point \eqref{prop:properties_elemental_DOs:2} completes the proof.
\item Follows from \eqref{prop:properties_elemental_DOs:3} and Proposition \ref{prop:elemental_jet_properties}.\eqref{prop:elemental_jet_properties:1}.
\item The proof is analogous to that of \cite[\jetspropoperatorsalsohigherorder]{FMW}.
\item It follows from Proposition \ref{prop:universal_DOs} (cf. Appendix \ref{section:appendix}) and \eqref{prop:properties_elemental_DOs:4} with either \eqref{prop:properties_elemental_DOs:3} or \eqref{prop:properties_elemental_DOs:5}.
\qedhere
\end{enumerate}
\end{proof}
As a consequence of \eqref{prop:properties_elemental_DOs:1} and \eqref{prop:properties_elemental_DOs:2} of Proposition \ref{prop:properties_elemental_DOs}, we obtain the following result.
\begin{cor}\label{cor:WDO_always_representable}
For any exterior algebra $\Omega^\bullet_d$ over $A$, the functor of elemental linear differential operators of order at most $n$ is representable, with representative object $\SJ^n_d$.
I.e., the map $-\circ\Sj^n_d$ induces an isomorphism
\begin{equation}
\AHom(\SJ^n_d E,F)\simeq \WD^n_d(E,F),
\end{equation}
natural in both $E$ and $F$ in $\AMod$.
\end{cor}

We will now consider the most general possibility, that the functor of differential operators is representable, but not necessarily by the jet module.
It turns out that the modules of elemental jets are the only possible representative objects even for the functors of linear differential operators.
\begin{prop}\label{prop:elemental_jets_true_representing}
Let $\Omega^\bullet_d$ be an exterior algebra over $A$, then the functor $\Diff^n_d(E,-)$ is representable if and only if the representing object is $\SJ^n_d E$ and in $\AMod$ we have
\begin{equation}
J^n_d E=\SJ^n_d(E)\oplus \coker(\widehat{p}^n_{d,E}).
\end{equation}
\end{prop}
\begin{proof}
Consider the following diagram of functors $\AMod\to \Mod$, where we assume that the functor $\Diff^n_d(E,-)$ is representable with representing object $R$ in $\AMod$.
\begin{equation}\label{diag:boat}
\begin{tikzcd}
\AHom(J^n_d E,-)\ar[dr,two heads]\ar[ddr,bend right,"-\circ\iota_{\SJ^n_d{E}}"']\ar[rr,"-\circ j^n_{d,E}"]&&\Hom(E,-)\\
&\Diff^n_d(E,-)\simeq \AHom(R,-)\ar[ru,hook]\ar[d,hook,dashed]\\
&\WD^n_d(E,-)\simeq \AHom(\SJ^n_d(E),-)\ar[ruu,bend right,hook,"-\circ\Sj^n_{d,E}"']\\
\end{tikzcd}
\end{equation}
We obtain the solid arrows via the definition of linear differential operators and elemental linear differential operators, whereas the dashed arrow exists by the uniqueness of the epi-mono factorization (or via the lifting property).
This arrow is forced to be a monomorphism, since $-\circ\Sj^n_{d,E}$ is a monomorphism.

We can now apply the Yoneda lemma on the left side of the diagram \eqref{diag:boat}, which gives us an induced diagram in $\AMod$
\begin{equation}
\begin{tikzcd}
\SJ^n_d E\ar[dr,"\varphi"']\ar[rr,hook,"\iota_{\SJ^n_d E}"]&&J^n_d E\\
&R^n_d(E)\ar[ru]
\end{tikzcd}
\end{equation}
Since $\iota_{\SJ^n_d E}$ is an inclusion, so is the map $\varphi\colon\SJ^n_d\to R^n_d(E)$.
From \eqref{diag:boat}, we see that the precomposition with $\varphi$ must be injective.
Then consider the cokernel projection $\pi\colon R\to \coker(\varphi)$ and the zero map between the same modules.
It follows that $\pi\circ\varphi=0=0\circ\varphi$, and thus $\pi=0$, which happens if and only if $\coker(\varphi)=0$, i.e.\ if and only if $\varphi$ is an epimorphism, and hence an isomorphism.
We will thus identify $R$ with $\SJ^n_d(E)$.

Observe that in \eqref{diag:boat}, the precomposition with $\iota_{\SJ^n_d E}$ must be surjective, so in particular the identity $\id_{\SJ^n_d E}$ has a preimage, which must thus be a retraction for $\iota_{\SJ^n_d E}$.
This implies that the following short exact sequence splits
\begin{equation}
\begin{tikzcd}
0\ar[r]&\SJ^n_d(E)\ar[r,hook,"\iota_{\SJ^n_d E}"]&J^n_d(E)\ar[r,two heads]&\coker(\iota_{\SJ^n_d E})=\coker(\widehat{p}^n_{d,E})\ar[r]&0.
\end{tikzcd}
\end{equation}
\end{proof}
\begin{rmk}
If the functor of differential operators of order at most $n$ on $E$ is representable, Proposition \ref{prop:elemental_jets_true_representing} gives us the following decomposition
\begin{align}
\AHom(J^n_d E,-)
\simeq\AHom(\SJ^n_d E\oplus \coker(\widehat{p}^n_{d,E},-)
\simeq\AHom(\SJ^n_d E,-)\oplus \AHom(\coker(\widehat{p}^n_{d,E}),-).
\end{align}
Via Proposition \ref{prop:zero_lift_representable}, we obtain
\begin{align}
\AHom(J^n_d E,-)
\simeq\WD^n_d(E,-)\oplus \ZL^n(E,-)
\simeq\Diff^n_d(E,-)\oplus \ZL^n(E,-).
\end{align}
\end{rmk}
As an immediate consequence of Proposition \ref{prop:elemental_jets_true_representing}, we obtain the following.
\begin{cor}\label{cor:representability_iff_WDO=DO}
Linear differential operators of order at most $n$ on $E$ in $\AMod$ are representable if and only if all elemental linear differential operators are differential operators.
\end{cor}
In other words, the functor of elemental linear differential operators is the truly representable one, and the functor of differential operators is representable only when these two notions coincide.

\subsection{Elemental infinity jet}
\label{ss:Elemental_infinity_jet}
As in \cite[\jetssinfinityjets]{FMW}, we can define the infinity jet also for elemental jets.

Let $A$ be an associative unital $\bk$-algebra, and $\Omega^\bullet_d$ an exterior algebra over it.
Consider the following diagram, called the \emph{elemental jet tower}, in the abelian category of functors $\AModB\to \AModB$, constructed using the jet projections.
\begin{equation}\label{diag:jettower}
\begin{tikzcd}
\cdots \SJ^n_d\ar[r,two heads,"\Spi^{n,n-1}_d"]&\SJ^{n-1}_d\cdots \ar[r,two heads,"\Spi^{3,2}_d"]& \SJ^2_d \ar[r,two heads,"\Spi^{2,1}_d"] & \SJ^1_d \ar[r,two heads,"\pi^{1,0}_d"] & \id_{\AModB}.
\end{tikzcd}
\end{equation}
\begin{defi}[$\infty$-jet functor]
We call the limit of the elemental jet tower in the category of functors $\AModB\to \AModB$ the \emph{elemental $\infty$-jet functor}, denoted
\begin{equation}
\SJ^{\infty}_d\colonequals \lim_{n\in\N} \SJ^n_d.
\end{equation}
\end{defi}
As for the holonomic case, we obtain the \emph{elemental $\infty$-jet projections} and the \emph{elemental $\infty$-jet prolongation}
\begin{align}
\Spi^{\infty,n}_d\colon \SJ^\infty_d\longtwoheadrightarrow \SJ^n_d,
&\hfill&
\Sj^\infty_d\colon \id_{\AModB}\longhookrightarrow \SJ^\infty_d.
\end{align}
such that for all $n\ge m$, we have
\begin{align}
\Spi^{n,m}_d\circ\Spi^{\infty,n}_d=\Spi^{\infty,m}_d,
&\hfill&
\Spi^{\infty,n}_d\circ\Sj^{\infty}_d=\Sj^n_d.
\end{align}

\begin{rmk}
Compared to the finite order elemental jets, we cannot in general write the elemental $\infty$-jet as a quotient of $A\otimes -$.
In fact, consider the following tower of short exact sequences
\begin{equation}
\begin{tikzcd}
&\vdots\ar[d,hook]&\vdots\ar[d,equals]&\vdots\ar[d,two heads,"\Spi^{n+1,n}_d"]\\
0\ar[r]&N^n_d\ar[d,hook]\ar[r,hook]&A\otimes-\ar[d,equals]\ar[r,two heads,"\widehat{p}^n_d"]&\SJ^n_d\ar[d,two heads,"\Spi^{n,n-1}_d"]\ar[r]&0\\
0\ar[r]&N^{n-1}_d\ar[d,hook]\ar[r,hook]&A\otimes-\ar[d,equals]\ar[r,two heads,"\widehat{p}^{n-1}_d"]&\SJ^{n-1}_d\ar[d,two heads,"\Spi^{n-1,n-2}_d"]\ar[r]&0\\
&\vdots\ar[d,hook]&\vdots\ar[d,equals]&\vdots\ar[d,two heads,"\Spi^{1,0}_d"]\\
0\ar[r]&\Omega^1_u\ar[r,hook]&A\otimes-\ar[r,two heads,"\widehat{p}^0_d"]&\SJ^0_d=\id_{\AModB}\ar[r]&0
\end{tikzcd}
\end{equation}
If we now take the limit of this tower, we obtain a long exact sequence
\begin{equation}
\begin{tikzcd}
0\ar[r]&\lim\limits_{n\in\N} N^n_d\ar[r,hook]&\lim\limits_{n\in\N} A\otimes-\ar[r]&\lim\limits_{n\in\N} \SJ^n_d\ar[r]&{\lim\limits_{n\in\N}}^1 N^n_d\ar[r]&{\lim\limits_{n\in\N}}^1 A\otimes-\ar[r]&\cdots
\end{tikzcd}
\end{equation}
where ${\lim\limits_{n\in\N}}^1$ is the right derived functor of the limit functor.
Since the central and the right tower are made of surjective maps, they satisfy the Mittag-Leffler condition, implying
\begin{align}
\begin{tikzcd}
{\lim\limits_{n\in\N}}^1 A\otimes-=0,
\end{tikzcd}
&\hfill&
\begin{tikzcd}
{\lim\limits_{n\in\N}}^1 \SJ^n_d=0.
\end{tikzcd}
\end{align}
Then, since the limit of inclusions is the intersection, we have $N^{\infty}_d\colonequals\bigcap_{n\in\N} N^n_d= \lim\limits_{n\in\N} N^n_d$.
Thus we obtain the exact sequence
\begin{equation}
\begin{tikzcd}
0\ar[r]&N^{\infty}_d\ar[r,hook]& A\otimes-\ar[r]&\SJ^\infty_d\ar[r,"\widehat{p}^{\infty}_d"]&{\lim\limits_{n\in\N}}^1 N^n_d\ar[r]&0
\end{tikzcd}
\end{equation}
Thus, the map $\widehat{p}^{\infty}_d$ is surjective if and only if ${\lim\limits_{n\in\N}}^1 N^n_d=0$.
\end{rmk}
\begin{defi}
We term $N^{\infty}_d\colonequals\bigcap_{n\in\N} N^n_d= \lim\limits_{n\in\N} N^n_d$ the \emph{functor of infinite order differential relations}.
\end{defi}
\begin{rmk}
This yields two immediate notions of infinite-order elemental differential operator. 
One notion is to consider $\bk$-linear maps $\Delta\colon E\rightarrow F$ which admit $A$-linear lifts to $\SJ^{\infty}_d E$. 
The other is to consider $\bk$-linear maps $\Delta\colon E\rightarrow F$ which admit $A$-linear lifts to $A\otimes E/N^{\infty}_dE$ or, equivalently, whose universal lifts vanish on $N^{\infty}_dE$.
\end{rmk}

\subsection{Criteria for differential operators}
The form of elemental linear differential operators allows us to give the following characterization.
\begin{prop}\label{prop:criterion_WDO}
Let $\Omega^\bullet_d$ be an exterior algebra on $A$, and let $\Delta\colon E\to F$ be a $\bk$-linear map between left $A$-modules.
\begin{enumerate}
\item\label{prop:criterion_WDO:1} $\Delta$ is an elemental linear differential operator of order at most $n$ if an only if 
\begin{align}\label{eq:criterion_WDO}
\sum_i a_i \Delta(e_i)=0
&\hfill&
\text{for all }
\sum_i a_i\otimes e_i\in N^n_d(E);
\end{align}
\item\label{prop:criterion_WDO:2} A linear differential operator $\Delta$ of order at most $n$ satisfies \eqref{eq:criterion_WDO};
\item\label{prop:criterion_WDO:3} If the functor $\Diff^n_d$ is representable, then $\Delta$ is a differential operator of order $n$ if and only if \eqref{eq:criterion_WDO} holds.
\end{enumerate}
\end{prop} 
\begin{proof}\
\begin{enumerate}
\item If $\Delta$ is an elemental linear differential operator, its lift $\widetilde{\Delta}\colon \SJ^n_d E\to F$ satisfies \eqref{eq:uniqueness_lift_WDO}.
If $\sum_i a_i\otimes e_i\in N_d(E)$, then $\sum_i a_i \Sj^n_{d,E} (e_i)=\sum_i a_i j^n_{d,E} (e_i)=0$.
Thus, by linearity of $\widetilde{\Delta}$ and \eqref{eq:uniqueness_lift_WDO} we obtain \eqref{eq:criterion_WDO}.
Vice versa, consider the map $\id_A\cdot \Delta\colon A\otimes E\to F$ mapping $a\otimes e$ to $a\Delta(e)$.
This map is $A$-linear and vanishes on $N^n_d(E)$ by \eqref{eq:criterion_WDO}, so it factors to the quotient $A\otimes E/N_d(E)=\SJ^n_d$ as an $A$-linear map $\widetilde{\Delta}\colon\SJ^n_d E\to F$.
By construction, it makes \eqref{diag:universal_prop_elemental_DOs} commute, making $\Delta$ an elemental linear differential operator of order at most $n$.
\item Follows from the fact that differential operators are elemental linear differential operators.
\item Follows from Corollary \ref{cor:representability_iff_WDO=DO}.\qedhere
\end{enumerate}
\end{proof}
\begin{rmk}\
\begin{enumerate}
\item For $n=0,1$, Proposition \ref{prop:criterion_WDO} provides a characterization of differential operators of order at most $n$, cf.\ Remark \ref{rmk:representability_low_degree}.
Condition \eqref{eq:criterion_WDO} is equivalent to $A$-linearity for $n=0$, cf.\ \cite[\jetssecDOzero]{FMW}, and to that of \cite[\jetsstorderwrtd]{FMW} for $n=1$.
\item For the universal exterior algebra, $n=0$ is the only relevant case, as $N^n_u=0$ for all $n>0$.
\end{enumerate}
\end{rmk}

We can now formulate conditions for differential operators of order at most $2$.
\begin{cor}
Let $\Omega^\bullet_d$ be an exterior algebra over $A$, and let $E$ in $\AMod$ be such that $\Tor^A_1(\Omega^1_d,E)=0$.
Then the following hold.
\begin{enumerate}
\item A $\bk$-linear map $\Delta\colon E\to F$ is an elemental linear differential operator of order at most $2$ if and only if
\begin{align}\label{eq:char2WDO}
\sum_{i,j}a_{i,j}\Delta(b_{i,j} e_j)=0,
&\hfill&
\text{ for all }
\sum_{i} a_{i,j}\otimes b_{i,j}\in N_d,
e_j\in E
\text{ such that }
\sum_{i,j}da_{i,j}\otimes_A db_{i,j}\otimes_A e_j=0.
\end{align}
\item A differential operator of order at most $2$ satisfies \eqref{eq:char2WDO}.
\item If moreover $\Omega^2_d=\Omega^2_{d,\text{max}}$, then elemental linear differential operators of order at most $2$ are differential operators and they are characterized by \eqref{eq:char2WDO}.
\end{enumerate}
\end{cor}
\begin{proof}
The statements follow directly from Proposition \ref{prop:rep_for_2jet} and Proposition \ref{prop:criterion_WDO}.
\end{proof}

\subsection{Category of finite order elemental linear differential operators}
As in \cite[\jetssdifferentialoperators]{FMW}, we can interpret linear differential operators as morphisms of a category.
But before continuing, we give a result concerning holonomic jets.
\begin{lemma}\label{lemma:jets_preserving_monos}
If $\Omega^1_d$ is flat in $\ModA$, then
\begin{enumerate}
\item\label{lemma:jets_preserving_monos:1} $J^n_d$ preserves monomorphisms;
\item\label{lemma:jets_preserving_monos:2} The map $l^{n,m}_d\colon J^{n+m}_d\to J^n_d\circ J^m_d$, cf.\ \cite[\jetslemmasmM]{FMW} is a mono.
\end{enumerate}
\end{lemma}
\begin{proof}\
\begin{enumerate}
\item Let $\phi\colon E\hookrightarrow F$ be a monomorphism in $\AModB$.
By naturality of $\iota_{J^n_d}$, we obtain
\begin{equation}
\begin{tikzcd}
J^n_d E\ar[d,"J^n_d(\phi)"']\ar[r,"\iota_{J^n_d E}"]&J^{(n)}_d E\ar[d,"J^{(n)}_d(\phi)"]\\
J^n_d F\ar[r,"\iota_{J^n_d F}"]&J^{(n)}_d F
\end{tikzcd}
\end{equation}
Notice that under the given hypotheses, $\iota_{J^n_d}$ is a monomorphism, cf.\ \cite[\jetsrmkiotaholinj]{FMW}, and the functor $J^n_d$ is left exact, cf.\ \cite[\jetsrmknhJexact]{FMW}.
This implies that $J^{(n)}_d(\phi)$ is a mono, and thus by categorical properties of monos, also $J^n_d(\phi)$ is a mono.
\item We proceed by induction on $n$.
First of all $l^{0,m}_d$ is an isomorphism.
Now consider the diagram given by \cite[\jetslemmasmM]{FMW}:
\begin{equation}\label{diag:inddeflnm}
\begin{tikzcd}[column sep=40pt]
J^{n+m}_d\ar[r,hookrightarrow,"l^{n+m}_d"]\ar[d,"l^{n,m}_d"']&J^1_d\circ J^{n+m-1}_d\ar[d,"J^1_d (l^{n-1,m}_d)"]\\
J^n_d\circ J^{m}_d\ar[r,hookrightarrow,"l^n_{d,J^m_d}"]&J^1_d\circ J^{n-1}_d\circ J^m_d
\end{tikzcd}
\end{equation}
By inductive hypothesis, $l^{m-1,n}_d$ is a mono, and by \eqref{lemma:jets_preserving_monos:1}, $J^1_d(l^{n-1,m}_d)$ is also a mono.
Since also $l^{n+m}_d$ is a mono, by categorical properties of monos then also $l^{n,m}_d$ is a mono.\qedhere
\end{enumerate}
\end{proof}
The following are some more properties that elemental jet functors inherit from the holonomic ones.
\begin{lemma}\label{lemma:elemental_smm}\ 
\begin{enumerate}
\item\label{lemma:elemental_smm:1} The functor $\SJ^n_d$ preserves epis, and if $\Omega^1_d$ is flat in $\ModA$, then it also preserves monos.
\item\label{lemma:elemental_smm:2} If $\Omega^1_d$ is flat in $\ModA$, the $(A,B)$-bilinear natural inclusion $l^{n,m}_d\colon J^{n+m}_d\hookrightarrow J^n_d\circ J^m_d$ factors to an $(A,B)$-bilinear natural inclusion $\Sl^{n,m}_d\colon \SJ^{n+m}_d\hookrightarrow \SJ^n_d\circ \SJ^m_d$.
There is a unique factorization compatible with the jet prolongations, i.e.\ such that the following two diagrams commute
\begin{equation}\label{diag:compatibility_elemental_l_with_l}
\begin{tikzcd}
\SJ^{n+m}_d\ar[d,hook,"\iota_{\SJ^{n+m}_d}"']\ar[r,hook,"\Sl^{n,m}_d"]&\SJ^n_d\circ \SJ^m_d\ar[d,"(\iota_{\SJ^n_d})_{J^m_d}\circ \SJ^n_d(\iota_{\SJ^m_d})",hook]&[60pt]\id_{\AModB}\ar[r,"\Sj^m_d",hook]\ar[d,"\Sj^{n+m}_d"',hook]&\SJ^m_d\ar[d,"\Sj^n_{d,\SJ^m_d}",hook]\\
J^{n+m}_d\ar[r,"l^{n,m}_d",hook]&J^n_d\circ J^m_d&\SJ^{n+m}_d\ar[r,"\Sl^{n,m}_d",hook]&\SJ^n_d\circ \SJ^m_d
\end{tikzcd}
\end{equation}
\end{enumerate}
\end{lemma}
\begin{proof}\
\begin{enumerate}
\item Let $\phi\colon E\to F$ be a map in $\AModB$.
By definition of elemental jet functors, the following diagram commutes
\begin{equation}\label{diag:elemental_jet_on_maps}
\begin{tikzcd}
A\otimes E\ar[d,"\id_A\otimes \phi"']\ar[r,two heads,"\widehat{p}_{d,E}"]&\SJ^n_d E\ar[d,"\SJ^n_d (\phi)"]\ar[r,hook]&J^n_d E\ar[d,"J^n_d(\phi)"]\\
A\otimes F\ar[r,two heads,"\widehat{p}_{d,F}"]&\SJ^n_d F\ar[r,hook]&J^n_d F
\end{tikzcd}
\end{equation}
If $\phi$ is an epimorphism, then $\id_A\otimes \phi$ is, by right exactness of the tensoring with a module.
Thus, the diagonal of the left square in \eqref{diag:elemental_jet_on_maps} is also an epimorphism, which implies $\SJ^n_d(\phi)$ is an epi.

Now, if $\Omega^1_d$ is flat in $\ModA$ and $\phi$ is a mono, then by Lemma \ref{lemma:jets_preserving_monos}, so is $J^n_d(\phi)$.
It follows that the diagonal of the right square in \eqref{diag:elemental_jet_on_maps} is a mono, and therefore so is $\SJ^n_d(\phi)$.
\item Let $l^{n,m}_d\colon J^{n+m}_d\to J^n_d\circ J^m_d$ be as in \cite[\jetslemmasmM]{FMW}, then equation \cite[\jetseqcommjlmn]{FMW} implies that the following diagram commutes in $\AModB$.
\begin{equation}\label{diag:composed_phat}
\begin{tikzcd}[column sep=40pt]
A\otimes -\ar[r,"\id_A\otimes j^m_d"]\ar[d,"\widehat{p}^{n+m}_d"']&A\otimes J^m_d\ar[d,"\widehat{p}^n_{d,J^m_d}"]\\
J^{n+m}_d\ar[r,"l^{n,m}_d"]&J^n_d\circ J^m_d
\end{tikzcd}
\end{equation}
Commutativity of \eqref{diag:composed_phat} induces a factorization through the images of the vertical maps, that is a map $\widehat{l}^{n,m}_d$ making the following diagram commutative.
\begin{equation}\label{diag:composed_phat_factorisation}
\begin{tikzcd}[column sep=40pt]
A\otimes -\ar[r,"\id_A\otimes j^m_d"]\ar[d,two heads,"\Sphat^{n+m}_d"']&A\otimes J^m_d\ar[d,two heads,"\Sphat^n_{d,J^m_d}"]\\
\SJ^{n+m}_d\ar[r,"\widehat{l}^{n,m}_d"]\ar[d,hook,"\iota_{\SJ^{n+m}_d}"']&\SJ^n_d\circ J^m_d\ar[d,hook,"(\iota_{\SJ^n_d})_{J^m_d}"]\\
J^{n+m}_d\ar[r,"l^{n,m}_d"]&J^n_d\circ J^m_d
\end{tikzcd}
\end{equation}
By applying the functor $A\otimes -$ to the factorization of $j^n_d$ through $\SJ^n_d$, we can provide a factorization for the map $\id_A\otimes j^,_d$ through $A\otimes \SJ^n_d$.
We can now add to \eqref{diag:composed_phat_factorisation} the following maps.
\begin{equation}\label{diag:wanted_elemental_l}
\begin{tikzcd}[column sep=40pt]
A\otimes -\ar[r,"\id_A\otimes \Sj^m_d"]\ar[d,two heads,"\Sphat^{n+m}_d"']&A\otimes \SJ^m_d\ar[d,two heads,"\Sphat^n_{d,\SJ^m_d}"]\ar[r,"\id_A\otimes \iota_{\SJ^n_d}"]&A\otimes J^m_d\ar[d,two heads,"\Sphat^n_{d,J^m_d}"]\\
\SJ^{n+m}_d\ar[rr,bend right=10pt,"\widehat{l}^{n,m}_d"']\ar[r,dashed,"\Sl^{n,m}_d"]\ar[d,hook,"\iota_{\SJ^{n+m}_d}"']&\SJ^n_d\circ \SJ^m_d \ar[r,"\SJ^n_d( \iota_{\SJ^m_d})"]&\SJ^n_d\circ J^m_d\ar[d,hook,"(\iota_{\SJ^n_d})_{J^m_d}"]\\
J^{n+m}_d\ar[rr,"l^{n,m}_d"']&&J^n_d\circ J^m_d
\end{tikzcd}
\end{equation}
The inner right square commutes by naturality of $\Sphat^n_d$ with respect to the inclusion $\iota_{\SJ^n_d}\colon\SJ^n_d\hookrightarrow J^n_d$.
We are now left to find a dashed map in \eqref{diag:wanted_elemental_l} that commutes in the diagram.
The existence of $\Sl^{n,m}_d$ follows from the lifting properties of epimorphisms and monomorphisms in the category of $\AModB$.
More explicitly, in the hypotheses of \eqref{lemma:elemental_smm:2}, we can apply \eqref{lemma:elemental_smm:1} to obtain that $\SJ^n_d(\iota_{\SJ^n_d})$ is a monomorphism.
Thus, we obtain the following square, where the vertical left map is an epi and the right one is a mono.
\begin{equation}\label{diag:llp_elemental_l}
\begin{tikzcd}[column sep=60pt]
A\otimes -\ar[d,two heads,"\Sphat^{n+m}_d"']\ar[r,"\Sphat^n_{d,J^m_d}\circ(\id_A\otimes \Sj^m_d)"]&\SJ^n_d\circ \SJ^m_d\ar[d,hook,"\SJ^n_d(\iota_{\SJ^n_d})"]\\
\SJ^{n+m}_d\ar[ur,dashed,"\Sl^{n,m}_d"]\ar[r,"\widehat{l}^{n,m}_d"']&\SJ^n_d\circ J^m_d
\end{tikzcd}
\end{equation}
Thus $\Sl^{n,m}_d$ exists and is unique by the left lifting property of epimorphism with respect to monomorphisms in abelian categories.
The commutativity of \eqref{diag:compatibility_elemental_l_with_l} follows from that of \eqref{diag:wanted_elemental_l}.
In particular, the left square of \eqref{diag:compatibility_elemental_l_with_l} is the bottom square of \eqref{diag:wanted_elemental_l}, and the right square is obtained by restricting  the top right square of \eqref{diag:wanted_elemental_l} to $\bk\otimes -\cong \id_{\AModB}$.

There is a unique $\Sl^{n,m}_d$ such that the squares \eqref{diag:compatibility_elemental_l_with_l} commute.
More specificly, it is enough to give the right square, as the values of an $(A,B)$-linear function on $\SJ^{n+m}_d$ are entirely determined by its values on the image of $\Sj^{n+m}_d$.
The map $\Sl^{n,m}_d$ must be a mono, since both $\iota_{\SJ^{n+m}_d}$ and $l^{n,m}_d$ are.\qedhere
\end{enumerate}
\end{proof}
We are now prepared to state the following result.
\begin{prop}\label{prop:cat_properties_WDOs}
Let $\Omega^\bullet_d$ be an exterior algebra over $A$ such that $\Omega^1_d$ is flat in $\ModA$, then
\begin{enumerate}
\item\label{prop:cat_properties_WDOs:1} The composition of linear differential operators extends to a composition of elemental linear differential operators
\begin{equation}
\circ \colon \WD^m_d(F,G)\times \WD^n_d(E,F)\longrightarrow \WD^{n+m}_d(E,G).
\end{equation}
\item\label{prop:cat_properties_WDOs:2} There is a category $\WD_d$ with the same objects as $\AMod$ and with morphisms between $E$ and $F$ given by $\WD_d(E,F)$.
\end{enumerate}
\end{prop}
\begin{proof}\
\begin{enumerate}
	\item The construction of the lift of the composition is the elemental jet analogue of that found in \cite[\jetspropdifferentialoperatorcomposition]{FMW}:
		Let $\Delta_1 \in \WD^n_d(E,F)$ and $\Delta_2 \in \WD^m_d(F,G)$.
		Then the elemental jet lift of the composition $\Delta_2 \circ \Delta_1$ is given explicitly by \begin{equation}
			\reallywidetilde{\Delta_2 \circ \Delta_1} = \widetilde{\Delta_2} \circ \SJ^m_d(\widetilde{\Delta_1}) \circ \Sl^{m,n}_{d,E},
			\label{eq:elementaljetcompositionlift}
		\end{equation}
		which we deduce from the following commutative diagram.
\begin{equation}\label{eq:elementaljetsoperatorcomposition}
\begin{tikzcd}
\SJ^{m+n}_d E\ar[r,"\Sl^{m,n}_{d,E}"]&\SJ^m_d(\SJ^n_dE) \arrow[rd,near start, "\SJ^m_d(\widetilde{\Delta}_1)"] \\
&\SJ^n_dE \arrow[rd,near start, "\widetilde\Delta_1"] \arrow[u,"\Sj^m_{d,\SJ^n_d E}"] &       \SJ^m_dF \arrow[rd, near start, "\widetilde\Delta_2"]     &[20pt]\\
&E\ar[uul,"\Sj^{m+n}_{d,E}"] \arrow[r, "\Delta_1"] \arrow[u,"\Sj^n_{d,E}"'] &       F \arrow[r, "\Delta_2"] \arrow[u,"\Sj^m_{d,F}"]   &       G
\end{tikzcd}
\end{equation}
The map $\Sl^{m,n}_d\colon \SJ^{n+m}_d\to\SJ^m_d\circ \SJ^n_d$ is given by \eqref{lemma:elemental_smm:2} of Lemma \ref{lemma:elemental_smm}.
\item This follows directly from \eqref{prop:cat_properties_WDOs:1}.\qedhere
\end{enumerate}
\end{proof}
\begin{defi}
We call the category $\WD_d$ described by Lemma \ref{prop:cat_properties_WDOs}.\eqref{prop:cat_properties_WDOs:2}, the \emph{category of finite order elemental linear differential operators}.
\end{defi}
We immediately obtain the following properties
\begin{prop}\label{prop:include_and_forget}
Let $\Omega^1_d$ be flat in $\ModA$, then
\begin{enumerate}
\item there are faithful inclusions of categories $\AMod\subseteq\Diff_d\subseteq\WD_d$;
\item there is a faithful forgetful functor $\WD_d\hookrightarrow \Mod$;
\item the restriction of the forgetful functors $\WD_d\hookrightarrow \Mod$ and $\Diff_d\hookrightarrow \Mod$ to $\AMod$ is the forgetful functor;
\end{enumerate}
\end{prop}
\begin{proof}\
\begin{enumerate}
\item The inclusions are given by \eqref{prop:properties_elemental_DOs:3} and \eqref{prop:properties_elemental_DOs:5} of Proposition \ref{prop:properties_elemental_DOs}.
\item Obtained via Proposition \ref{prop:properties_elemental_DOs}.\eqref{prop:properties_elemental_DOs:5}.
\item Follows by construction.
\qedhere
\end{enumerate}
\end{proof}
We have the following result concerning finite limits.
\begin{prop}
Let $\Omega^1_d$ be flat in $\ModA$, then the inclusions $\AMod\subseteq\WD_d$ preserves finite limits.

Given a cone for a diagram, the universal arrow to the limit is a differential operator of order at most equal to the maximum of the orders of the differential operators of the cone.
\end{prop}
\begin{proof}
Let $I$ be a finite category and consider a diagram $F\colon I\to \AMod$.
Let $\lim_I F$ be the limit in $\AMod$, we will now show that this is also the limit in $\WD_d$.
First of all, by Proposition \ref{prop:include_and_forget}, we know that applying the forgetful functor $\WD_d\to \Mod$ to $F$ is equivalent to applying the forgetful functor $\AMod\to \Mod$ to it.
The latter is a right adjoint, and thus it preserves limits.
Now let $E$ be in $\WD_d$ and consider a cone $\Delta\colon E\to F$ in $\WD_d$, where $E$ here is interpreted as the constant diagram.
Since $\lim_I F$ is also the limit in $\Mod$, we know that there is a unique compatible $\Delta_{!}\colon E\to \lim_I F$ in $\Mod$.
We need to prove that it is a differential operator.
For all $i\in I$, we have $\Delta_i\in\WD^{n_i}_d(E,F(i))$ for some finite order $n_i$.
Let $n\colonequals \max\{n_i|i\in I\}$, then for all $i\in I$, we have $\Delta_i\in \WD^n_d(E,F(i))\cong \AHom(\SJ^n_d E,F(i))$, thus corresponding to a cone $\widetilde{\Delta}\colon \SJ^n_d E\to F$ in $\AMod$.
The limit universal property ensures the existence of a unique $\widetilde{\Delta}_{!}\colon \SJ^n_d E\to \lim_I F$ in $\AMod$ commuting with the cone.
\begin{equation}
\begin{tikzcd}
\SJ^n_d E\ar[rd,dashed,"\widetilde{\Delta}_!"']\ar[rrd,"\widetilde{\Delta}_i"]\ar[rrrd,bend left=10pt,"\widetilde{\Delta}_{i'}"]\\
E\ar[u,"\Sj^n_{d,E}",hook]\ar[r,dashed,"\Delta_!"']&\lim\limits_I F\ar[r,"\pi_i"']\ar[rr,bend right,"\pi_{i'}"']&F(i)\ar[r]&F(i')
\end{tikzcd}
\end{equation}
Now, for all $i\in I$,
\begin{align}
\pi_i\circ \widetilde{\Delta}_!\circ\Sj^n_{d,E}
=\widetilde{\Delta}_i\circ\Sj^n_{d,E}
=\Delta_i
=\pi_i \circ\Delta_!.
\end{align}
By the universal property of the limit in $\Mod$, we must thus have $\widetilde{\Delta}_!\circ\Sj^n_{d,E}=\Delta_!$, making $\Delta_!$ an elemental linear differential operator of order at most $n$.
\end{proof}
This result does not necessarily hold for infinite limits, as the universal arrow is not necessarily a differential operator of finite order.

\subsection{Primitive jets}
Similarly to the definition of elemental jets, we define the functor $\CJ^n_d$ as the image of the natural transformation $\widehat{p}^{(n)}_d$ defined by
\begin{align}\label{eq:nonholelemental}
\widehat{p}^{(n)}_{d,E}\colon A\otimes E\longrightarrow J^{(n)}_d E,
&\hfill&
a\otimes e\longmapsto aj^{(n)}_d(e),
\end{align}
which we term the \emph{primitive $n$-jet functor}.

We also define the following functor: $N^{(n)}_d\colonequals\ker(\widehat{p}^{(n)}_d)$, and notice that as for .
\begin{rmk}
Further, one can consider the semiholonomic analogue of these constructions, but they coincide with $\CJ^n_d$ and $N^{(n)}_d$, since $\iota_{J^{[n]}_d}\colon J^{[n]}_d\hookrightarrow J^{(n)}_d$ is always a monomorphism.
\end{rmk}

The relations between nonholonomic jet prolongations and jet projections (cf.\ \cite[\jetseqjpi]{FMW}) give us the following commutative diagram for all $m\le n$.
\begin{equation}\label{diag:phatn_compatible_with_pi}
\begin{tikzcd}
A\otimes-\ar[d,"\widehat{p}^{(n)}_d"']\ar[dr,"\widehat{p}^{(m)}_d"]&[30pt]\\
J^{(n)}_d\ar[r,"\pi^{(n,m)}_d"]&J^{(m)}_d
\end{tikzcd}
\end{equation}

Analogously to Remark \ref{rmk:prosphat}, we obtain the following
\begin{rmk}\label{rmk:prosphatn}\
\begin{enumerate}
\item In low degrees, we have $N^{(0)}_d=\Omega^1_u$ and $N^{(1)}_d=N_d$.
\item Commutativity of \eqref{diag:phat_compatible_with_pi} implies that for all $m\le n$ we have $N^{(n)}_d\subseteq N^{(m)}_d$.
\end{enumerate}
\end{rmk}

We can relate this new notion of jet functor with the elemental one via the following result.
\begin{lemma}\label{lemma:skelprimitivecompare}
We have the following:\
\begin{enumerate}
\item\label{lemma:skelprimitivecompare:1} The map $\iota_{J^n_d}\colon J^{n}_d\to J^{(n)}_d$ restricts to a natural epimorphism
\begin{equation}\label{eq:skelprimitivecompare}
\breve{p}^n_d\colon \SJ^n_d\longtwoheadrightarrow \CJ^n_d
\end{equation}
which is a natural isomorphism if $\iota_{J^n_d}\colon J^n_d\to J^{(n)}_d$ is a monomorphism.
\item\label{lemma:skelprimitivecompare:2} There exists a natural inclusion
\begin{equation}\label{eq:holnonholNds}
N^n_d\longhookrightarrow N^{(n)}_d
\end{equation}
which is compatible with the inclusions of the same map with lower $n$ and is a natural isomorphism if and only if $\breve{p}^n_d$ is a natural isomorphism.
\item\label{lemma:skelprimitivecompare:3} $\CJ^n_d E=A\otimes E/N^{(n)}_d$, and thus for $n>1$ we have $\CJ^n_d=J^n_u/N^{(n)}_d$.
\item\label{lemma:skelprimitivecompare:4} For $n=0,1,2$, the map $\breve{p}^n_d$ is a natural isomorphism.
\item\label{lemma:skelprimitivecompare:5} If $\Omega^1_d$ is flat in $\ModA$, then $\breve{p}^n_d$ is a natural isomorphism for all $n\in\N$.
\item\label{lemma:skelprimitivecompare:6} $\CJ^n_u=\SJ^n_u=J^n_u=A\otimes E$ for all $n\in\N$.
\end{enumerate}
\end{lemma}
\begin{proof}\
\begin{enumerate}
\item By compatibility of the arrow $\iota_{J^n_d}\colon J^{n}_d\to J^{(n)}_d$ with the jet prolongation, we have the commutativity of the outer square in the following diagram
\begin{equation}\label{diag:primitiveskelcomp}
\begin{tikzcd}
\ar[r]A\otimes -\ar[rr, bend left,"\widehat{p}^n_d"]\ar[r,two heads]\ar[d,equals]&\SJ^n_d\ar[d,dashed,two heads ,"\breve{p}^n_d"']\ar[r,hookrightarrow] & J^n_d \ar[d,"\iota_{J^n_d}"]\\
A\otimes - \ar[rr,bend right,"\widehat{p}^{(n)}_d"']\ar[r,two heads]&\CJ^n_d\ar[r,hookrightarrow]&J^{(n)}_d
\end{tikzcd}
\end{equation}
This commutativity induces the compatible map $\breve{p}^n_d$ between the images of the horizontal maps in \eqref{diag:primitiveskelcomp}, by the properties of the epi-mono factorization.
The map $\breve{p}^n_d$ is an epimorphism because it is the last map in a composition which is itself an epimorphism.

When $\iota_{J^n_d}$ is a monomorphism, then $\breve{p}^n_d$ is also a monomorphism, as it is the first map in a composition which is a monomorphism.

\item We obtain the desired map as follows, via the kernel universal property.
\begin{equation}
\begin{tikzcd}
N^n_d\ar[d,dashed,hook]\ar[r,hook]&A\otimes -\ar[d,equals] \ar[r,two heads]&\SJ^n_d\ar[d,two heads ,"\breve{p}^n_d"]\\
N^{(n)}_d\ar[r,hook]&A\otimes - \ar[r,two heads]&\CJ^n_d
\end{tikzcd}
\end{equation}
By the snake lemma, this map is a monomorphism, and its cokernel is isomorphic to the kernel of $\breve{p}^n_d$.
This implies the last part of the statement.
Universality yields the compatibility with the other inclusion maps.

\item It follows by definition of $\CJ^n_d$, as $\widehat{p}^n_d$ surjects onto $\CJ^n_d$ with kernel $N^{(n)}_d$.

\item For $n=0,1$, $J^{(n)}_d=J^n_d$, so the statement is trivial.
For $n=2$, we have that $\iota_{J^2_d}=l^2_d$ is a monomorphism, so $\breve{p}^n_d$ is an isomorphism by \eqref{lemma:skelprimitivecompare:1}.
Thus, $\breve{p}^n_d$ gives the desired isomorphism.
\item Using a similar argument, if $\Omega^1_d$ is flat in $\Mod_A$, then $\iota_{J^{n}_d}$ is a monomorphism, cf.\ \cite[\jetsrmkiotaholinj]{FMW}, implying $\SJ^n_d\simeq\CJ^n_d$ via \ref{lemma:skelprimitivecompare:1}.
\item By Theorem \ref{theo:universal_higher_jets} (cf. Appendix \ref{section:appendix}), we know that $\iota_{J^n_u}$ is a monomorphism, so the result follows from \eqref{lemma:skelprimitivecompare:1}.\qedhere
\end{enumerate}
\end{proof}
\begin{rmk}
By construction, $\CJ^n_d$ only depends on the first grade $\Omega^1_d$ of the exterior algebra.
Hence, Lemma \ref{lemma:skelprimitivecompare} implies that for $n\le 2$, or for all $n$ when $\Omega^1_d$ flat in $\ModA$, all the information required to construct elemental jet functors $\SJ^n_d$ is contained solely in $d\colon A\rightarrow \Omega^1_d$.
\end{rmk}

Analogously to the case of elemental jet, we can define a jet prolongation
\begin{equation}
\Cj^n_d\colon \id_{\AModB}\longhookrightarrow \CJ^n_d,
\end{equation}
as the unique factorization of $j^{(n)}_d\colon \id_{\AModB}\to J^{(n)}_d$ through $\CJ^n_d$.

The diagram \eqref{diag:phatn_compatible_with_pi} implies that $\pi^{(n,m)}_d$ factors uniquely as the projection map
\begin{equation}
\Cpi^{n,m}_d\colon \CJ^n_d\longtwoheadrightarrow \CJ^m_d.
\end{equation}
These maps are natural epimorphism, since for all $E$ in $\AModB$, any element $\widehat{p}^{(m)}_d(a\otimes e)\in\CJ^m_d E$, can be seen as $\Cpi^{n,m}_d\circ \widehat{p}^{(n)}_d(a\otimes e)$.

Furthermore, we see that any map $\pi^{(n,n-1;m)}_d$ yields the same factorization $\Cpi^{n,n-1}_d$ in virtue of the following result.
\begin{prop}
We have an inclusion $\CJ^n_d\subseteq J^{[n]}_d$.
\end{prop}
\begin{proof}
It is enough to prove that the compositions $\pi^{(n,n-1;m)}_d\circ \widehat{p}^n_d$ coincide for all $m$, cf.\ \cite[\jetsdefotherprojections]{FMW}.
This is true because the same happens for $j^{(n)}_d$ in place of $\widehat{p}^n_d$.
Cf.\ \cite[\jetspropnhjtosh]{FMW}.
\end{proof}

We can also give an associated notion of differential operator
\begin{defi}
\label{def:primitivedifferentialoperators}	
Let $E,F\in \AMod$.
A $\bk$-linear map $\Delta\colon E \rightarrow F$ is called a \textit{primitive linear differential operator} of order at most $n$ with respect to the first order differential calculus $\Omega^1_d$, if it factors through the prolongation operator $\Cj^n_d\colon \id_{\AMod}\to \CJ^n_d$, i.e.\ there exists an $A$-module map $\widetilde \Delta \in \AHom(\CJ^n_d E,F)$ such that the following diagram commutes:
	\begin{equation}
	\begin{tikzcd}\label{diag:universal_prop_primitive_DOs}
		\CJ_d^nE \arrow[dr, "\widetilde\Delta"] & \\
		E \arrow[r,"\Delta"] \arrow[u,"\Cj^n_{d,E}"] & F 
	\end{tikzcd}
	\end{equation}
	If $n$ is minimal, we say that $\Delta$ is a \emph{primitive linear differential operator of order $n$} with respect to the first order differential calculus $\Omega^1_d$.

	We denote the set of primitive linear differential operators $E\to F$ of order at most $n$ by $\CD^n_d(E,F)$, and the set of all primitive linear differential operators of finite order by $\CD_d(E,F)$.
\end{defi}

As for elemental jets, the primitive jets are also representative objects for the functor of primitive differential operators.
\begin{prop}\label{prop:properties_primitive_DOs}
Let $\Omega^1_d$ be a first order differential calculus over $A$, then
\begin{enumerate}
\item\label{prop:properties_primitive_DOs:1} $\CD^n_d(E,F)$ is the image of $-\circ \Cj^n_d\colon \AMod(\CJ^n_d E,F)\to \Hom(E,F)$, and thus $\CD^n_d(E,F)$ and $\CD_d(E,F)$ are $\bk$-submodules of $\Hom(E,F)$.
\item\label{prop:properties_primitive_DOs:2} The lift of a primitive linear differential operator of order at most $n$ to the corresponding elemental $n$-jet module is unique.
\item\label{prop:properties_primitive_DOs:3} The functor of primitive linear differential operators of order at most $n$ is representable, with representative object $\CJ^n_d$.
I.e., the map $-\circ\Cj^n_d$ induces an isomorphism
\begin{equation}
\AHom(\CJ^n_d E,F)\simeq \CD^n_d(E,F),
\end{equation}
natural in both $E$ and $F$ in $\AMod$.
\item\label{prop:properties_primitive_DOs:4}
\begin{equation}\label{diag:inclusion_primitive_DOs}
\begin{tikzcd}[column sep=2pt,row sep=8pt]
&\Hom\\
&\WD^n_d\ar[u,hook]\\
\Diff^n_d\ar[ur,hook]&&\CD^n_d\ar[ul,hook]\\
&\SDiff^n_d\ar[ul,hook]\ar[ur,hook]\\
&\NDiff^n_d\ar[u,hook]
\end{tikzcd}
\end{equation}
\item\label{prop:properties_primitive_DOs:5} $\CD^0_d(E,F)=\AHom(E,F)$ and $\CD^1_d(E,F)=\WD^1_d(E,F)=\Diff^1_d(E,F)$, $\CD^2_d(E,F)=\WD^2_d(E,F)$, and if $\Omega^1_d$ is flat in $\ModA$, then $\WD^n_d= \CD^n_d(E,F)$ for all $n$.
\item\label{prop:properties_primitive_DOs:6} $\CD^m_d(E,F)\subseteq \CD^n_d(E,F)$ for all $m\le n$, thus giving a $\bk$-module filtration
\begin{equation}
\AHom(E,F)\subseteq \WD^1_d(E,F)\subseteq \cdots\subseteq \WD^n_d(E,F)\subseteq \cdots\subseteq \WD_d(E,F)\subseteq \Hom(E,F).
\end{equation}
\item\label{prop:properties_primitive_DOs:7} $\CD^0_u(E,F)=\AHom(E,F)$ and $\CD^n_u(E,F)=\Hom(E,F)$ for $n>0$.
\item\label{prop:properties_primitive_DOs:8} $\Delta$ is a primitive linear differential operator of order at most $n$ if an only if 
\begin{align}
\sum_i a_i \Delta(e_i)=0
&\hfill&
\text{for all }
\sum_i a_i\otimes e_i\in N^{(n)}_d(E);
\end{align}
\end{enumerate}
\end{prop}
\begin{proof}
We can prove \eqref{prop:properties_primitive_DOs:1}, \eqref{prop:properties_primitive_DOs:2}, \eqref{prop:properties_primitive_DOs:6}, and \eqref{prop:properties_primitive_DOs:7}, \textit{mutatis mutandis} as in Proposition \ref{prop:properties_elemental_DOs}.

Point \eqref{prop:properties_primitive_DOs:8} follows \textit{mutatis mutandis} from Proposition \ref{prop:criterion_WDO}.\eqref{prop:criterion_WDO:1}.

Point \eqref{prop:properties_primitive_DOs:3} follows from \eqref{prop:properties_primitive_DOs:1} and \eqref{prop:properties_primitive_DOs:2}.

For point \eqref{prop:properties_primitive_DOs:4}, we consider the following diagram of jet functors
\begin{equation}
\begin{tikzcd}
A\otimes -\ar[r,two heads]&\SJ^n_d\ar[r,two heads]\ar[d,hook]&\CJ^n_d\ar[d,hook]\\
&J^n_d\ar[r]&J^{[n]}_d\ar[r,hook]&J^{(n)}_d
\end{tikzcd}
\end{equation}
The jet prolongations form a cone from $E$ onto this diagram, which yields the diagram \eqref{diag:inclusion_primitive_DOs} of inclusions of functors of differential operators.

Point \eqref{prop:properties_primitive_DOs:5} follows from representability, cf.\ \eqref{prop:properties_primitive_DOs:3}, Corollary \ref{cor:WDO_always_representable}, together with Lemma \ref{lemma:skelprimitivecompare}, \eqref{lemma:skelprimitivecompare:4} and \eqref{lemma:skelprimitivecompare:5}.
\end{proof}

Just as for the other notions of jets, given enough regularity, we can include the jet functor into a composition of jet functors
\begin{lemma}
\label{lemma:primitive_smm}\ 
\begin{enumerate}
\item\label{lemma:primitive_smm:1} The functor $\CJ^n_d$ preserves epis, and if $\Omega^1_d$ is flat in $\ModA$, then it also preserves monos.
\item\label{lemma:primitive_smm:2} If $\Omega^1_d$ is flat in $\ModA$, then the $(A,B)$-bilinear natural inclusion $\Sl^{n,m}_d\colon \SJ^{n+m}_d\hookrightarrow \SJ^n_d\circ \SJ^m_d$ turns into an $(A,B)$-bilinear natural inclusion $\Col^{n,m}_d\colon \CJ^{n+m}_d\hookrightarrow \CJ^n_d\circ \CJ^m_d$.
It can also be constructed explicitly as $a\Cj^{n+m}_{d,E}(e)\mapsto a\Cj^{n}_{d,\Cj^m_d E}\circ \Cj^{m}_d(e)$ for $E$ in $\AMod$ and $a\in A$ and $e\in E$.
\end{enumerate}
\end{lemma}
\begin{proof}\
\begin{enumerate}
\item This point is proven \textit{mutatis mutandis} as in Lemma \ref{lemma:elemental_smm}.\eqref{lemma:elemental_smm:1}.
\item It follows as a consequence of Lemma \ref{lemma:elemental_smm}.\eqref{lemma:elemental_smm:2} together with Lemma \ref{lemma:skelprimitivecompare}.\eqref{lemma:skelprimitivecompare:5}.\qedhere
\end{enumerate}
\end{proof}

Analogously to the elemental case, we can also define a notion of symmetric forms.
\begin{defi}\label{def:primitivesymforms}
We term $\CS^n_d\colonequals N^{(n-1)}_d/N^{(n)}_d$ the functor of \emph{primitive symmetric forms}.
For $n=0$ we set $\CS^0_d=\id_{\AModB}$.
\end{defi}
Notice that if we define $N^{(-1)}_d(E)=A\otimes E$, then the definition of $\CS^n_d$ is compatible with that of $\CS^0_d$.
\begin{prop}\label{prop:symmformprimitiveskelcomparison}
The inclusion \eqref{eq:holnonholNds} induces a natural transformation
\begin{equation}\label{eq:symmformprimitiveskelcomparison}
\breve{q}^n_d\colon \SS^n_d\longrightarrow \CS^n_d.
\end{equation}
\begin{enumerate}
\item\label{prop:symmformprimitiveskelcomparison:1} For $n=0,1,2$, the map $\breve{q}^n_d$ is an isomorphism and for $n=3$, it is an epimorphism.
\item\label{prop:symmformprimitiveskelcomparison:2} If $\Omega^1_d$ is flat in $\ModA$, the map $\breve{q}^n_d$ is an isomorphism for all $n$.
\end{enumerate}
\end{prop}
\begin{proof}
By Lemma \ref{lemma:skelprimitivecompare}.\eqref{lemma:skelprimitivecompare:2}, we have that the wight square in the following diagram commutes
\begin{equation}\label{diag:symmformprimitiveskelcomparison}
\begin{tikzcd}
N^n_d\ar[d,hook]\ar[r,hook]&N^{n-1}_d\ar[d,hook] \ar[r,two heads]&\SS^n_d\ar[d,dashed,"\breve{q}^n_d"]\\
N^{(n)}_d\ar[r,hook]&N^{(n-1)}_d \ar[r,two heads]&\CS^n_d
\end{tikzcd}
\end{equation}
The universal property of the cokernel gives the dashed map.

For \eqref{prop:symmformprimitiveskelcomparison:1}, the case $n=0$ follows by definition of $\CS^0_d$.
For $n=1,2,3$ it follows from Lemma \ref{lemma:skelprimitivecompare}, points \eqref{lemma:skelprimitivecompare:2} and \eqref{lemma:skelprimitivecompare:4} together with the snake lemma applied to \eqref{diag:symmformprimitiveskelcomparison}.

For \eqref{prop:symmformprimitiveskelcomparison:2}, it follows analogously applying Lemma \ref{lemma:skelprimitivecompare}.\eqref{lemma:skelprimitivecompare:5}
\end{proof}

\begin{prop}\label{prop:primitiveexactsequence}
The primitive jet projections $\Cpi^{n,n-1}_d$ are epimorphisms and $\ker(\Cpi^{n,n-1}_d\colon\CJ^n_d\to \CJ^{n-1}_d)=\CS^n_d$.
Moreover, the following is a map of short exact sequence
\begin{equation}\label{diag:primitive_jet_ses}
\begin{tikzcd}
0\ar[r]&\SS^n_d\ar[d,"\breve{q}^n_d"']\ar[r,hook,"\Siota^n_d"]&\SJ^n_d\ar[d,two heads,"\breve{p}^n_d"]\ar[r,two heads,"\Spi^{n,n-1}_d"]&\SJ^{n-1}_d\ar[r]\ar[d,two heads,"\breve{p}^{n-1}_d"]&0\\
0\ar[r]&\CS^n_d\ar[r,hook,"\Ciota^n_d"]&\CJ^n_d\ar[r,two heads,"\Cpi^{n,n-1}_d"]&\CJ^{n-1}_d\ar[r]&0
\end{tikzcd}
\end{equation}
If $\breve{p}^n_d$ is an isomorphism, then the map $\breve{q}^n_d$ is a monomorphism, and if $\breve{p}^{n-1}_d$ is an isomorphism, then the map $\breve{q}^n_d$ is an epimorphism.
\end{prop}
\begin{proof}
The exactness of the top sequence follows from Proposition \ref{prop:elementalexactsequence}, while the exactness of the bottom one follows in the same way \textit{mutatis mutandis}.
The commutativity of \eqref{diag:primitive_jet_ses} follows from direct computation.

The final statements of the proposition follow from applying the snake lemma to \eqref{diag:primitive_jet_ses}.
\end{proof}

\begin{rmk}
Analogously to §\eqref{ss:Elemental_infinity_jet}, one can develop the theory of the infinity primitive jet \textit{mutatis mutandis}.
\end{rmk}
\section{Symbols of differential operators}
In the spirit of \cite[§5.2, p.~86]{alekseevskij28geometry}, we define the notion of symbol of a differential operator.
\begin{defi}[Symbol of a differential operator]\label{def:symbolquotientdef}
Let $E$ and $F$ be in $\AMod$, we define
\begin{equation}
\Symb^n_d(E,F)\colonequals \Diff^n_d(E,F)/\Diff^{n-1}_d(E,F),
\end{equation}
for all $n\ge 0$, with the convention $\Diff^{-1}_d(E,F)=0$.
Let $\symb^n_d\colon \Diff^n_d(E,F)\to \Symb^n_d(E,F)$ be the quotient projection.
The \emph{symbol of a differential operator} $\Delta\in\Diff^n_d(E,F)$ is the equivalence class $\symb^n_d(\Delta)$ containing $\Delta$.

We can give analogous definitions for the nonholonomic, semiholonomic, elemental, and primitive differential operators.
\end{defi}
\begin{rmk}
The quotient $\Symb^n_d$ is bifunctorial in $E$ and $F$, and the maps $\symb^n_d=\symb^n_{d,E,F}$ constitute a natural epimorphism $\symb^n_d\colon \Diff^n_d\twoheadrightarrow\Symb^n_d$.
\end{rmk}

\subsection{Representability of symbols}\label{ss:representabilitysymbols}
We will now give a more explicit representation of symbols.
In order to do that, we will assume
\begin{equation}\label{eq:symforms_gen_by_proj}
\im(\iota^n_{d,E})\subseteq Aj^n_{d}(E)=\SJ^n_dE.
\end{equation}
\begin{rmk}
The property \eqref{eq:symforms_gen_by_proj} is satisfied in particular if $J^n_d E$ is a representing object for differential operators, cf.\ Corollary \ref{cor:representability_diff_op}.
Another sufficient condition for \eqref{eq:symforms_gen_by_proj} to hold  is that the $n$-jet sequence at $E$ is left exact and $S^n_d(E)$ is generated by prolongations of differential relations of order $n-1$, cf.\ Proposition \ref{prop:surjectivephatN}.
\end{rmk}
We can characterise this property in terms of lifts of the zero map, cf.\ §\ref{ss:Lifts_of_the_zero_map}.
\begin{prop}\label{prop:characterization_symm_generated_by_prol}
The following are equivalent:
\begin{enumerate}
\item\label{prop:characterization_symm_generated_by_prol:1} $\im(\iota^n_{d,E})\subseteq Aj^n_{d}(E)=\SJ^n_dE$;
\item\label{prop:characterization_symm_generated_by_prol:2} For all $F$ in $\AMod$, and for all lifts $z\colon J^n_d E\to F$ of the $0$ map, $z\circ \iota^n_{d,E}=0$;
\item\label{prop:characterization_symm_generated_by_prol:3} The composition of $\iota^n_{d,E}$ with the cokernel projection $J^n_d E\twoheadrightarrow \coker(\widehat{p}^n_{d,E})\cong J^n_d E/\SJ^n_d E$ vanishes.
\end{enumerate}
\end{prop}
\begin{proof}\
\item [\eqref{prop:characterization_symm_generated_by_prol:1}$\Rightarrow$\eqref{prop:characterization_symm_generated_by_prol:2}]For any lift $z\colon J^n_d E\to F$ of the $0$ map, consider $z\circ \iota^n_{d,E}$.
For all $\sigma\in S^n_d E$, by \eqref{prop:characterization_symm_generated_by_prol:1}, we can write $\iota^n_d (\sigma)=\sum_i a_i j^n_d (e_i)$.
Therefore, $z\circ \iota^n_{d,E}(\sigma)=z(\sum_i a_i j^n_{d,E} (e_i))=\sum_i a_i z\circ j^n_{d,E} (e_i))=0$.
\item [\eqref{prop:characterization_symm_generated_by_prol:2}$\Rightarrow$\eqref{prop:characterization_symm_generated_by_prol:3}]
This step follows from the fact that the projection is a lift of $0$, as $j^n_{d,E}$ lands in its kernel.
\item [\eqref{prop:characterization_symm_generated_by_prol:3}$\Rightarrow$\eqref{prop:characterization_symm_generated_by_prol:1}] If the mentioned composition vanishes, $\iota^n_{d,E}$ factors through the kernel of the projection $J^n_d E\twoheadrightarrow J^n_d /\SJ^n_d$, that is $\SJ^n_d E$, and therefore so does its image.
\end{proof}
\begin{prop}\label{prop:symbols_representation_well-defined}
If $\im(\iota^n_{d,E})\subseteq \SJ^n_{d}(E)$, then the mapping $r^n_{d,E,F}$ defined by
\begin{align}\label{eq:rmap}
r^n_{d,E,F}\colon\Symb^n_d(E,F)\longrightarrow \AHom(S^n_d(E),F),
&\hfill&
\symb^n_d(\Delta)\longmapsto \widetilde{\Delta}\circ \iota^n_{d,E}.
\end{align}
is well-defined and natural in $E$ and $F$.
Moreover,
\begin{enumerate}
\item\label{prop:symbols_representation_well-defined:1} if the $n$-jet sequence at $E$ is right exact, then $r^n_{d,E,-}$ is a monomorphism;
\item\label{prop:symbols_representation_well-defined:2} $r^n_{d,E,-}$ is surjective if and only if $\iota^n_{d,E}$ is a section (admits a left inverse) in $\AMod$.
\end{enumerate}
\end{prop}
\begin{proof}
In order to show that this map is well-defined we have to show that this map is independent of two choices: the choice of differential operator corresponding to a given symbol and the choice of the lift of said operator to $J^n_d E$.
From the $n$-jet sequence we obtain the following commutative diagram (without the dashed arrows).
\begin{equation}\label{diag:hatchet}
\begin{tikzcd}
0\ar[r]& \ZL^{n-1}(E,-)\ar[r,hook]& \AHom(J^{n-1}_dE,-)\ar[r,two heads,"-\circ j^{n-1}_{d,E}"]\ar[d,"-\circ\pi^{n,n-1}_{d,E}"']&[20pt] \Diff^{n-1}_d(E,-)\ar[d,hook]\\
0\ar[r]& \ZL^{n}(E,-)\ar[r,hook]& \AHom(J^{n}_dE,-)\ar[r,two heads,"-\circ j^{n}_{d,E}"]\ar[d,"-\circ\iota^{n}_{d,E}"']& \Diff^n_d(E,-)\ar[d,two heads]\ar[dl,dashed,"\widehat{r}^n_{d,E,-}"]\\
&& \AHom(S^{n}_d(E),-)& \Symb^n_d(E,-)\ar[l,dashed,"r^n_{d,E,-}"]
\end{tikzcd}
\end{equation}
We first show the existence of $\widehat{r}^n_{d,E,-}$ commuting in \eqref{diag:hatchet}, and then we prove it factors as $r^n_{d,E,-}$.
Given a differential operator $\Delta$, we show that $\widetilde{\Delta}\circ \iota^n_{d,E}$ is independent of the choice of the lift $\widetilde{\Delta}\colon J^n_d E\to F$.
Any other lift is of the form $\widetilde{\Delta}+\alpha$, where $\alpha\in \ZL^n_d(E)$, so we just need to prove that $\alpha\circ \iota^n_{d,E}=0$.
By hypothesis, for all $\sigma\in S^n_d E$, there exist some $a_i\in A$ and $e_i\in E$ such that $\iota^n_{d,E}(\sigma)=\sum_i a_i j^n_{d,E}(e_i)$.
Thus,
\begin{equation}
\alpha\circ \iota^n_{d,E}
=\alpha\left(\sum_i a_i j^n_{d,E}(e_i)\right)
=\sum_i a_i \alpha\circ j^n_{d,E}(e_i)
=0.
\end{equation}
It follows that $\widehat{r}^n_{d,E,-}$ is well-defined.
In order to show that $r^n_{d,E,-}$ is well-defined, we show that $\widehat{r}^n_{d,E,-}$ vanishes on differential operators of order $n-1$.
Let $\eta\in \Diff^{n-1}_d(E,F)$, and let $\widetilde{\eta}\colon J^{n-1}_d E\to F$ be one of its lifts.
Then $\widetilde{\eta}\circ \pi^{n,n-1}_d$ is a lift of $\eta$ to $J^n_d E$.
Hence,
\begin{equation}
\widehat{r}^n_{d,E,F}(\eta)
=\widetilde{\eta}\circ\pi^{n,n-1}_{d,E}\circ\iota^n_{d,E}
=0.
\end{equation}
It follows that $r^n_{d,E,-}$ is well-defined and commutes in \eqref{diag:hatchet}.

For \eqref{prop:symbols_representation_well-defined:1} we take the top two rows of \eqref{diag:hatchet}, complete them by the kernel universal property.
By the snake lemma, the top three vertical morphisms are injective, and we have an epimorphism $\rho$ from the cokernel of $-\circ\pi^{n,n-1}_{d,E}$ to $\Symb^n_d(E,-)$, fitting in the following diagram.
\begin{equation}\label{diag:waraxe}
\begin{tikzcd}
0\ar[r]& \ZL^{n-1}(E,-)\ar[d]\ar[r,hook]& \AHom(J^{n-1}_dE,-)\ar[r,two heads,"-\circ j^{n-1}_{d,E}"]\ar[d,"-\circ\pi^{n,n-1}_{d,E}"']&[20pt] \Diff^{n-1}_d(E,-)\ar[d,hook]\\
0\ar[r]& \ZL^{n}(E,-)\ar[r,hook]& \AHom(J^{n}_dE,-)\ar[r,two heads,"-\circ j^{n}_{d,E}"]\ar[dd,"-\circ\iota^{n}_{d,E}"',bend right=80pt]\ar[d,two heads]& \Diff^n_d(E,-)\ar[d,two heads]\\
&& \coker(-\circ \pi^{n,n-1}_{d,E})\ar[r,two heads,"\rho"]\ar[d]& \Symb^n_d(E,-)\ar[dl,"r^n_{d,E,-}"]\\
&& \AHom(S^{n}_d(E),-)
\end{tikzcd}
\end{equation}
The $n$-jet sequence is right exact, so the vertical sequence in \eqref{diag:hatchet} is left exact, and thus the kernel map of $-\circ\iota^n_{d,E}$ is the inclusion $-\circ\pi^{b,b-1}_{d,E}$.
Thus, $\coker(-\circ\pi^{n,n-1}_{d,E})=\im(-\circ\iota^n_{d,E})$.
Since the bottom right squares in \eqref{diag:hatchet} and \eqref{diag:waraxe} commute, and since $-\circ\iota^n_{d,E}$ surjects onto its image, it follows that the bottom right triangle of \eqref{diag:waraxe} also commutes.
Thus, the composition $r^n_{d,E,-}\circ\rho$ is the inclusion of the image of $-\circ\iota^n_{d,E}$.
It follows that $\rho$ is an inclusion, and hence an isomorphism, making $r^n_{d,E,-}$ a monomorphism.

For \eqref{prop:symbols_representation_well-defined:2}, since $\iota^n_{d,E}$ is a section, the map $-\circ\iota^n_{d,E}\colon \AHom(J^n_d E,-)\to \AHom(S^n_d (E),-)$ is surjective.
Thus $\AHom(S^n_d (E),-)$ coincides with the image of $\iota^n_{d,E}$, and hence $r^n_{d,E,-}$ and $\rho$ in \eqref{diag:waraxe} are one inverse to the other.
Now, if $r^n_{d,E,-}$ is surjective, then so is $-\circ \iota^n_{d,E}\colon \AHom(J^n_d E,-)\to \AHom(S^n_d (E),-)$ by \eqref{diag:hatchet}, as it is composition of epis.
In particular, the preimage of $\id_{S^n_d(E)}$ under $-\circ \iota^n_{d,E}$ will be a right inverse for $\iota^n_{d,E}$.
\end{proof}
Whenever \eqref{eq:symforms_gen_by_proj} holds, one might prefer to generalize the classical notion of symbol in a different direction.
\begin{defi}\label{def:restrictionsymbol}
The \emph{$n$-th order restriction symbol} of a linear differential operator $\Delta:E\to F$ of order at most $n$ is the following precomposition
\begin{equation}
\widetilde{\Delta}\circ \iota^n_d\colon S^n_d(E)\longrightarrow F
\end{equation}
\end{defi}
This definition is independent of the lift $\widetilde{\Delta}$ because of \eqref{eq:symforms_gen_by_proj}.
Without assuming \eqref{eq:symforms_gen_by_proj}, we can always define the notion of $n$-th order restriction symbol for elemental (resp.\ primitive) jets by precomposing the (unique)-lift of the $n$-th order elemental (resp.\ primitive) linear differential operator of order at most $n$, by $\Siota^n_d$ (resp.\ $\Ciota^n_d$).
\begin{rmk}\label{rem:identifyrestrictionsymbol}
When the map $r^n_{d,E,F}$ defined in \eqref{eq:rmap} is a monomorphism for all $F$, the functor $\Symb^n_d(E,-)$ can be embedded into the representable functor $\AHom(S^n_d(E),-)$.
When this holds, we can identify symbols of a differential operator $\Delta\colon E\to F$ with the restriction symbol, that is $\symb^n_d(\Delta)=\widetilde{\Delta}\circ \iota^n_{d,E}$.
Furthermore, if $r^n_{d,E,-}$ is an isomorphism, then the functor $\Symb^n_d(E,-)$ is representable with representing object $S^n_d(E)$.

In particular, representability of the symbol functor holds in the classical case, i.e.\ \cite[\jetssscman]{FMW}.
\end{rmk}
\begin{rmk}\label{rmk:representability_maps_are_endobilinear}
Naturality yields that for all $E,F\in\AMod$, the maps
\begin{align}
-\circ j^n_{d,E}\colon \AHom(J^n_d E,F)\longrightarrow \Diff^n_d(E,F),
&\hfill&
r^n_{d,E,F}\colon\Symb^n_d(E,F)\longrightarrow \AHom(S^n_d(E),F).
\end{align}
are $(\AHom(F,F),\AHom(E,E))$-bilinear, where the multiplicative actions on $\Diff^n_d(E,F)$ and $\Symb^n_d(E,F)$ are given by \cite[\jetscoralgebrasofDOs]{FMW} and Corollary \ref{cor:symb_graded_algebra}, respectively, and the action on modules of the form $\AHom(G(E),F)$ for a functor $F$ is given by
\begin{align}
\psi f \phi\colonequals \psi\circ f\circ G(\phi),
&\hfill&
\textrm{for }
f\in\AHom(G(E),F),\ 
\psi\in\AHom(F,F),\ 
\phi\in\AHom(E,E).
\end{align}
\end{rmk}
Before proceeding, we give another result that will expand \cite[\jetsproptensorcomparison]{FMW}.
\begin{lemma}\label{lemma:2jet_right_exact}
Consider the natural transformation $\gamma^n_d\colon J^n_d A\otimes_A \to J^n_d$, cf.\ \cite[\jetsproptensorcomparison]{FMW}.
\begin{enumerate}
\item\label{lemma:2jet_right_exact:1} For $n=0,1$, the map $\gamma^n_d$ is an isomorphism;
\item\label{lemma:2jet_right_exact:2} For $n=2$, the map $\gamma^2_d$ is always an epimorphism;
\item\label{lemma:2jet_right_exact:3} The map $\gamma^n_d$ is always compatible with the map $\tau^n_d$, cf.\ \cite[\jetspropStensorcomparison]{FMW}, i.e.\ 
\begin{equation}\label{diag:compatibility_gamma_tau}
\begin{tikzcd}
S^n_d\otimes_A -\ar[r,"\iota^n_d\otimes \id"]\ar[d,"\tau^n_d"]&J^n_d A\otimes_A -\ar[d,"\gamma^n_d"]\\
S^n_d\ar[r,"\iota^n_d"]&J^n_d
\end{tikzcd}
\end{equation}
\item\label{lemma:2jet_right_exact:4} The $2$-jet sequence is always right exact.
\begin{equation}\label{diag:2jet_res}
\begin{tikzcd}
S^2_d\ar[r,"\iota^n_d"]&J^2_d \ar[r,"\pi^{2,1}_d",two heads]&J^1_d\ar[r]&0
\end{tikzcd}
\end{equation}
\end{enumerate}
\end{lemma}
\begin{proof}\
\begin{enumerate}
\item True by definition, cf.\ \cite[\jetssonejetfunctor]{FMW}.
\item Let us consider the diagram \cite[\jetsdiagtensorcomparison]{FMW} for $n=2$.
We have
\begin{equation}\label{diag:tensor_comparison}
\begin{tikzcd}[column sep=45pt]
J^2_d A\otimes_A -\ar[r,"l^2_{d,A}\otimes_A \id"]\ar[d,dashed,"\gamma^2_d"']&J^{(2)}_d A\otimes_A -\ar[d,"J^1_d(\gamma^1_d)"',"\wr"]\ar[r,twoheadrightarrow,"\widetilde{\DH}_d\otimes_A \id"]&\Omega^1_d\ltimes \Omega^2_d A\otimes_A -\ar[d,"(\Omega^1_d\ltimes \Omega^2_d)(\gamma^0_d)","\wr"']\\
J^2_d \ar[r,hookrightarrow,"l^2_d"']&J^{(2)}_d \ar[r,twoheadrightarrow,"\widetilde{\DH}_d"']&\Omega^1_d\ltimes \Omega^2_d
\end{tikzcd}
\end{equation}

The central and right vertical arrows are isomorphisms by \eqref{lemma:2jet_right_exact:1} and functoriality.
This implies that the right square in \eqref{diag:tensor_comparison} is an isomorpgism of maps, and thus there is an isomorphism induced by the kernel functor.
By definition of $J^n_d$, the bottom sequence in \eqref{diag:tensor_comparison} is a short exact sequence.
In particular $J^2_d$ is the kernel of $\widetilde{\DH}_d$, and thus, the epi-mono factorization of $\gamma^2_d$ determines the epi-mono factorization of $l^2_{d,A}\otimes_A \id$.
In particular, the image of $\gamma^2_d$ can be identified with that of $l^2_{d,A}\otimes_A \id$ up to isomorphisms.
Since the left tensoring with a module is a right exact functor, the top sequence is right exact, and in particular the image of $l^2_{d,A}\otimes_A \id$ is the kernel of $\widetilde{\DH}_d$ up to isomorphism.
It follows that $\gamma^2_d$ is surjective.
\item We prove this point by induction.
For $n=0,1$, the statement is true by definition of jet functor and functor of symmetric forms.
For $n\ge 2$, consider the following diagram
\begin{equation}\label{diag:compatibility_gamma_tau_induction}
\begin{tikzcd}
S^n_d\otimes_A -\ar[ddd,"\tau^n_d"']\ar[rrrr,"\iota^n_d\otimes_A \id"]\ar[rd,"\iota^n_{\wedge,A}\otimes_A \id"]&[-20pt]&[30pt]&[30pt]&[-20pt]J^n_d A\otimes_A - \ar[ld,"l^n_{d,A}\otimes_A \id"']\ar[ddd,"\gamma^n_d"]\\
&\Omega^1_d\circ S^{n-1}_d\otimes_A -\ar[d,"\Omega^1_d(\tau^{n-1}_d)"']\ar[r,"\Omega^1_d(\iota^{n-1}_{d,A})\otimes_A \id"] &\Omega^1_d\circ J^{n-1}_d A\otimes_A - \ar[d,"\Omega^1_d(\gamma^{n-1}_d)"]\ar[r,"\iota^1_{d,J^{n-1}_d A}\otimes_A \id"] &J^1_d\circ J^{n}_d A\otimes_A -\ar[d,"J^1_d(\gamma^{n-1}_d)"]\\
&\Omega^1_d\circ S^{n-1}_d\ar[r,"\Omega^1_d(\iota^{n-1}_d)"']&\Omega^1_d\circ J^{n-1}_d\ar[r,"\iota^1_{d,J^{n-1}_d }"']&J^1_d\circ J^{n-1}_d\\
S^n_d\ar[rrrr,"\iota^n_d"']\ar[ru,hookrightarrow,"\iota^n_{\wedge}"']&&&& J^n_d\ar[lu,hookrightarrow,"l^n_d"]
\end{tikzcd}
\end{equation}
The left central square in \eqref{diag:compatibility_gamma_tau_induction} is commutative by inductive hypothesis and by the functoriality of $\Omega^1_d$.
The central right square commutes by the naturality of $\iota^1_d$ with respect to the map $\gamma^{n-1}_d$, cf.\ \cite[\jetsproponejses]{FMW}.
Both the top and bottom squares also commute by definition of $\iota^n_d$, cf.\ \cite[\jetslemmauniqiota]{FMW}.
Finally, also the outer left and right squares of \eqref{diag:compatibility_gamma_tau_induction} commute, cf.\ \cite[\jetspropStensorcomparison]{FMW} and \cite[\jetsproptensorcomparison]{FMW}, respectively.
It follows that the two sides of the outer square in \eqref{diag:compatibility_gamma_tau_induction} coincide when composed with $l^n_d$, i.e.\ 
\begin{equation}
l^n_d\circ\gamma^n_d\circ(\iota^n_d\otimes_A\id)
=l^n_d\circ\iota^n_d\circ\tau^n_d.
\end{equation}
The commutativity of \eqref{diag:compatibility_gamma_tau} follows from the fact that $l^n_d$ is, by definition, a monomorphism.
\item We already know that the map $\pi^{2,1}_d$ is an epi, because $J^1_d=\SJ^1_d$, and thus for all $E$ in $\AModB$, $J^1_d E$ is generated by prolongations $j^1_{d,E}(e)$, which can be obtained as $\pi^{2,1}_{d,E}(j^2_{d,E}(e))$, for $j^2_{d,E}(e)\in J^2_d E$.

We are left to prove that the sequence is exact in $J^2$, and thus that $\ker(\pi^{2,1}_d)=\Im(\iota^n_d)$.
Let thus $E$ be in $\AModB$, and let $\xi\in J^2_d E$ such that $\pi^{2,1}_{d,E}(\xi)=0$.
From point \eqref{lemma:2jet_right_exact:3} and \cite[\jetsrmkgammacomppi]{FMW}, we have the following commutative diagram.
\begin{equation}\label{diag:comparison_2jet_sequence_with_tensors}
\begin{tikzcd}
S^2_d\otimes_A E\ar[d,two heads,"\gamma^2_{d,E}"]\ar[r,"\iota^2_{d,A}\otimes_A \id_E"]&[40pt]J^2_d A\otimes_A E \ar[d,two heads,"\tau^2_{d,E}"]\ar[r,"\pi^{2,1}_{d,A}\otimes_A \id_E",two heads]&[40pt]J^1_dA\otimes E\ar[d,"\wr"',"\tau^1_{d,E}"]\ar[r]&0\\
S^2_d(E)\ar[r,"\iota^2_{d,E}"]&J^2_d E \ar[r,"\pi^{2,1}_{d,E}"]&J^1_d E\ar[r]&0
\end{tikzcd}
\end{equation}
By point \eqref{lemma:2jet_right_exact:2}, we can always find an element $\xi'\in J^2_d A\otimes_A E$ such that $\gamma^n_d(\xi')=\xi$.
When evaluated on $A$, the $2$-jet sequence is exact, since $\Tor^A_1(\Omega^1_d,A)=0$, cf.\ \cite[\jetspropdefeth]{FMW} and \cite[\jetsproptorsiontwojetses]{FMW}.
Thus, if we tensor this sequence on the right by $E$, we get a right exact sequence, namely, the top sequence in \eqref{diag:comparison_2jet_sequence_with_tensors}.
By commutativity of \eqref{diag:comparison_2jet_sequence_with_tensors}, we have that
\begin{equation}
\tau^1_{d,E}\circ (\pi^{2,1}_{d,A}\otimes_A \id_E)(\xi')
=\pi^{2,1}_{d,E}(\tau^2_{d,E}(\xi'))
=\pi^{2,1}_{d,E}(\xi)
=0.
\end{equation}
Since $\tau^1_{d,E}$ is an iso by \eqref{lemma:2jet_right_exact:1}, we have that $\xi'\in\ker(\pi^{2,1}_{d,A}\otimes_A \id_E)=\Im(\iota^2_{d,A}\otimes_A \id_E)$.
Thus, there exists $\sigma\in S^2_d\otimes_A E$ such that $(\iota^2_{d,A}\otimes_A \id_E)(\sigma)=\xi'$.
We now have
\begin{equation}
\iota^2_{d,E}(\gamma^2_{d,E}(\sigma))
=\tau^2_{d,E}\circ(\iota^2_{d,A}\otimes_A \id_E)(\sigma)
=\tau^2_{d,E}(\xi')
=\xi.
\end{equation}
Thus, $\xi\in\Im(\iota^2_{d,E})$, proving right exactness of \eqref{diag:2jet_res}.
\qedhere
\end{enumerate}
\end{proof}
In low grades we now obtain more regularity via the following.
\begin{prop}\label{prop:symbol_representability_low_dimension}\
\begin{enumerate}
\item\label{prop:symbol_representability_low_dimension:1} For $n=0,1$, the map $r^n_d$, cf.\ \eqref{eq:rmap}, is well-defined and an isomorphism;
\item\label{prop:symbol_representability_low_dimension:2} For $n=2$, if $\Omega^2_d=\Omega^2_{d,\text{max}}$, the map $r^2_{d,E,-}$ is well defined and a mono.
Furthermore, if we also have $\Omega^1_d$ and $E$ projective in $\AMod$, then $r^2_{d,E,-}$ is an isomorphism.
\end{enumerate}
\end{prop}
\begin{proof}\
\begin{enumerate}
\item For $n=0,1$, we have that $\Im(\iota^n_{d,E})\subseteq \SJ^n_d (E)$, cf.\ Proposition \ref{prop:elemental_jet_properties}.\eqref{prop:elemental_jet_properties:1}.
In particular, for $n=0$, $\Symb^0(E,F)=\AHom(E,F)$, and $\iota^0_d=\id$, so the map $r^0_{d,E,F}$ is essentially the identity map.
For $n=1$, we know that the $1$-jet sequence is exact and $\iota^1_d$ has $\widetilde{d}$ as a left inverse, so by Proposition \ref{prop:symbols_representation_well-defined}, the map $r^1_d$ is an isomorphism.
\item By Proposition \ref{prop:rep_for_2jet}.\eqref{prop:rep_for_2jet:1}, the map $r^2_d$ is well defined.
By Lemma \ref{lemma:2jet_right_exact}.\eqref{lemma:2jet_right_exact:4} and Proposition \ref{prop:symbols_representation_well-defined}.\eqref{prop:symbols_representation_well-defined:1}, we obtain that $r^2_d$ is a mono.

If now $E$ is projective in $\AMod$, in particular $\Tor^A_1(\Omega^1_d,E)=0$, so the $2$-jet sequence at $E$ is exact, cf.\ \cite[\jetspropdefeth]{FMW} and \cite[\jetsproptorsiontwojetses]{FMW}.
Moreover, since $\Omega^1_d$ and $E$ are projective in $\AMod$, so is $J^1_d E$, cf.\ \cite[\jetspropOmJstab]{FMW}, so the $2$-jet sequence is actually split exact.
Proposition \ref{prop:symbols_representation_well-defined}.\eqref{prop:symbols_representation_well-defined:2} then yields the isomorphism.\qedhere
\end{enumerate}
\end{proof}

As a consequence, we obtain the following characterization of left connections.
\begin{prop}\label{prop:symbolsofconnections}
	Let $E$ be a left $A$-module.
	A linear differential operator $\nabla\colon E \rightarrow \Omega^1_d(E)$ of order at most $1$ is a left connection if and only if its first order symbol is the identity, i.e.,
	\begin{equation}
		\symb^1_d(\nabla) = \id_{\Omega^1_d(E)}.
	\end{equation}
\end{prop}
\begin{proof}
	By Proposition \ref{prop:symbol_representability_low_dimension}.\eqref{prop:symbol_representability_low_dimension:1}, we can identify symbols with their image under the map $r^1_d$.
	
	Suppose that $\nabla$ is a left connection.
	Then, by the proof of \cite[\jetspropconnexionsplits]{FMW}, we have $\widetilde{\nabla} \circ \iota^1_{d,E} = \id_{\Omega^1_d(E)}$, which establishes the forward implication.
	Conversely, suppose $\nabla$ has symbol $\id_{\Omega^1_d(E)}$.
	Then $\widetilde{\nabla} \circ \iota^1_{d,E}=\id_{\Omega^1_d(E)}$, and so it is a left-splitting operator of the $1$-jet exact sequence in $\AMod$.
	Hence $\nabla$ is a left connection by \cite[\jetspropconnexionsplits]{FMW}.
\end{proof}

Finally, we show that one can directly compute the restriction to $S^2_{d,\textrm{max}}$ of the lift of a differential operator of degree $2$ via the following result.
\begin{lemma}\label{lemma:explicitrestrictedsymbol}
Let $\Delta\in\Diff^2_d(E,F)$, then for all lifts $\widetilde{\Delta}\colon J^2_d E\to F$, the restriction $\widetilde{\Delta}\circ \iota^2_d$ evaluated at $\sum_{i,j} da_{i,j}\otimes db_{i,j}\otimes_A e_j\in S^2_{d,\textrm{min}}(E)\subseteq S^2_d(E)$ is
\begin{equation}
\widetilde{\Delta}\circ \iota^2_d\left(\sum_{i,j} da_{i,j}\otimes_A db_{i,j}\otimes_A e_j\right)
=-\sum_{i,j} a_{i,j}\Delta(b_{i,j} e_j)
\end{equation}
for all $\sum_{i} a_{i,j}\otimes b_{i,j}\in N_d$ and $e_j\in E$.
\end{lemma}
\begin{proof}
We first show that the following diagram commutes
\begin{equation}\label{diag:restrictionsymbol}
\begin{tikzcd}
N_d\otimes_A E\ar[r,two heads]\ar[d,"-d\otimes_A d\otimes_A E"']&N_d(E)\ar[r,hook]&J^2_u E\ar[d,"\widehat{p}^n_{d,E}"'] \ar[dr,"\widetilde{\Delta}_u"]\\
S^2_d(E)\ar[rr,"\iota^2_{d,E}"]&&J^2_d E \ar[r,"\widetilde{\Delta}"]&[40pt]F
\end{tikzcd}
\end{equation}
Here, the map
\begin{align}
\widetilde{\Delta}_u\colon J^2_u E=A\otimes E\longrightarrow F,
&\hfill&
a\otimes e\longmapsto a\Delta (e)
\end{align}
is the (unique) lift of $\Delta$ to $J^2_u E$, and
\begin{align}
d\otimes_A d\otimes_A E\colon N_d(E)\longrightarrow S^2_d(E),
&\hfill&
\sum_{i,j} a_{i,j}\otimes b_{i,j}\otimes_A e_j\longmapsto \sum_{i,j} da_{i,j}\otimes_A db_{i,j} \otimes_A e_j
\end{align}
is well defined and has image $S^2_{d,\textrm{min}}(E)$, cf.\ \cite[\jetseqStwomindd]{FMW}.

The left square of \eqref{diag:restrictionsymbol} commutes because for all elements $\sum_{i,j} a_{i,j}\otimes b_{i,j} e_j\in N^2_d(E)$ we have
\begin{equation}
\begin{split}
&\iota^2_{d,E}\left((-d\otimes_A d\otimes_A E)\left(\sum_{i,j} a_{i,j}\otimes b_{i,j}\otimes_A e_j\right)\right)
=-\iota^2_{d,E}\left(\sum_{i,j} da_{i,j}\otimes_A db_{i,j}\otimes_A e_j\right)\\
&\qquad=-\sum_{i,j}[1\otimes a_{i,j}-a_{i,j}\otimes 1]\otimes_A [1\otimes b_{i,j}-b_{i,j}\otimes 1]\otimes_A e_j\\
&\qquad=-\sum_{i,j}\left([1\otimes 1]\otimes_A [a_{i,j}\otimes b_{i,j}]-[1\otimes 1]\otimes_A [a_{i,j}b_{i,j}\otimes 1]\right.\\
&\qquad\quad\left.-[a_{i,j}\otimes 1]\otimes_A [1\otimes b_{i,j}]+[a_{i,j}\otimes b_{i,j}]\otimes_A [1\otimes 1]\right)\otimes_A e_j\\
&\qquad=\sum_{i,j} [a_{i,j}\otimes 1]\otimes_A [1\otimes b_{i,j}]\otimes_A e_j\\
&\qquad=\sum_{i,j} a_{i,j}j^2_{d,E}( b_{i,j} e_j)\\
&\qquad=\widehat{p}^2_d\left(\sum_{i,j} a_{i,j}\otimes b_{i,j} e_j\right)
\end{split}
\end{equation}
It follows that on the elements of $S^2_{d,\textrm{min}}(E)$, which are therefore in the image of $-d\otimes_A d\otimes_A E$, we can compute
\begin{equation}
\begin{split}
\widetilde{\Delta}\circ \iota^2_d\left(\sum_{i,j} da_{i,j}\otimes_A db_{i,j}\otimes_A e_j\right)
&=-\widetilde{\Delta}\circ \iota^2_d\circ(-d\otimes_A d \otimes_A E)\left(\sum_{i,j} a_{i,j}\otimes b_{i,j}\otimes_A e_j\right)\\
&=-\widetilde{\Delta}_u\left(\sum_{i,j} a_{i,j}\otimes b_{i,j} e_j\right)\\
&=-\sum_{i,j} a_{i,j}\Delta(b_{i,j} e_j).
\end{split}
\end{equation}
\end{proof}
\begin{rmk}
If $\Omega^2_d=\Omega^2_{d,\textit{max}}$, then Lemma \ref{lemma:explicitrestrictedsymbol} allows to compute the symbol of any differential operator of order at most $2$ the whole $S^2_d(E)$ via Proposition \ref{prop:symbol_representability_low_dimension}.\eqref{prop:symbol_representability_low_dimension:2}.
\end{rmk}

\subsection{Symbols of elemental differential operators}
As mentioned in Definition \ref{def:symbolquotientdef}, we can give define a notion of symbols also for elemental differential operators.
Thus, we write $\WS^n_d(E,F)\colonequals \WD^n_d(E,F)/\WD^{n-1}_d(E,F)$, for all $n\ge 0$, with the convention $\WD^{-1}_d(E,F)=0$.
We maintain the notation $\symb^n_d\colon \WD^n_d(E,F)\to \WS^n_d(E,F)$ for the quotient projection.

We can compare the functor of symbols with that of symbols of elemental differential operators via the following proposition.
\begin{prop}
There exists a canonical binatural transformation
\begin{align}\label{eq:comparing_symbols}
\Symb^n_d(E,F) \longrightarrow \WS^n_d(E,F),
&\hfill&
\symb^n_d(\Delta)\longmapsto \symb^n_d(\Delta).
\end{align}
This map is injective for the modules $E$ and $F$ if differential operators of order at most $n-1$ on $E$ are representable, and is surjective if differential operators of order at most $n$ on $E$ are representable.
\end{prop}
\begin{proof}
The inclusions $\Diff^n_d\subseteq \WD^n_d$ is compatible with the inclusion of the modules of differential operators into modules of differential operators of higher order.
By the cokernel universal property, we obtain \eqref{eq:comparing_symbols} as the dashed map in the following diagram.
\begin{equation}\label{diag:elemental_comparison}
\begin{tikzcd}
0\ar[r]&\Diff^{n-1}_d\ar[d,hook]\ar[r,hook]&	\Diff^n_d\ar[d,hook]\ar[r,two heads]&\Symb^n_d\ar[d,dashed]\ar[r]&0\\
0\ar[r,two heads]&\WD^{n-1}_d\ar[r,hook]&	\WD^n_d\ar[r]&\WS^n_d\ar[r]&0
\end{tikzcd}
\end{equation}

The last two statements follow directly from applying the snake lemma to \eqref{diag:elemental_comparison}, thanks to Corollary \ref{cor:representability_iff_WDO=DO}.
\end{proof}

Compared to the functors of symbols, the functors of symbols of elemental differential operators behave, in general, more regularly with respect to the functor $\SS^n_d$.
We thus obtain an analogue of Proposition \ref{prop:symbols_representation_well-defined} for elemental differential operators and elemental symmetric forms.
\begin{prop}\label{prop:elemental_symbols_representation}
There exists a monomorphism $\Sr^n_d$ defined as
\begin{align}\label{eq:Srmap}
\Sr^n_{d,E,F}\colon\WS^n_d(E,F)\longrightarrow \AHom(\SS^n_d(E),F),
&\hfill&
\symb^n_d(\Delta)\longmapsto \widetilde{\Delta}\circ \Siota^n_{d,E}.
\end{align}
natural in $E$ both $F$.
Moreover, the map $\Sr^n_{d,E,-}$ is a natural isomorphism if and only if $\Siota^n_{d,E}$ is a section (it admits a left inverse).
E.g.\ when $\SJ^n_d E$ is projective.
\end{prop}
\begin{proof}
By Proposition \ref{prop:elementalexactsequence}, we have the short exact sequence \eqref{diag:elemental_jet_ses}.
For all $F$, the functor $\AHom(-, F)$ is left exact, so the following is a left exact sequence.
\begin{equation}\label{diag:representable_elemental_jet_les}
\begin{tikzcd}
0\ar[r]&\AHom(\SJ^{n-1}_d,F)\ar[r,hook,"-\circ \Spi^{n,n-1}_d"]&	\AHom(\SJ^n_d,F)\ar[r,"-\circ \Siota^n_d"]&\AHom(\SS^n_d,F).
\end{tikzcd}
\end{equation}
All maps appearing in \eqref{diag:representable_elemental_jet_les} are natural in $F$.

By Corollary \ref{cor:WDO_always_representable}, the two leftmost functors are isomorphic to the corresponding functors of elemental differential operators.
We thus get the following diagram of bifunctors, which commutes by the compatibility between projections and prolongations in the elemental case.
\begin{equation}\label{diag:elemental_bifunctors}
\begin{tikzcd}
0\ar[r]&\WD^{n-1}_d\ar[d,"\wr","(-\circ\Sj^{n-1}_d)^{-1}"']\ar[r,hook]&[40pt]	\WD^n_d\ar[d,"\wr"',"(-\circ\Sj^{n-1}_d)^{-1}"]\ar[r,two heads]&[40pt]\WS^n_d\ar[d,dashed,"\Sr^n_d"]\ar[r]&0\\
0\ar[r]&\AHom(\SJ^{n-1}_d,-)\ar[r,hook,"-\circ \Spi^{n,n-1}_d"]&	\AHom(\SJ^n_d,-)\ar[r,"-\circ \Siota^n_d"]&\AHom(\SS^n_d,-)
\end{tikzcd}
\end{equation}
The map $\Sr^n_d$ in \eqref{diag:elemental_bifunctors} exists because $\WS^n_d$ is by definition the cokernel of the inclusion $\WD^{n-1}_d \subseteq \WD^n_d$, or equivalently it corresponds to the cokernel of the isomorphic inclusion $-\circ\Spi^{n,n-1}_d$, and the bottom row of \eqref{diag:elemental_bifunctors} is a complex.
If we apply the snake lemma to \eqref{diag:elemental_bifunctors}, we obtain that $\Sr^n_d$ has to be a monomorphism.
The vertical isomorphisms in \eqref{diag:elemental_bifunctors} correspond to mapping a differential operator $\Delta$ to the (unique) lift $\widetilde{\Delta}$, and thus the monomorphism $\Sr^n_d$ corresponds to the inclusion of the image of the map described in \eqref{eq:Srmap}.

The proof of the last statement of the proposition is proven exactly as Proposition \ref{prop:symbols_representation_well-defined}.\eqref{prop:symbols_representation_well-defined:2}.
\end{proof}
\begin{rmk}
The analogue of Proposition \ref{prop:elemental_symbols_representation} for primitive differential operators and symbols holds \textit{mutatis mutandis}.
\end{rmk}
\begin{rmk}
The analogue of Lemma \ref{lemma:explicitrestrictedsymbol} also holds for elemental and primitive differential operators \textit{mutatis mutandis}.
\end{rmk}

\subsection{Symbol category and symbol algebras}\label{ss:symbcat}
In \cite[\jetscordiffcat]{FMW} it is shown that there is a category $\Diff_d$ whose objects are those of $\AMod$ and whose maps are differential operators of finite order.
Moreover, this category is enriched over $\Mod$, cf.\ \cite[\jetsremDOcatareenrichedinMod]{FMW}, and more specifically over ascending filtered $\bk$-modules.
\begin{prop}\label{prop:symbol_composition}
The composition of finite order differential operators induces a well-defined composition on the corresponding symbols
\begin{align}\label{eq:symbol_composition}
\circ\colon \Symb^n_d(F,G)\times \Symb^m_d(E,F)\longrightarrow \Symb^{n+m}_d(E,G),
&\hfill&
\symb^n_d(\Delta)\circ \symb^m_d(\Delta')\colonequals \symb^{n+m}_d(\Delta\circ\Delta'),
\end{align}
for $E$, $F$, $G$ in $\AMod$.
This induces a graded $\bk$-module morphism
\begin{equation}
\circ\colon \Symb^{\bullet}_d(F,G)\otimes \Symb^{\bullet}_d(E,F)\longrightarrow \Symb^{\bullet}_d(E,G),
\end{equation}
where $\Symb^{\bullet}(E,F)_d\colonequals\bigoplus_{n=0}^\infty \Symb^n_d(E,F)$.
\end{prop}
\begin{proof}
We start by proving that the composition \eqref{eq:symbol_composition} is well-defined.
Consider two linear differential operators equivalent to $\Delta$ and $\Delta'$ respectively.
By definition, they must be of the form $\Delta+\alpha$ and $\Delta'+\alpha'$, respectively, where $\alpha\in\Diff^{n-1}_d(F,G)$ and $\alpha'\in\Diff^{m-1}_d(E,F)$.
We get
\begin{equation}
\symb^n_d(\Delta+\alpha)\circ \symb^m_d(\Delta'+\alpha')
=\symb^{n+m}_d\left((\Delta+\alpha)\circ(\Delta'+\alpha')\right)
=\symb^{n+m}_d\left(\Delta\circ\Delta'+\Delta\circ\alpha'+\alpha\circ\Delta'+\alpha\circ\alpha'\right).
\end{equation}
Since the composition of differential operators is compatible with the filtration, we have that
\begin{align}
\Delta\circ\alpha'+\alpha\circ\Delta'\in\Diff^{n+m-1}_d(E,G)
&\hfill&
\text{and}
&\hfill&
\alpha\circ\alpha'\in\Diff^{n+m-2}_d(E,G)\subseteq\Diff^{n+m-1}_d(E,G).
\end{align}
It follows that
\begin{equation}
\symb^{n+m}_d\left(\Delta\circ\Delta'+\Delta\circ\alpha'+\alpha\circ\Delta'+\alpha\circ\alpha'\right)
=\symb^{n+m}_d\left(\Delta\circ\Delta'\right),
\end{equation}
and thus the composition is well-defined.

This gives us an operation
\begin{equation}
\circ\colon \Symb^{\bullet}_d(F,G)\times \Symb^{\bullet}_d(E,F)\longrightarrow \Symb^{\bullet}_d(E,G),
\end{equation}
which is compatible with the grading.
This multiplication factors through the tensor product $\Symb^{\bullet}_d(F,G)\otimes \Symb^{\bullet}_d(E,F)$, with the induced grading, because the composition of differential operators is $\bk$-bilinear, and $\symb^n_d$ is $\bk$-linear.
\end{proof}
Notice that $\Symb^\bullet_d(E,F)$ is the graded module associated to the ascending filtered module $\Diff_d(E,F)$ for all $E$, $F$ in $\AMod$.

Proposition \ref{prop:symbol_composition} allows us to define the following category.
\begin{defi}
Let the \emph{symbol category}, denoted $\Symb^{\bullet}_d$, be the enriched category over graded $\bk$-modules having as objects left $A$-modules, and as morphisms between $E$ and $F$, the elements of the graded module $\Symb^\bullet_d(E,F)$, with the natural composition induced by Proposition \ref{prop:symbol_composition}.
\end{defi}
\begin{cor}\label{cor:symb_graded_algebra}
From the category $\Symb^{\bullet}_d$, the module $\Symb^\bullet_d(E,E)$ inherits a structure of graded $\bk$-algebra with respect to the composition, and $\Symb^\bullet_d(E,F)$ inherits a structure of $(\Symb^\bullet_d(F,F),\Symb^\bullet_d(E,E))$-bimodule, for all $E$, $F$ in $\AMod$ from the $(\Diff_d(F,F),\Diff_d(E,E))$-bimodule structure on $\Diff_d(E,F)$, cf.\ \cite[\jetscoralgebrasofDOs]{FMW}.

In particular, $\Symb^{\bullet}_d(E,F)$ will also inherit a structure of graded $(\AHom(F,F),\AHom(E,E))$-bimodule.
Hence, $\Symb^{\bullet}_d(A,A)$ has an induced structure of graded $(A^{\op},A^{\op})$-bimodule cf.\ \cite[\jetscorAopzeroorder]{FMW}, and therefore of $A$-bimodule by swapping the actions.
Since the composition map is compatible with the $A$-actions, this induces a structure of graded $A$-algebra on $\Symb^{\bullet}_d(A,A)$.
\end{cor}

\begin{defi}
\label{defi:symbol_algebra}
Given $E$ in $\AMod$, we term the $\bk$-algebra $\Symb^\bullet_d(E,E)$ the \emph{symbol algebra on $E$}.
\end{defi}

\begin{rmk}
Observe that what has been developed in this subsection (i.e. § \ref{ss:symbcat}) also holds \textit{mutatis mutandis} for elemental (and primitive) differential operators and their symbols if $\Omega^1_d$ is flat in $\ModA$.
\end{rmk}

\begin{rmk}
Classically, given a smooth manifold $M$ and a vector bundle $E$ over $M$, $\Symb^{\bullet}(A,A)$ is (a dense subset of) the algebra of smooth functions on the cotangent bundle $T^*M$ (cf.\ \cite[§5.4]{alekseevskij28geometry} and \cite[Proposition 10.12]{nestruev2020smooth}). 
Hence we can regard $\Symb^{\bullet}_d(A,A)$ (and its elemental and primitive analogues) as providing noncommutative notions of ‘‘total space algebra'' of $\Omega^1_d$. 
\end{rmk}

When we are given a left connection on $E$ in $\AMod$, this provides a natural way to construct a differential operator of order at most $1$ on $E$ with a given symbol.
\begin{prop}\label{prop:quantisation}\
\begin{enumerate}
\item\label{prop:quantisation:1} Let $E$ and $F$ be in $\AMod$, and let $\nabla\colon E\to \Omega^1_d(E)$ be a left connection on $E$.
Then there exists a right inverse to $\symb^1_d$ given by
\begin{equation}
q^1\colon \Symb^1_d(E,F)\longrightarrow \Diff^1_d(E,F),
\end{equation}
corresponding, at the level of the associated representable functor, to the precomposition
\begin{align}\label{eq:nabla_quantisation}
\AHom(\Omega^1_d(E),F)\longrightarrow \AHom(J^1_d E,F),
&\hfill&
\phi\longmapsto \phi\circ \widetilde{\nabla}.
\end{align}
This map is left $\AHom(F,F)$-linear.
\item\label{prop:quantisation:2} There is a canonical natural right inverse to $\symb^1_d$
\begin{equation}
q^1\colon \Symb^1_d(A,-)\longrightarrow \Diff^1_d(A,-)
\end{equation}
corresponding to the precomposition with $\widetilde{d}\colon J^1_d A\to \Omega^1_d$.
This map is right $A$-linear.
\end{enumerate}
\end{prop}
\begin{proof}\
\begin{enumerate}
\item By Remark \ref{rmk:representability_low_degree} and Proposition \ref{prop:symbol_representability_low_dimension}, using the natural isomorphism, it is equivalent to define the map $q^1$ on the representable functors, as shown in \eqref{eq:nabla_quantisation}.
As in the proof of Proposition \ref{prop:symbolsofconnections}, we invoke that $\widetilde{\nabla}\circ\iota^1_{d,E}=\id_{\Omega^1_d (E)}$ from \cite[\jetspropconnexionsplits]{FMW}, which proves that $\symb^1_d\circ q^1=\id_{\Diff^1_d(E,F)}$.

Left $\AHom(F,F)$-linearity follows from associativity, since $q^{1}$ is a precomposition.
\item It follows from \eqref{prop:quantisation:1} by taking the canonical left connection $d\colon A\to \Omega^1_d$.
Right $A$-linearity follows from $A^\op=\AHom(A,A)$-linearity.\qedhere
\end{enumerate}
\end{proof}

\section{Vector fields}
In this section we are presenting a generalization of the notion of vector field.
Classically, vector fields can be interpreted as those differential operators of order at most $1$ on the algebra of smooth functions that annihilate the subalgebra of locally constant functions.
Therefore, we give the following definition for the noncommutative setting.
\begin{defi}\label{def:vectorfields}
Given a first order differential calculus $d\colon A\to\Omega^1_d$, we denote the \emph{right $A$-module of vector fields on $A$} by
	\begin{align}
		\VF\colonequals \Ann(\ker(d))\cap \Diff^1_d(A,A),
	\end{align}
	where $\operatorname{Ann}$ is the annihilator.
	The right module structure is given by
	\begin{equation}\label{eq:right_action_VF}
		X\cdot a = R_a \circ X
	\end{equation}
	for all $a\in A$, where $R_a\colon A\to A$ is the right multiplication by $a$.
\end{defi}
\begin{prop}
	The right $A$-module structure on $\VF$ is well-defined.
\end{prop}
\begin{proof}
	Let $a\in A$ and $X \in \VF$.
	Since $R_a$ is a differential operator of order $0$, cf.\ \cite[\jetspropzeroorder]{FMW}, $X\cdot a$ is still a first order operator, cf.\ \cite[\jetspropdifferentialoperatorcomposition]{FMW}, and it annihilates $\ker(d)$, since $X$ does.
	Hence, the right action is well-defined.
\end{proof}
In the literature, there are already two notions of modules of vector fields, called right and left, cf.\ \cite[§2.7, p.~167]{BeggsMajid}, first proposed in \cite[Main theorem, p.~1201]{borowiec1996cartan}.
In fact, our notion coincides with that of left vector fields, namely $\AHom(\Omega^1_d,A)$, where the left and right $A$-actions are given by precomposition and composition with the right multiplication map, respectively, cf.\ \cite[§2.7, p.~168]{BeggsMajid}.
\begin{prop}\label{prop:fields_and_tangent_space}
	We have the following
	\begin{enumerate}
		\item\label{prop:fields_and_tangent_space:i} $\VF$ is isomorphic to $\AHom(\Omega^1_d,A)$ as a right $A$-module via the maps
	\begin{align}\label{eq:iso_vfs_BM_left_vfs}
		\VF\longrightarrow \AHom(\Omega^1_d,A),
		&\hfill&
		X \longmapsto \widetilde{X} \circ \iota^1_{d,A},
	\end{align}
	\begin{align}\label{eq:iso_vfs_BM_left_vfs_inverse}
		\AHom(\Omega^1_d, A)\longrightarrow \VF,
		&\hfill&
		\theta \longmapsto \theta \circ d.
	\end{align}
		\item\label{prop:fields_and_tangent_space:ii} Consider the $\bk$-linear operator on $\Diff^1_d(A,A)$ given by 
			\begin{equation}\label{eq:endomorphism_selecting_vf}
				\Delta \longmapsto \widetilde{\Delta}\circ \iota^1_{d,A} \circ d.
			\end{equation}
			The space $\VF$ is precisely the $1$-eigenspace of this map.
	\end{enumerate}
\end{prop}
\begin{proof}\
\begin{enumerate}
\item The map \eqref{eq:iso_vfs_BM_left_vfs} is well defined, as the lift to $J^1_d A$ is unique, cf.\ \cite[\jetspropuniquelift]{FMW}.
	The maps \eqref{eq:iso_vfs_BM_left_vfs} and \eqref{eq:iso_vfs_BM_left_vfs_inverse} are inverse to eachother.
	In one direction, for all $X\in\VF$ and $a\in A$, we use \cite[\jetseqonejetd]{FMW}, and \cite[\jetsequniquefirstorderlift]{FMW} to obtain
	\begin{equation}	
		\begin{split}
		\widetilde{X}\circ\iota^1_{d,A}\circ d(a)
		&=\widetilde{X}\circ\iota^1_{d,A}\circ p_{d,A}\circ d_u(a)\\
		&=\widetilde{X}\circ \widehat{p}_{d,A}( d_u(a))\\
		&=\widetilde{X}([1\otimes a-a\otimes 1])\\
		&=X(a)-aX(1)\\
		&=X(a),
	\end{split}
	\end{equation}
	where the last equality follows from the definition of $\VF$, as $1\in\ker(d)$.
	
	Vice versa, for all $\theta\in\AHom(\Omega^1_d,A)$, we have that $\theta\circ d=\theta\circ\widetilde{d}\circ j^1_{d,A}$, and hence $\widetilde{\theta\circ d}=\theta\circ\widetilde{d}$.
	If we now apply \eqref{eq:iso_vfs_BM_left_vfs}, we obtain $\theta\circ\widetilde{d}\circ \iota^1_{d,A}=\theta$, since $\widetilde{d}\circ \iota^1_{d,A}=\id_{J^1_d A}$, cf.\ \cite[\jetssssSplitting]{FMW}.
	
	Finally, \eqref{eq:iso_vfs_BM_left_vfs_inverse} is right $A$-linear, since $(\theta a)\circ d= (R_a \circ \theta)\circ d=R_a \circ (\theta\circ d)=(\theta\circ d)a$.

	\item Since \eqref{eq:endomorphism_selecting_vf} coincides with the composition of \eqref{eq:iso_vfs_BM_left_vfs} and \eqref{eq:iso_vfs_BM_left_vfs_inverse}, and hence the identity, when restricted to vector fields, we get that $\VF$ is contained in the $1$-eigenspace of this operator.
	On the other hand, if $\Delta\in\Diff^1_d(A,A)$ is in the $1$-eigenspace, we have $\Delta=\widetilde{\Delta}\circ \iota^1_{d,A}\circ d$, and thus $\Delta$ vanishes on $\ker(d)$, yielding $\Delta\in\VF$.\qedhere
\end{enumerate}
\end{proof}

\begin{rmk}
The map \eqref{eq:iso_vfs_BM_left_vfs} is just the restriction of the map $r^1_d$, cf.\ \eqref{eq:rmap}, to $\VF\subseteq \Diff^1_d(A,A)$.
Since \eqref{eq:iso_vfs_BM_left_vfs} is an isomorphism, we have $\AHom(\Omega^1_d,A)\cong \Symb^1(A,A)$.

The fact that this restriction is an isomorphism implies that the inverse map \eqref{eq:iso_vfs_BM_left_vfs_inverse} is a right inverse for $\symb^1_d$, and in particular it is the one induced by $d$ as in Proposition \ref{prop:quantisation}.\eqref{prop:quantisation:2}.
Thus, a vector field is a canonical representative for each first order symbol.

This induces a splitting of first order differential operators into differential operators of order $0$, i.e.\ the bimodule of right multiplications $R_A\cong A^\op$ by $A$, cf.\ \cite[\jetspropzeroorder]{FMW}, and vector fields
\begin{equation}
\Diff^1_d(A,A)=R_A\oplus_A \VF.
\end{equation}
\end{rmk}
\begin{prop}\label{prop:tangentbundle_is_bundle}
	The right $A$-module of vector fields $\VF$ can be equipped with a bimodule structure, by introducing the left action
	\begin{equation}\label{eq:VFleftaction}
		aX
		= \widetilde{X} \circ \iota^1_{d,A} \circ R_a \circ d,
	\end{equation}
	where $R_a$ is the right multiplication operator of $a\in A$ on $\Omega^1_d$.

With this action, the right $A$-linear isomorphism of Proposition \ref{prop:fields_and_tangent_space} becomes an $A$-bilinear isomorphism.
\end{prop}
\begin{proof}
	The formula \eqref{eq:VFleftaction} defines a vector field.
	The presence of $d$ ensures that $\ker(d)$ is annihilated, and since $R_a$, $\iota^1_{d,A}$, and $\widetilde{X}$ are differential operators of order $0$, the expression defines a differential operator of order at most $1$.
	
	The formula \eqref{eq:VFleftaction} defines an action, since we have
	\begin{equation}
	\begin{split}
		a(bX)
		&= \widetilde{bX} \circ \iota^1_{d,A} \circ R_a \circ d\\
		&= \reallywidetilde{(\widetilde{X} \circ \iota^1_{d,A} \circ R_b \circ d)} \circ \iota^1_{d,A} \circ R_a \circ d\\
		&= \reallywidetilde{(\widetilde{X} \circ \iota^1_{d,A} \circ R_b \circ \widetilde{d}\circ j^1_{d,A})} \circ \iota^1_{d,A} \circ R_a \circ d\\
		&= \widetilde{X} \circ \iota^1_{d,A} \circ R_b \circ \widetilde{d} \circ \iota^1_{d,A} \circ R_a \circ d\\
		&= \widetilde{X} \circ \iota^1_{d,A} \circ R_b \circ R_a \circ d\\
		&= \widetilde{X} \circ \iota^1_{d,A} \circ R_{ab} \circ d\\
		&= (ab)X.
	\end{split}
	\end{equation}
	Here we used that $\widetilde{d} \circ \iota^1_{d,A} = \id_{\Omega^1_d}$ and the fact that jet lifts are operators of order $0$.
	
	Furthermore, this left $A$-action is compatible with the right $A$-action \eqref{eq:right_action_VF}, as for all $a,b\in A$ and $X\in \VF$, we have
	\begin{equation}
	(aX)b
	=R_b\circ (\widetilde{X}\circ\iota^1_{d,A}\circ R_a\circ d)
	=R_b\circ \widetilde{X}\circ\iota^1_{d,A}\circ R_a\circ d
	=\widetilde{R_b\circ X}\circ\iota^1_{d,A}\circ R_a\circ d
	=a(Xb).
	\end{equation}

	For the bilinearity of the isomorphisms \eqref{eq:iso_vfs_BM_left_vfs} and \eqref{eq:iso_vfs_BM_left_vfs_inverse}, we only need to check that either map is also left $A$-linear, so we choose \eqref{eq:iso_vfs_BM_left_vfs_inverse}.
	For all $\theta\in\AHom(\Omega^1_d,A)$, we have
	\begin{equation}
		(a\theta)\circ d
		= \theta \circ R_a \circ d
		= \theta\circ \widetilde{d}\circ \iota^1_{d,A} \circ R_a \circ d
		= \reallywidetilde{\theta\circ d}\circ \iota^1_{d,A} \circ R_a \circ d
		= a(\theta \circ d),
	\end{equation}
	which yields the desired left $A$-linearity.
\end{proof}
\begin{cor}\label{cor:commutative_differentials_yield_commutative_vector_fields}
Let $A$ be a commutative $\bk$-algebra and let the calculus be a quotient of the $A$-bimodule of K\"ahler differentials (e.g.\ the de Rham calculus $\Omega^1_{dR}$), then the left action on $\VF$ by an element $a\in A$ coincides with the right action by the same element.
\end{cor}
\begin{proof}
	Commutativity yields that the left action by an element $a\in A$ coincides with the right action by the same element ($R_a$), both on $A$ and on $\Omega^1_{d}$.
	Thus, $R_a$ commutes with $\widetilde{X}\circ \iota^1_d$ for all $X\in\VF$ and $a\in A$.
	Hence,
	\begin{equation}
		aX
		= \widetilde{X} \circ \iota^1_{d,A} \circ R_a \circ d
		= R_a \circ \widetilde{X} \circ \iota^1_{d,A} \circ d
		=R_a \circ X
		=Xa.
	\end{equation}
\end{proof}
\begin{prop}[Generalized Leibniz]\label{prop:generalized_Leibniz}
With the left action \eqref{eq:VFleftaction}, we get the following equality for all $X\in\VF$ and $a,b\in A$.
\begin{equation}\label{eq:generalized_Leibniz}
X(ab)
= a(X(b)) + (bX)(a).
\end{equation}
\end{prop}
\begin{proof}
The statement follows from the following computation.
\begin{equation}
\begin{split}
(bX)(a)
&=\widetilde{X}\circ\iota^1_{d,A}\circ R_b\circ d(a)\\
&=\widetilde{X}\circ\iota^1_{d,A}( (da)b)\\
&=\widetilde{X}\circ\iota^1_{d,A}( d(ab)-adb)\\
&=\widetilde{X}\circ\iota^1_{d,A}\circ d(ab)-a(\widetilde{X}\circ\iota^1_{d,A}\circ d(b))\\
&=X(ab)-a(X(b)).
\end{split}
\end{equation}
\end{proof}

\begin{rmk}
Proposition \ref{prop:generalized_Leibniz} highlights the fact that vector fields $\VF$ in the sense of Definition \ref{def:vectorfields}, with the canonical inclusion $\VF\subset \End(A)$, form a (canonical) \emph{left Cartan pair}, in the sense of \cite[Definition 2.1, p.~1199]{borowiec1996cartan}.
Hence, they are a bimodule of \emph{left generalized vector fields} in the sense of \cite[Definition 2.1, p.~1199]{borowiec1996cartan}.
\end{rmk}
In the commutative case, \eqref{eq:generalized_Leibniz} reduces to the standard Leibniz rule, as we can see in the following result.
\begin{cor}\label{cor:commutative_generalised_Leibniz_is_Leibniz}
Let $A$ be a commutative $\bk$-algebra and let the calculus be a quotient of the $A$-bimodule of K\"ahler differentials (e.g.\ the de Rham calculus $\Omega^1_{dR}$), then for all $X\in\VF$ and $a,b\in A$ we have
\begin{equation}
X(ab)
= aX(b) + X(a)b.
\end{equation}
\end{cor}
\begin{proof}
	It follows directly from Proposition \ref{prop:generalized_Leibniz} and Corollary \ref{cor:commutative_differentials_yield_commutative_vector_fields}.
\end{proof}

\begin{cor}
The left action \eqref{eq:VFleftaction} on $\VF$ can be equivalently expressed as 
\begin{equation}
aX=X\circ R_a - R_{X(a)}.
\end{equation}
\end{cor}
\begin{proof}
The result follows immediately from Proposition \ref{prop:generalized_Leibniz}.
\end{proof}

\begin{rmk}
In general, $\ker(d)$ can be far larger than $\bk$.
For instance, let $\bk$ be a commutative ring with unit, and $S\subseteq \bk$ a multiplicative subset not containing any zero divisor.
Let $A=S^{-1}\bk$, i.e.\ the localization of $\bk$ by $S$, e.g.\ $\bk=\Z$ and $A=\Q$.
We have that the quotient map $A\otimes A\to A\otimes_{S^{-1}\bk} A$ has kernel generated by elements of the form
\begin{align}
x\frac{r}{s}	\otimes y-x\otimes \frac{r}{s}y,
&\hfill&
\textrm{for }x, y\in A, r\in\bk, s\in S.
\end{align}
If we multiply this element by $s\in S\subset k$, we obtain $0$, but this multiplication is injective, because $s$ is now invertible in $A$.
It follows that $A\otimes A\cong A\otimes_A A$, which is isomorphic to $A$ via the multiplication map.
Hence, in this scenario $\Omega^1_u=0$, and thus $\ker(d)=\ker(d_u)=A$.

Let $E$, and $F$ be in $\AMod$.
Since $J^1_u E=E$, we have that $\AHom(E,F)=\Diff^0_u(E,F)=\Diff^1_u(E,F)=\Hom(E,F)$, cf.\ \cite[\jetspropzeroorder, Proposition 4.3, p.~15]{FMW}.
This is an alternative proof of the fact that a $\bk$-linear map between $S^{-1}\bk$-modules is automatically $S^{-1}\bk$-linear.
For instance, all $\Z$-linear maps between $\Q$-modules are $\Q$-linear.
\end{rmk}
\subsection{Interior product}
\label{ss:Interior_product}
\begin{defi}
Given $\omega\in\Omega^1_d$ and $X\in\VF$, we define the \emph{interior product} of $\omega$ with $X$ by
	\begin{align}
		\langle\omega, X\rangle 
		\colonequals \widetilde{X}\circ\iota^1_{d,A}(\omega)\in A.
	\end{align}
\end{defi}

Notice that we can interpret the map \eqref{eq:iso_vfs_BM_left_vfs} as
\begin{align}\label{eq:VF_to_dual_forms}
\VF\longrightarrow \AHom(\Omega^1_d,A)
&\hfill&
X\longmapsto \langle -, X\rangle,
\end{align}
where $\langle -,X\rangle$ is the map $\omega\longmapsto \langle\omega, X\rangle$.

\begin{prop}\label{prop:pairing_props}\
\begin{enumerate}
\item\label{prop:pairing_props:1} The map $\langle-, -\rangle\colon \Omega^1_d\times \VF\to A$ descends to an $A$-bilinear map $\Omega^1_d\otimes_A \VF\to A$.
\item\label{prop:pairing_props:2} $\langle-,X\rangle\circ d=X$ for all $X\in\VF$, and more specifically, the map \eqref{eq:VF_to_dual_forms} is an isomorphism.
As a consequence, $\langle-,-\rangle$ is nondegenerate in its right component.
\end{enumerate}
\end{prop}
\begin{proof}\
\begin{enumerate}
\item Let $X\in\VF$, $\omega\in\Omega^1_d$, and $a\in A$.
By Proposition \ref{prop:tangentbundle_is_bundle}, the mapping $X\mapsto\langle-,X\rangle$ is left $A$-linear, so we have
\begin{equation}
\begin{split}
\langle \omega, aX\rangle
&=\langle -, aX\rangle(\omega)
=(a\langle -, X\rangle)(\omega)
=\langle -, X\rangle(R_a(\omega))
=\langle \omega a, X\rangle.
\end{split}
\end{equation}
Thus the inner product factors to a map $\langle-,-\rangle\colon \Omega^1_d\otimes_A \VF\to A$.

By Proposition \ref{prop:fields_and_tangent_space}.\eqref{prop:fields_and_tangent_space:i}, we have right $A$-linearity of this map, as
\begin{equation}
\begin{split}
\langle \omega, Xa\rangle
&=\langle -, Xa\rangle(\omega)
=(\langle -, X\rangle a)(\omega)
=R_a\circ \langle -, X\rangle(\omega)
=\langle \omega, X\rangle a.
\end{split}
\end{equation}
Left $A$-linearity instead follows from the fact that $\langle -,X\rangle\in\AHom(\Omega^1_d,A)$.
\item It follows from Proposition \ref{prop:fields_and_tangent_space}.\eqref{prop:fields_and_tangent_space:i}.\qedhere
\end{enumerate}
\end{proof}
This pairing corresponds to the evaluation maps mentioned in \cite[§2.7, p.~168]{BeggsMajid}.

\subsection{Brackets of vector fields}
\label{ss:brakets_VF}
In this subsection we provide a generalization of the Lie bracket of vector fields to the noncommutative setting.
There are two perspectives on the classical notion of Lie bracket.
One of them sees it as a binary operation on the vector space of vector fields satisfying certain algebraic rules.
The other sees it as a composition of antisymmetrized vector fields.

We take the latter approach, where we suitably generalize antisymmetrization.
The composition of two vector fields still vanishes on $\ker(d)$, but in general it is a linear differential operator of order at most $2$.
The antisymmetrization operation ensures that the composition of antisymmetrized vector fields, is still a differential operator of order at most $1$, so that the bracket takes values in $\VF$.

\begin{prop}\label{prop:natural_pairing}
	There is a well-defined natural pairing between $\Omega^1_d \otimes_A \Omega^1_d$ and $\VF \otimes_A\VF$, given by the evaluation map,
		\begin{align}\label{eq:natural_pairing}
		\llangle - , -\rrangle\colon (\Omega^1_d\otimes_A \Omega^1_d)\times(\VF \otimes_A\VF) \longrightarrow A,
		&\hfill&
		(\alpha \otimes_A \beta,X\otimes_A Y) \longmapsto  \langle \alpha\langle\beta,X\rangle, Y\rangle,
		\end{align}
		descending to an $A$-bilinear map $\Omega^1_d\otimes_A \Omega^1_d\otimes_A \VF \otimes_A\VF \to A$.

		Moreover, the map $\VF \otimes_A \VF \rightarrow \AHom(\Omega^1_d \otimes_A \Omega^1_d,A)$ is a morphism in $\AModA$.
\end{prop}
\begin{proof}
	This mapping is well defined, as for all $\alpha,\beta\in\Omega^1_d$, $X,Y\in\VF$, and $a\in A$, we have
	\begin{equation}
	\begin{split}
		&\llangle \alpha a \otimes_A \beta, X \otimes_A Y\rrangle
		=\langle \alpha a \langle\beta,X\rangle, Y\rangle
		=\langle \alpha \langle a \beta,X\rangle, Y\rangle
		=\llangle \alpha  \otimes_A a \beta, X \otimes_A Y\rrangle,\\
		&\llangle \alpha \otimes_A \beta, X a \otimes_A Y\rrangle
		=\langle \alpha\langle\beta,X a \rangle, Y\rangle
		=\langle \alpha \langle \beta,X\rangle a, Y\rangle
		=\langle \alpha \langle \beta,X\rangle, a Y\rangle
		=\llangle \alpha  \otimes_A \beta, X \otimes_A a  Y\rrangle,
	\end{split}
	\end{equation}
	by Proposition \ref{prop:pairing_props}.\eqref{prop:pairing_props:1}.
	For the same result, we prove that it descends to the tensor product over $A$
	\begin{equation}
	\llangle \alpha \otimes_A \beta a, X \otimes_A Y\rrangle
		=\langle \alpha \langle\beta a,X\rangle, Y\rangle
		=\langle \alpha \langle \beta,a X\rangle, Y\rangle
		=\llangle \alpha  \otimes_A \beta, a X \otimes_A Y\rrangle,
	\end{equation}
	and that the resulting map is $A$-bilinear, as so is $\langle -,-\rangle \colon \Omega^1_d\otimes_A \VF$.
	
	Using the last result, we also obtain that the map $X\otimes_A Y\mapsto \llangle - , X\otimes_A Y\rrangle$ is $A$-bilinear, in fact.
	\begin{equation}
	\begin{split}
		&\llangle - , a X \otimes_A Y\rrangle
		=\llangle R_a - ,  X \otimes_A Y\rrangle
		=\llangle - ,  X \otimes_A Y\rrangle\circ R_a
		=a\llangle - ,  X \otimes_A Y\rrangle,\\
		&\llangle - , X \otimes_A Y a\rrangle
		=R_a \circ \llangle - ,  X \otimes_A Y\rrangle
		=\llangle - ,  X \otimes_A Y\rrangle a.\\
	\end{split}
	\end{equation}
\end{proof}

\begin{defi}\label{def:Asym}
	Let
	\begin{align}
		\Ann(S^2_d)
		\colonequals \bigcap_{\sigma\in S^2_d}\ker(\llangle \sigma,- \rrangle)
		\subseteq \VF \otimes_A \VF
	\end{align} 
	be the $A$-bimodule of \emph{antisymmetrized (in a suitable noncommutative sense) vector fields}, and let the $A$-bimodule
	\begin{align}
	\Asym(\VF)
	= \otimes_A^{-1}(\Ann(S^2_d))
	\subseteq \VF \otimes \VF
	\end{align} 
	be its preimage over the canonical projection $\otimes_A \colon \VF \otimes \VF \rightarrow \VF \otimes_A \VF$.
\end{defi}
One can view $\Asym(\VF)$ as the $A$-bimodule of (linear combinations of) vector fields that yield a vector field under composition.
\begin{theo}\label{theo:asymtodiff}
	We have the composition map
	\begin{align}
		\comp\colon\VF \otimes \VF \longrightarrow \Diff^2_d(A,A),
		&\hfill&
		X\otimes Y \longmapsto Y \circ X.
	\end{align}
	Suppose that $\Omega^2_d=\Omega^2_{d,\text{max}}$, and $\sum_i X_i \otimes Y_i \in \Asym(\VF)$, then $\comp(\sum_i X_i \otimes Y_i)\in \VF$.
\end{theo}
\begin{proof}
	Since $\VF$ consists of linear differential operators of order at most one, the composition is a linear differential operator of order at most $2$, so the map $\comp$ is well defined.
	We compute the symbol $\symb^2_d(\comp(\sum_i X_i \otimes Y_i))$ using the map $r^2_{d,A,A}$ from Proposition \ref{prop:symbols_representation_well-defined}.
	Since we are considering second order jets and the second order of the exterior algebra is maximal, $r^2_{d,A,A}$ is a well defined mono, cf.\ Proposition \ref{prop:symbol_representability_low_dimension}.\eqref{prop:symbol_representability_low_dimension:2}.
	Furthermore, we have
	\begin{equation}	\label{eq:humongous}
	\begin{split}
		r^2_{d,A,A}\left(\symb^2_d\left(\comp\left(\sum_i X_i \otimes Y_i\right)\right)\right)\left(\sum_j \alpha_j\otimes \beta_j\right)
		&=\reallywidetilde{\comp\left(\sum_i X_i \otimes Y_i\right)}  \circ \iota^2_{d,A} \left(\sum_j \alpha_j\otimes \beta_j\right)\\
		&=\sum_{i} \widetilde{Y_i}\circ J^1_d(\widetilde{X_i}) \circ  l^{1,1}_{d,A}\circ \iota^2_{d,A} \left(\sum_j\alpha_j\otimes \beta_j\right)\\
		&=\sum_{i} \widetilde{Y_i}\circ J^1_d(\widetilde{X_i}) \circ  l^{2}_{d,A}\circ \iota^2_{d,A} \left(\sum_j \alpha_j\otimes \beta_j\right)\\
		&=\sum_{i} \widetilde{Y_i}\circ J^1_d(\widetilde{X_i}) \circ \iota^1_{d,J^1_d A} \circ\Omega^1_d(\iota^1_{d,A})\circ\iota^2_{\wedge,A}\left(\sum_j \alpha_j\otimes \beta_j\right)\\
		&=\sum_{i,j} \widetilde{Y_i}\circ \iota^1_{d,A}\circ \Omega^1_d(\widetilde{X_i}) \circ\Omega^1_d(\iota^1_{d,A})(\alpha_j\otimes \beta_j)\\
		&=\sum_{i,j} \langle \Omega^1_d(\widetilde{X_i} \circ \iota^1_{d,A})\circ\iota^2_{\wedge,A}(\alpha_j\otimes \beta_j),Y_i\rangle\\
		&=\sum_{i,j} \langle \alpha_j(\widetilde{X_i} \circ \iota^1_{d,A})(\beta_j),Y\rangle\\
		&=\sum_{i,j} \langle \alpha_j \langle \beta_j,X_i\rangle  ,Y_i\rangle\\
		&=\sum_{i,j}\llangle \alpha_j\otimes \beta_j, X_i\otimes_A Y_i\rrangle\\
		&=\llangle \sum_{j} \alpha_j\otimes \beta_j, \sum_{i} X_i\otimes_A Y_i\rrangle\\
		&=0,
	\end{split}
	\end{equation}
	for all $\sum_j \alpha_j\otimes \beta_j \in S^2_d$, cf.\ \cite[\jetsoperatorcomposition]{FMW} and cf.\ \cite[\jetsestwojetDt]{FMW}.
	By Proposition \ref{prop:symbol_representability_low_dimension}.\eqref{prop:symbol_representability_low_dimension:2}, the map $r^2_{d,A,A}$ is injective, and hence the symbol $\symb^2_d(\comp(\sum_i X_i \otimes Y_i))$ vanishes.

	Therefore, $\comp(\sum_i X_i \otimes Y_i)$ is a differential operator of order at most one.
	This operator annihilates $\ker(d)$, as so does $X$, and hence $\comp(\sum_i X_i \otimes Y_i)\in\VF$.
\end{proof}
The computations in \eqref{eq:humongous}, show the commutativity of the following diagram displaying the ingredients of this construction.
\begin{equation}
        \begin{tikzcd}[column sep=40pt]
        			&{\VF}&\Diff^1_d(A,A)\\
                {\operatorname{Asym}(\VF)} & {\VF\otimes\VF} & {\Diff^2_d(A,A)} & {\Symb^2_d(A,A)} \\
                {\Ann(S^2_d)} & {\VF\otimes_A\VF} & {{}_A\operatorname{Hom}(\Omega^1_d \otimes_A \Omega^1_d, A)} & {{}_A\operatorname{Hom}(S^2_d, A)}
                \arrow[hook, from=1-3, to=2-3]
                \arrow[hook, from=1-2, to=1-3]
                \arrow["0", from=1-3, to=2-4]
                \arrow[dashed, from=2-1, to=1-2]
                \arrow["{\otimes_A}"', from=2-1, to=3-1]
                \arrow["{\otimes_A}", from=2-2, to=3-2]
                \arrow[hook, from=3-1, to=3-2]
                \arrow[hook, from=2-1, to=2-2]
                \arrow[phantom, "{\lrcorner}"{description, pos=0.1}, from=2-1, to=3-2]
                \arrow["\comp", from=2-2, to=2-3]
                \arrow["{\symb^2_d}", from=2-3, to=2-4]
                \arrow["{\llangle -,-\rrangle}", from=3-2, to=3-3]
                \arrow["-\circ\iota^2_{\wedge,A}",from=3-3, to=3-4]
                \arrow["r^2_{d,A,A}",hook, from=2-4, to=3-4]
        \end{tikzcd}
\end{equation}

For the following constructions we will assume that $\Omega^2_d=\Omega^2_{d,\text{max}}$ for a given first order differential calculus.
\begin{defi}\label{def:quantumliebracket}
	Given a choice of a $\bk$-linear retract
	\begin{equation}
		\wp\colon\VF \otimes \VF \longrightarrow\Asym(\VF),
	\end{equation}
	we define the map $[-,-]^{\wp}_d \colon \VF \times \VF \rightarrow \VF$ by 
	\begin{equation}\label{eq:quantumliebracket}
		[X,Y]^{\wp}_d \colonequals  c \circ \wp(Y \otimes X),
	\end{equation} 
	and call it the \emph{bracket of vector fields} corresponding to $\wp$.
\end{defi}
\begin{cor}\label{cor:bracket_is_VF}
	For $X,Y \in \VF$, we have $[X,Y]^{\wp}_d \in \VF$.	
\end{cor}
\begin{proof}
	Immediate by Proposition \ref{theo:asymtodiff} given \eqref{eq:quantumliebracket}.
\end{proof}

\begin{rmk}
	Observe that it is necessary to define the bracket on $\VF\otimes \VF$, because even classically, the Lie bracket does not factor through $\VF\otimes_A\VF$, e.g.\
	\begin{equation}
		c(\partial_x \otimes x \cdot \partial_x - x\cdot \partial_x \otimes \partial_x)
		= [x\partial_x, \partial_x]_{dR}
		= -\partial_x.
	\end{equation}
\end{rmk}

\begin{prop}[Classical correspondence]\label{prop:classicalLiebracket}
Let $\bk=\R$ and $A=\smooth{M}$, for a smooth manifold $M$.
If $A$ is equipped with the exterior algebra $\Omega^{\bullet}_{dR}$ arising from the de Rham differential,
then the bracket $[-,-]_{dR}$ corresponding to the canonical antisymmetric map $\VF\otimes_{\R} \VF \longrightarrow \Asym(\VF)$ given by $X\otimes Y \mapsto X\otimes Y -Y\otimes X$ is equivalent to the classical Lie bracket.
\end{prop}
\begin{proof}
Straightforward computation.
\end{proof}

\subsection{Brackets of vector fields on algebras equipped with an inner product}
\label{ss:bracket_examples}
For this subsection, we let $\bk=\R$ or $\C$.
We will treat examples where the algebra comes equipped with an inner product.
In these cases we are able to use an orthogonal projection to define a canonical bracket of vector fields.

\begin{theo}\label{theo:inner_product_brackets}
	Let $\bk$ be a field and $(A,\langle -,-\rangle)$ be a finite-dimensional associative $\bk$-algebra equipped with a $\bk$-bilinear inner product.
	Then for each first order differential calculus $\Omega^1_d$ there is a canonical bracket $[-,-]^{\wp}_d$.
\end{theo}
\begin{proof}
First $\Hom(A,A)\cong A^\ast\otimes A$, where $A^\ast\colonequals \Hom(A,\bk)$.
The inner product induced on the dual space yields an inner product on $A^\ast\otimes A$ given by
\begin{equation}
\langle \alpha\otimes a, \beta\otimes b\rangle
=\langle \alpha, \beta\rangle\langle a, b\rangle.
\end{equation}
Considering the following sequence of inclusions
\begin{equation}
\VF
\subseteq \Diff^1_d(A,A)
\subseteq \Hom(A,A).
\end{equation}
we get an induced inner product on $\VF$.
Subsequently, this induces an inner product on $\VF\otimes \VF$.
Hence, there is a canonical map $\wp$, namely the orthogonal projection to the $A$-subbimodule $\Asym(\VF)$, i.e.\ $\wp\colon \VF\otimes \VF\to \Asym(\VF)$.
This gives the bracket of vector fields via Definition \ref{def:quantumliebracket}.
\end{proof}

\begin{defi}
	A \emph{group algebra} for a finite group $G$ is the unital associative algebra $\bk[G]$ of formal $\bk$-linear combinations of group elements, where the multiplication is the group product extended by linearity.
\end{defi}
\begin{defi}\label{def:CliffordAlgebra}
	Let $V$ be a finite-dimensional $\bk$-vector space equipped with a (nondegenerate) symmetric bilinear form $b\colon V \times V \rightarrow \bk$.
	The \emph{Clifford algebra} $\Cl(V,b)$ is the unital associative algebra generated by formal products of elements of $V$, subject to the relations 
	\begin{equation}
	        v v + b(v, v) 1 = 0
	\end{equation}
	for $v \in V$.
\end{defi}

\begin{cor}\label{cor:clifford_brackets}
	Suppose that either $A=\bk[G]$ is the group algebra of a finite group $G$, or $A$ is a Clifford algebra for a bilinear form associated to an inner product.
Then any first order differential calculus $\Omega^1_d$ on $A$ induces a canonical bracket $[-,-]^{\wp}_d$ on the associated vector fields $\VF$.
\end{cor}
\begin{proof}
Since $\bk[G]$ is equipped with a canonical basis, namely $G$, it also comes equipped with a canonical inner product, defined by declaring this basis to be orthonormal.
Note that this inner product is invariant under the left regular representation of $G$ on $\bk[G]$.

Definition \ref{def:CliffordAlgebra} yields a linear embedding $V \rightarrow \Cl(V,b)$.
When the bilinear form arises from an inner product on $V$, then this inner product can be extended to a canonical inner product on $\Cl(V,b)$, uniquely specified by being invariant under the left multiplication operators $\{L_v | v \in V\}$.

Now the result follows from Theorem \ref{theo:inner_product_brackets}.
\end{proof}

\subsubsection{Terminal calculus for Clifford algebras}

Let us specialize to the case $\bk=\R$ and let $A=\Cl(V)\colonequals \Cl(V,\langle -,- \rangle)$ be the Clifford algebra of a Euclidean inner product $\langle -,- \rangle$ on $V$.

Consider $A=\Cl(V)$, and let $e_1, \dots, e_k$ be an orthonormal basis, where $k=\dim(V)$.
Then we may construct the terminal calculus $\Omega^1_V = \Omega^1_{\{e_1, \dots, e_k \}}$ over $A$, and this calculus is canonical for $A$.
A basis for $\Cl(V)$ is given by the unit $1$ together with the ordered products $e_{i_1 \dots i_l} = e_{i_l} \cdots e_{i_l}$, where $i_{1} < i_{2} < \dots < i_{l}$, where $l\le k$.

\begin{prop}
	The calculus $\Omega^1_V$ is parallelizable, with left basis $de_1, \dots, de_k$.
	The structure equations for $\Omega^1_V$ are given by
	\begin{equation}\label{eq:cliffordstructure}
		d e_{ij} = e_j de_i - e_i de_j -\frac{2}{k-1}\sum_{s=1}^k e_{ij} e_s de_s
	\end{equation}
\end{prop}
\begin{proof}
	The $V$--terminal can be considered as a terminal calculus for a finite number of elements, by taking a basis of $V$. 
	Of course, this will be $e_1, \dots, e_k$.
	The direct sum calculus $\Omega^1 = \Omega^1_{e_1} \oplus \dots \oplus \Omega^1_{e_k}$ then has the property that an algebra element $f$ satisfies $df = 0$ if and only if it is in the kernel of the differential of each component.
	Thus, the direct sum construction implicitly describes the intersection of differential relations.
	Therefore, $\Omega^1 \simeq \Omega^1_V$ as first order calculi, and we can compute properties such as structure relations on the left--hand side.
	Thus, using the formula for the terminal calculus for a singleton from \cite[\jetsdefiSterminal]{FMW}, we obtain
	\begin{equation}
		\begin{split}
			&de_1 = (0, -2 e_{12}, \dots, -2 e_{1k})\\
			&de_2 = (2 e_{12}, 0, -2 e_{23},  \dots, -2 e_{1k})\\
			&\vdots\\
			&de_{k-1} = (2 e_{1 (k-1)},  \dots, 2 e_{(k-2)(k-1)}, 0, -2 e_{(k-1)k})\\
			&de_{k} = (2 e_{1 k},  \dots, 2 e_{(k-1)k}, 0)
		\end{split}
	\end{equation}
	These are obviously independent and provide a left basis over $\Cl(\R^k)$, so $\Omega^1_V$ is parallelizable of cotangent dimension $k$.
	The calculus is also central as a bimodule, with central generators
	\begin{equation}
		E_i = (0, \dots, 0, 1, 0, \dots, 0),
	\end{equation}
	where the nonzero entry is in the $i$th slot.
	We can rewrite the differential basis in terms of the central basis by inverting a $k\times k$ matrix with entries in $\Cl(\R^k)$. 
	Doing so, one obtains
	\begin{equation}
		\begin{split}
			&E_1 = \tfrac{1}{2(k-1)}(-(k-2)de_{1} -e_{12}de_2    - e_{13}de_3    - \dots -e_{1k}de_k )\\
			&E_2 = \tfrac{1}{2(k-1)}(e_{12}de_1    -(k-2)de_{2}  - e_{23}de_3    - \dots -e_{2k}de_k )\\
			&E_3 = \tfrac{1}{2(k-1)}(e_{12}de_1    +e_{23}de_3    - (k-2)de_{3}  - \dots -e_{3k}de_k )\\
			&\vdots\\
			&E_k = \tfrac{1}{2(k-1)}(e_{1k}de_1    +e_{2k}de_3    + e_{2k}de_3    + \dots   -(k-2)de_{k})
		\end{split}
	\end{equation}
	Then using the Leibniz rule, we obtain
	\begin{equation}
		de_{ij} = e_i de_j + de_i \cdot e_j = (0, \dots, 0,  -2 e_j, 0, \dots, 0, 2 e_i, \dots 0) = -2e_j E_i + 2 e_i E_j,
	\end{equation}
	where the nonzero slots are in the $i$th and $j$th positions (and $i<j$).
	Inserting the formula 
	\begin{equation}
		\begin{split}
			E_i 	&= \tfrac{1}{2(k-1)}(e_{1i}de_1+\dots +e_{(i-1)i}de_{i-1} - (k-2)de_{i} -e_{i(i+1)}de_{i+1}   - \dots -e_{ik}de_k )\\
				&= \tfrac{-1}{2(k-1)}((k-2)de_i +    \sum_{j\not= i} e_i e_j de_j)
		\end{split}
	\end{equation}
	into $de_{ij}$ (and the same for $j$), we obtain
	\begin{equation}
		de_{ij} = \frac{1}{k-1}\big( \sum_{s\not=i} e_j e_i e_s de_s +(k-2)e_j de_i -\sum_{s\not=j} e_i e_j e_s de_s -(k-2)e_i de_j \big)
	\end{equation}
	and extracting the $j$th term from the first sum and the $i$th term from the second sum, as well as using anticommutativity to exchange product orders, we obtain the expression
	\begin{equation}
		de_{ij}	= \frac{1}{k-1}\big( \sum_{i\not= s\not=j} -2e_i e_j e_s de_s +(k-3)e_j de_i -(k-3)e_i de_j \big)
		= \frac{1}{k-1}\big( (k-3)e_j de_i -(k-3)e_i de_j - 2\sum_{i\not= s\not=j} e_{ij}e_s de_s \big)
	\end{equation}
	and subtracting terms while adding equivalent terms into the sum yields the desired expression \eqref{eq:cliffordstructure}.
\end{proof}

\begin{prop}
	The terminal calculus $\Omega_{i,j}$ for $\Hq$ from \cite[\jetsssquaternions]{FMW} is $\Omega^1_V$ for $\Cl(V)$ when $V=\R^2$.
\end{prop}
\begin{proof}
	It is well known that $\Cl(\R^2)\simeq \Hq$ as associative algebras.
	Under this isomorphism, the subspace spanned by $\{i,j\}$ is precisely the image of the embedding $\R^2 \rightarrow \Cl(\R^2)$.
\end{proof}
\begin{eg}
	Let us compute the bracket of vector fields for $\Hq$ equipped with $\Omega_{i,j}$.
The vector fields can be written as sums of partial derivatives $\partial_i, \partial_j$ with quaternion coefficients on the right (i.e. Right multiplication operators $R_x$ on the left).
We have that the canonical projector $\wp$ is characterized by its kernel. 
Computing the orthogonal complement to the image of $\Asym(\VF)$ in $\VF \otimes \VF$, we obtain
\begin{equation}
	\ker \wp = \langle  s_\Hq(\partial_i \otimes \partial_j - \partial_j \otimes \partial_i) \rangle_\Hq
\end{equation}
where we the span is by right multiplication by elements of $\Hq$, and
\begin{equation}
	s_\Hq(x) = -x + i\cdot x \cdot i+ j\cdot x \cdot j+k\cdot x \cdot k.
\end{equation}
This allows us to compute the brackets, and we compile the multiplication table for a basis of quaternionic vector fields into Table \ref{tab:quatbrackets} below:

\begin{table}[H]
	\centering
	\begin{tabular}{|c||c|c|c|c|c|c|c|c|}
		\hline
		$[-,-]_\wp$ 		&$\partial_i$ 		&$\partial_i \cdot i$ 	&$\partial_i \cdot j$ 				&$\partial_i \cdot k$				&$\partial_j$ 		&$\partial_j \cdot i$ 	&$\partial_j \cdot j$				&$\partial_j \cdot k$ \\
	%	\hline\hline
		\hhline{|=# =| =| =| =| =| =| =| =|}
		$\partial_i$ 		&$0$			&$\partial_i$ 		&$0$						&$-\partial_i \cdot j$				&$0$			&$\partial_j$ 					&$0$			&$-\partial_j \cdot j$\\
		\hline                                                                                                                                                                                                                                                                  
		$\partial_i\cdot i$ 	&$0$			&$\partial_i \cdot i$	&$0$						&$\partial_i \cdot k$				&$0$			&$\partial_j \cdot i$				&$0$			&$\partial_j \cdot k$\\
		\hline                                                                                                                                                                                                                                                                  
		$\partial_i\cdot j$ 	&$0$			&$\partial_i \cdot j$	&$0$						&$\partial_i$					&$-\partial_i$		&$-\partial_i \cdot i + \partial_j \cdot j$	&$\partial_i \cdot j$	&$\partial_i \cdot k + \partial_j$\\         
		\hline                                                                                                                                                                                                                                                                  
		$\partial_i\cdot k$ 	&$0$			&$\partial_i \cdot k$	&$0$	 					&$-\partial_i \cdot i$				&$\partial_i \cdot i$	&$\partial_j \cdot k - \partial_i$		&$\partial_i \cdot k$	&$-\partial_i \cdot j - \partial_j \cdot i$\\
		\hline
		$\partial_j$ 		&$0$			&$0$ 	 		&$\partial_i$					&$\partial_i \cdot i$				&$0$			&$0$						&$\partial_j$		&$\partial_j \cdot i$\\ 
		\hline
		$\partial_j\cdot i$ 	&$-\partial_j$		&$\partial_j \cdot i$	&$\partial_i \cdot i - \partial_j \cdot j$	&$\partial_j \cdot k - \partial_i$		&$0$			&$0$						&$\partial_j \cdot i$	&$-\partial_j$\\ 
		\hline
		$\partial_j\cdot j$ 	&$0$			&$0$			&$\partial_i \cdot j$				&$\partial_i \cdot k$				&$0$			&$0$						&$\partial_j \cdot j$	&$\partial_j \cdot k$\\
		\hline
		$\partial_j\cdot k$ 	&$-\partial_j \cdot j$	&$\partial_j \cdot k$	&$\partial_i \cdot k + \partial_j$	 	&$-\partial_i \cdot j - \partial_j \cdot i$	&$0$			&$0$						&$\partial_j \cdot k$	&$-\partial_j \cdot j$\\
		\hline 
	\end{tabular}
	\caption{Brackets of quaternionic vector fields}
	\label{tab:quatbrackets}
\end{table}
\end{eg}
\begin{rmk}
	
We note that, while the algebra $\Hq \simeq \Cl(\R^2)$ can be seen as a super-algebra, it is not a supercommutative superalgebra, and in particular the brackets do not turn out to be graded commutators in general.
Instead, the right span of $\partial_i$ forms a subalgebra under the bracket, and for this subalgebra the bracket is simply the composition.
The same is true for the right span of $\partial_j$.
For brackets of two vector fields with opposite basis elements, the bracket turns out to be be the anti-commutator if at least one coefficient is $k$, or both coefficients are equal to $1$, and to be the commutator if at least one coefficient is one of $i$ or $j$ and the coefficients are different.
If the coefficient is the same, but the basis fields are different, then we get the anti-commutator.
Here commutator, anti-commutator or composition refers to the vector fields seen as differential operators, e.g.
\begin{equation}
	[\partial_i, \partial_j \cdot i]_\wp = \partial_i \circ \partial_j \cdot i -\partial_j \cdot i \circ \partial_i = \partial_j.
\end{equation}
\end{rmk}

\begin{eg}
	Let us now consider the clifford algebra $\Cl(\R^3)$ equipped with $\Omega^1_V$.
	The structure equations become
	\begin{equation}
		\begin{split}
			de_{12} &= -e_{123}de_3\\
			de_{13} &= e_{123}de_2\\
			de_{23} &= -e_{123}de_1
		\end{split}
		\label{eq:cl3structureeq}
	\end{equation}
	and we note that the element $e_{123}$ is central and locally constant: $de_{123}=0$.
	Therefore it is not enough for a differential operator $X$ of order at most 1 satisfy $X(1)=0$ to be a vector field, following Definition \ref{def:vectorfields} we must also have $X(e_{123})=0$.
	Then the vector fields are again spanned by operators of the form $X(a) = \partial_b \cdot c (a)$ where $b$ is a basis element of $V=\R^3$ and $c \in \Cl(\R^3)$.
	The structure equations provides the description $N_d=\langle d_ue_{12} +e_{123}d_ue_3, d_ue_{23} -e_{123}d_ue_1, d_ue_{23} +e_{123}d_ue_1 \rangle_{\Cl(\R^3)} $
	We apply the construction \cite[\jetseqStwomindefNd]{FMW} to obtain a set of generators for $S^2_d$, i.e.
	\begin{equation}
		d\otimes d: N_d \rightarrow S^2_d, \quad \sum_i a_i \otimes_\R b_i \mapsto \sum_i da_i \otimes_{\Cl(\R^3)} db_i
	\end{equation}
	The above construction then takes a generating set to a generating set.
	But then, since all coefficients in the structure equations are locally constant, it is immediate that $S^2_d = \{0 \otimes 0\}$ since $\{0\}$ is a generating set.
	Therefore, we have
	\begin{equation}
		\VF \otimes_{\Cl(\R^3)} \VF = \Ann(S^2_d)
	\end{equation}
	and by the retraction property from Def. \ref{def:quantumliebracket}, we have $\Asym(\VF) \simeq \VF \otimes_\R \VF$ with $\wp$ an isomorphism.
	Therefore we have
	\begin{prop}
		For the algebra $\Cl(\R^3)$ equipped with the maximal prolongation of the calculus $\Omega^1_V$, there is a unique bracket and it coincides with the composition map: for all vector fields $X,Y$
		\begin{equation}
			[X,Y]_\wp = X \circ Y.
		\end{equation}
	\end{prop} 
\end{eg}

\subsection{Elemental and primitive brackets}
\label{ss:elemental_brakets_VF}
In this subsection we provide a variation of the construction in §\ref{ss:brakets_VF} for the construction of Lie brackets for vector fields that makes use of elemental jets.
The reason for considering this generalization is that differential operators of order at most $1$ coincide with elemental differential operators of order at most $1$.
This means that a linear combination of compositions of vector fields is a vector field if and only if it is an elemental differential operator of order at most $1$.
Therefore, in this section we derive new sufficient conditions for establishing this closure in the setting of elemental differential operators, and further show that these conditions are also necessary, cf.\ Theorem \ref{theo:wasymtodiff}.

For this section we will start from the initial data of a first order differential calculus $\Omega^1_d$ over the $\bk$-algebra $A$ such that $\Omega^1_d$ is flat in $\ModA$.
This in particular also implies that the following construction coincides with its primitive analogue.

In the spirit of Definition \ref{def:Asym}, we give the following.
\begin{defi}\label{def:elementalAsym}
	Let $\Omega^1_d$ be a first order differential calculus over $A$ such that $\Omega^1_d$ is flat in $\ModA$, then we have
	\begin{align}
		\Ann(\SS^2_d)
		\colonequals \bigcap_{\sigma\in \SS^2_d}\ker(\llangle \sigma,- \rrangle)
		\subseteq \VF \otimes_A \VF
	\end{align} 
	be the $A$-subbimodule of \emph{elementally antisymmetrized vector fields}, and let the $A$-subbimodule
	\begin{align}
	\WAsym(\VF)
	= \otimes_A^{-1}(\Ann(\SS^2_d))
	\subseteq \VF \otimes \VF
	\end{align} 
	be its preimage over the canonical projection $\otimes_A \colon \VF \otimes \VF \to \VF \otimes_A \VF$.
\end{defi}

\begin{theo}\label{theo:wasymtodiff}
	Let $\Omega^1_d$ be a first order differential calculus over $A$.
	Consider the composition map
	\begin{align}
		\Scomp\colon\VF \otimes \VF \longrightarrow \WD^2_d(A,A),
		&\hfill&
		X\otimes Y \longmapsto Y \circ X.
	\end{align}
	If $\Omega^1_d$ is flat in $\ModA$, and $\sum_i X_i \otimes Y_i \in \WAsym(\VF)$, then $\Scomp(\sum_i X_i \otimes Y_i)\in \VF$.
\end{theo}
\begin{proof}
Before starting with the proof, notice that, for the maximal prolongation $\Omega^{\bullet}_{d,\text{max}}$ of $\Omega^1_d$, in the special case of $n=m=1$, the bottom square of \eqref{diag:wanted_elemental_l} reduces to the following:
\begin{equation}
\begin{tikzcd}
	{\SJ^2_d} & {\SJ^1_d \SJ^1_d} & {\SJ^1_d J^1_d} \\
	{J^2_d} && {J^1_dJ^1_d}
	\arrow["{\Sl^{1,1}_d}", hook, from=1-1, to=1-2]
	\arrow[Rightarrow, no head, from=1-2, to=1-3]
	\arrow[Rightarrow, no head, from=1-3, to=2-3]
	\arrow["{\iota_{\SJ^2_d}}"', hook, from=1-1, to=2-1]
	\arrow["{l^{1,1}_d}"', hook, from=2-1, to=2-3]
\end{tikzcd}	
	\label{diag:reduced_wanted_skel_l}
\end{equation}
where $J^2_d$ denotes the $2$-jet corresponding to $\Omega^{\bullet}_{d,\text{max}}$.
The map $l^{1,1}_d=l^2_d$ is a mono by construction, and the commutativity then implies that $\Sl^{1,1}_d$ is also a mono.

\begin{equation}
\label{diag:spaceship}
\begin{tikzcd}
	{\SS^2_d} & {S^2_{d,\text{min}}} & {\Omega^1_d \Omega^1_d} & {\Omega^1_dJ^1_d = \Omega^1_d \SJ^1_d} \\
	{\SJ^2_d} & {J^2_d} && {J^1_dJ^1_d = \SJ^1_d \SJ^1_d}
	\arrow["{\Sl^2_d}", bend right=.6cm, dashed, hook, from=2-1, to=2-4]
	\arrow["{\iota_{\SJ^2_d}}", hook, from=2-1, to=2-2]
	\arrow["{l^2_d}", hook, from=2-2, to=2-4]
	\arrow["{\Siota^2_d}"', hook, from=1-1, to=2-1]
	\arrow["{\iota^2_d}", hook, from=1-2, to=2-2]
	\arrow["{\iota_{\SS^2_d}}"', hook, from=1-1, to=1-2]
	\arrow["{\iota^2_\wedge}"', hook, from=1-2, to=1-3]
	\arrow["{\Omega^1_d(\iota^1_d)}"', hook, from=1-3, to=1-4]
	\arrow["{\iota^1_{d,J^1_d}}", hook, from=1-4, to=2-4]
\end{tikzcd}
\end{equation}
The right square of \eqref{diag:spaceship} commutes by the upper left square in \cite[\jetsestwojetDt]{FMW}.
The left square of \eqref{diag:spaceship} commutes by Lemma \ref{lemma:elemental_forms_are_forms}, cf.\ Remark \ref{rmk:conditions_lemma elemental_symm_forms}.
The bottom triangle commutes by \eqref{diag:reduced_wanted_skel_l}.

	With these preliminaries now complete, we are now ready to prove the theorem.
	Since $\VF$ consists of (elemental) linear differential operators of order at most $1$, the composition is an (elemental) linear differential operator of order at most $2$, so the map $\Scomp$ is well defined.
	We compute the symbol $\symb^2_d(\Scomp(\sum_i X_i \otimes Y_i))$ using the mono $\Sr^2_{d,A,A}$ from Proposition \ref{prop:elemental_symbols_representation}. 
	Furthermore, for all $\sum_j \alpha_j\otimes \beta_j \in \SS^2_d$, we have
	\begin{equation}	\label{eq:lesshumongous}
	\begin{split}
		\Sr^2_{d,A,A}\left(\symb^2_d\left(\Scomp\left(\sum_i X_i \otimes Y_i\right)\right)\right)\left(\sum_j \alpha_j\otimes \beta_j\right)
		&=\reallywidetilde{\Scomp\left(\sum_i X_i \otimes Y_i\right)}  \circ \Siota^2_{d,A} \left(\sum_j \alpha_j\otimes \beta_j\right)\\
		&=\sum_{i} \widetilde{Y_i}\circ J^1_d(\widetilde{X_i}) \circ  \Sl^{1,1}_{d,A}\circ \Siota^2_{d,A} \left(\sum_j \alpha_j\otimes \beta_j\right)\\
		&=\sum_{i} \widetilde{Y_i}\circ J^1_d(\widetilde{X_i}) \circ \iota^1_{d,J^1_d A} \circ\Omega^1_d(\iota^1_{d,A})\circ\iota^2_{\wedge,A} \circ \iota_{S^2_d}\left(\sum_j \alpha_j\otimes \beta_j\right)\\
		&=\sum_{i,j} \widetilde{Y_i}\circ \iota^1_{d,A}\circ \Omega^1_d(\widetilde{X_i}) \circ\Omega^1_d(\iota^1_{d,A})(\alpha_j\otimes \beta_j)\\
		&=0,
	\end{split}
	\end{equation}
	The first equality follows from the definition of $\Sr^2_d$ and $\symb^2_d$.
	The second equality follows from \eqref{eq:elementaljetcompositionlift} and the definition of $\Scomp$.
	The third equality holds by \eqref{diag:spaceship}. 
	Since $\SJ^1_d = J^1_d$, the final two equalities follow from the computations in \eqref{eq:humongous}.
	By Proposition \ref{prop:elemental_symbols_representation}, the map $\Sr^2_{d,A,A}$ is a monomorphism, and hence the symbol $\symb^2_d(\Scomp(\sum_i X_i \otimes Y_i))$ vanishes.

	Therefore, $\Scomp(\sum_i X_i \otimes Y_i)$ is an (elemental) differential operator of order at most $1$.
	This operator annihilates $\ker(d)$ because $X_i$ does, and hence $\Scomp(\sum_i X_i \otimes Y_i)\in\VF$.
\end{proof}
The content of the proof of Theorem \ref{theo:wasymtodiff} can be summarized in the following commutative diagram.
\begin{equation}\label{diag:elementalbracket}
        \begin{tikzcd}[column sep=40pt]
        			&{\VF}&\WD^1_d(A,A)\\
                {\WAsym(\VF)} & {\VF\otimes\VF} & {\WD^2_d(A,A)} & {\WS^2_d(A,A)} \\
                {\Ann(\SS^2_d)} & {\VF\otimes_A\VF} & {{}_A\operatorname{Hom}(\Omega^1_d \otimes_A \Omega^1_d, A)} & {{}_A\operatorname{Hom}(\SS^2_d, A)}
                \arrow[hook, from=1-3, to=2-3]
                \arrow[hook, from=1-2, to=1-3]
                \arrow["0", from=1-3, to=2-4]
                \arrow[dashed, from=2-1, to=1-2]
                \arrow["{\otimes_A}"', from=2-1, to=3-1]
                \arrow["{\otimes_A}", from=2-2, to=3-2]
		\arrow[from=3-1, hook, to=3-2]
		\arrow[phantom, "{\lrcorner}"{description, pos=0.1}, from=2-1, to=3-2]
		\arrow[from=2-1, hook, to=2-2]
                \arrow["\Scomp", from=2-2, to=2-3]
                \arrow["{\symb^2_d}", from=2-3, to=2-4]
                \arrow["{\llangle -,-\rrangle}", from=3-2, to=3-3]
		\arrow["-\circ\iota^2_{\wedge,A} \circ \iota_{\SS^2_d}",from=3-3, to=3-4]
                \arrow["\Sr^2_{d,A,A}",hook, from=2-4, to=3-4]
        \end{tikzcd}
\end{equation}

\begin{prop}\label{prop:WAsymisKer}
	Let $\Omega^1_d$ be a first order differential calculus over $A$ such that $\Omega^1_d$ is flat in $\ModA$.	
	Then following holds:
	\begin{equation}
		\WAsym(\VF)
		=\ker(\symb^2_d \circ \Scomp)
		\subseteq \VF \otimes \VF.
\end{equation}
\end{prop}
\begin{proof}

The annihilator $\Ann(\SS^2_d)$ is, by definition, the kernel of the composition $(-\circ\iota^2_{\wedge,A} \circ \iota_{\SS^2_d})\circ \llangle -,- \rrangle $.
	By the commutativity of diagram \eqref{diag:elementalbracket}, and the universal property of the kernel, there exists a unique map from ${\ker(\symb^2 \circ \Scomp)} \rightarrow \Ann(\SS^2_d)$ making the following square commute.
\begin{equation}\label{diag:wannabepb}
	\begin{tikzcd}
	{\ker(\symb^2 \circ \Scomp)} & {\VF \otimes \VF} \\
	{\Ann(\SS^2_d)} & {\VF \otimes_A \VF}
	\arrow["{\otimes_A}", two heads, from=1-2, to=2-2]
	\arrow[hook, from=1-1, to=1-2]
	\arrow["{\otimes_A}", dashed, two heads, from=1-1, to=2-1]
	\arrow[hook, from=2-1, to=2-2]
	\end{tikzcd}
	\end{equation}	
	By Definition \ref{def:elementalAsym}, all we have to do now is to prove that \eqref{diag:wannabepb} is a pullback square in $\Mod$.
	Let $M$ be in $\Mod$, and let $x\colon M\to \VF\otimes \VF$ and $y\colon M\to \Ann(\SS^2_d)$ be two $\bk$-linear maps such that the left square of the following diagram commutes.
\begin{equation}\label{eq:pbproof}
\begin{tikzcd}[column sep=40pt]
                M & {\VF\otimes\VF} & {\WD^2_d(A,A)} & {\WS^2_d(A,A)} \\
                {\Ann(\SS^2_d)} & {\VF\otimes_A\VF} & {{}_A\operatorname{Hom}(\Omega^1_d \otimes_A \Omega^1_d, A)} & {{}_A\operatorname{Hom}(\SS^2_d, A)}
                \arrow["y"', from=1-1, to=2-1]
                \arrow["{\otimes_A}", from=1-2, to=2-2]
		\arrow[from=2-1, hook, to=2-2]
		\arrow["x",from=1-1, to=1-2]
                \arrow["\Scomp", from=1-2, to=1-3]
                \arrow["{\symb^2_d}", from=1-3, to=1-4]
                \arrow["{\llangle -,-\rrangle}", from=2-2, to=2-3]
		\arrow["-\circ\iota^2_{\wedge,A} \circ \iota_{\SS^2_d}",from=2-3, to=2-4]
                \arrow["\Sr^2_{d,A,A}",hook, from=1-4, to=2-4]
        \end{tikzcd}
\end{equation}
	By \eqref{diag:elementalbracket}, the right square of \eqref{eq:pbproof} commutes.
	We compose the diagonal map of this square with the map $(-\circ\iota^2_{\wedge,A} \circ \iota_{\SS^2_d})\circ \llangle -,- \rrangle$, which is zero, since the bottom horizontal composition of \eqref{eq:pbproof} vanishes.
	Therefore, also the composition $\Sr^2_{d,A,A}\circ \symb^2_d\circ \Scomp\circ x$ vanishes.
	By Proposition \eqref{prop:elemental_symbols_representation}, the map $\Sr^2_{d,A,A}$ is a mono, so $\symb^2_d\circ \Scomp\circ x=0$.
	This implies that the map $x$ factors uniquely through the kernel inclusion $\ker(\symb^2_d\circ \Scomp)\hookrightarrow \VF\otimes \VF$.
	Call this factorization $\overline{x}\colon M\to \ker(\symb^2_d\circ \Scomp)$.
	The diagram \eqref{diag:wannabepb} is a pullback square if we prove that $\otimes_A\circ \overline{x}=y$.
	This equation is true if and only if both sides coincide once we compose them with the inclusion map $\Ann(\SS^2_d)\hookrightarrow \VF\otimes \VF$.
	This is true as both equal $\otimes_A\circ x$.
\end{proof}
\begin{defi}\label{def:elementalquantumliebracket}
	Given a choice of a $\bk$-linear retract
	\begin{equation}
		\wp\colon\VF \otimes \VF \longrightarrow\WAsym(\VF),
	\end{equation}
	we define the map $[-,-]^{\wp}_d \colon \VF \times \VF \rightarrow \VF$ by 
	\begin{equation}
		[X,Y]^{\wp}_d \colonequals  \Scomp \circ \wp(Y \otimes X),
	\end{equation} 
	and call it the \emph{elemental bracket of vector fields} corresponding to $\wp$.
\end{defi}
Unlike in the setting of §\ref{ss:brakets_VF}, one can also obtain a converse without further assumptions:
\begin{theo}\label{thm:bracketsarebrackets}
	Let $\Omega^1_d$ be a first order differential calculus over $A$ such that $\Omega^1_d$ is flat in $\ModA$.
	Let $[-,-]\colon \VF \times \VF \rightarrow \VF$ be a $\bk$-bilinear map which consists of linear combinations of compositions, in the sense that there exists a $\bk$-linear endomorphism ${{B_{[-,-]}}}$ of $\VF\otimes \VF$ such that the bottom pentagon of the following diagram commutes.
	\begin{equation}
	\label{diag:conversebracket}
	\begin{tikzcd}
		\WAsym(\VF) \\
		{\VF \otimes \VF} & {\VF \otimes \VF} \\
		{\VF \times \VF} & \VF & {\WD^2_d(A,A)}
		\arrow[hook, from=1-1, to=2-2]
		\arrow["{\wp}", dashed, from=2-1, to=1-1]
		\arrow["{{B_{[-,-]}}}"', from=2-1, to=2-2]
		\arrow["\otimes", from=3-1, to=2-1]
		\arrow["\Scomp", from=2-2, to=3-3]
		\arrow["{{[-,-]}}"', from=3-1, to=3-2]
		\arrow[hook,  from=3-2, to=3-3]
	\end{tikzcd}
	\end{equation}
	Then there exists a unique $\bk$--linear map $\wp\colon\VF \otimes \VF \rightarrow \WAsym(\VF)$, making the top triangle in \eqref{diag:conversebracket} commute.
	Moreover, this bracket is a bracket of vector fields in the sense of Definition \ref{def:elementalquantumliebracket}, and we have
	\begin{equation}
		[-,-] = [-,-]^\wp_d.
	\end{equation}
\end{theo}
\begin{proof}
	Since the inclusion of vector fields in elemental differential operators $\VF \hookrightarrow \WD^2_d(A,A)$ factors through the inclusion $\Diff^1_d(A,A) = \WD^1_d(A,A) \hookrightarrow \WD^2_d(A,A)$, the commutativity of the pentagon in \eqref{diag:conversebracket} implies that $ \symb^2_d \circ \Scomp \circ B_{[-,-]}\circ \otimes$ is the zero map.
	Since the map $\otimes$ is an epimorphism, we also have $ \symb^2_d \circ \Scomp \circ B_{[-,-]}=0$.
	Therefore, $B_{[-,-]}$ factors uniquely through the kernel map $\ker(\symb^2_d \circ \Scomp)\hookrightarrow \VF\otimes \VF$.
	By Proposition \ref{prop:WAsymisKer}, we get that the latter is the inclusion $\WAsym(\VF) \rightarrow \VF \otimes \VF$, and we label the factorization $\wp$.
\end{proof}

\begin{rmk}
	Theorem \ref{thm:bracketsarebrackets} says that any ‘‘bracket-like map'' of vector fields which can be interpreted in terms of linear combinations of operator compositions, are examples of elemental brackets of vector fields in the sense of Definition \ref{def:elementalquantumliebracket}.
	If the corresponding $\wp$ factors through $\Asym(\VF)$, then it is a bracket of vector fields also in the sense of Definition \ref{def:quantumliebracket}.
\end{rmk}
\begin{cor}
	Let $\Omega^1_d$ be a first order differential calculus over $A$ such that $\Omega^1_d$ is flat in $\ModA$.
	Then, there exists a bijective correspondence between elemental brackets of vector fields in the sense of Definition \ref{def:elementalquantumliebracket} and $\bk$-linear maps $[-,-]\colon \VF \times \VF \rightarrow \VF$ for which there exists a $\bk$-linear endomorphism ${{B_{[-,-]}}}$ of $\VF\otimes \VF$ such that the bottom pentagon of \eqref{diag:conversebracket} commutes.
\end{cor}

\appendix
\section{Universal exterior algebra}\label{section:appendix}
In this appendix we apply the machinery developed in \cite{FMW} to an algebra equipped with the universal exterior algebra.
\subsection{Universal symmetric forms}
Recall that $\Omega^n_u=T^n_u=(\Omega^1_u)^{\otimes_A n}$, and the wedge product coincides with the tensor product $\otimes_A$, cf.\ \cite[Theorem 1.33, p.~24]{BeggsMajid} with the due generalization to a generic commutative unital ring $\bk$.
Furthermore, in \cite[\jetsexunisymmetricforms]{FMW} we described the functor of symmetric forms for the universal exterior algebra, and we obtained 
\begin{align}
S^0_u=\id_{\AMod},
&\hfill&
S^1_u=\Omega^1_u,
&\hfill&
S^n_u=0,
&\hfill&
\text{for all }n\ge 2.
\end{align}
By definition, we can find all the maps $\iota^n_{\wedge}\colon S^n_u\to \Omega^1_u\circ S^{n-1}_u$, and we obtain the identity for $n=1$ and the zero map for $n\ge 2$.

We can now build the Spencer $\delta$-complex for the universal exterior algebra, cf.\ \cite[\jetsssSpencer]{FMW}.
\begin{prop}\label{prop:universal_Spencer-acyclic}
The Spencer $\delta$-cohomology of the universal exterior algebra vanishes in all degrees ($H^{\bullet,\bullet}_{\delta_u}=0$).
\end{prop}
\begin{proof}
The only $\Omega^k_u\circ S^h_u$ that do not vanish are the ones corresponding to $h=0$ or $h=1$.
Thus the only maps $\delta^{h,k}_u$ that are not zero maps are $\delta^{1,k}_u\colon \Omega^k_u\circ S^1_u\to \Omega^{k+1}_u$, and these happen to be identity maps (or isomorphisms).
It follows that the generic Spencer $\delta$-complex is of the following form.
\begin{equation}
\begin{tikzcd}
0\ar[r]&\cdots\ar[r]&0\ar[r]& \Omega^k_u\circ S^1_u\ar[r,equals]& \Omega^{k+1}_u\ar[r]&0
\end{tikzcd}
\end{equation}
These are exact sequences for all $k$.
\end{proof}

\subsection{Universal jet functors}
We now compute the universal jet functors as endofunctors of $\AModB$.
The $0$-jet functor is the identity functor regardless of the exterior algebra considered.
In \cite[\jetsssuniversalonejet]{FMW} we defined the first universal jet functor as
\begin{equation}
J^1_u=A\otimes -\colon \AModB\longrightarrow \AModB.
\end{equation}
By definition we obtain the nonholonomic jets by iterating $J^1_u$, and thus
\begin{equation}
J^{(n)}_u=A^{\otimes n}\otimes -\colon \AModB\longrightarrow \AModB.
\end{equation}

We can finally compute the universal holonomic jets by the following result.
\begin{theo}\label{theo:universal_higher_jets}
Let $A$ be a $\bk$-algebra, then the corresponding universal jet functors are of the following form
\begin{align}
J^0_u=\id_{\AModB}\colon \AModB\longrightarrow\AModB,
&\hfill&
J^n_u=A\otimes - \colon \AModB\longrightarrow\AModB,
&\hfill&
\forall n\ge 1,
\end{align}
with natural inclusions for $n\ge 2$ at every $E$ in $\AModB$ given by
\begin{align}
l^n_{u,E}\colon J^n_u E=A\otimes E\longrightarrow J^1_u J^{n-1}_u E=A\otimes A\otimes E,
&\hfill&
a\otimes e\longmapsto a\otimes 1\otimes e.
\end{align}
In particular we have
\begin{align}
\iota_{J^n_u E}\colon J^n_u E=A\otimes E\longrightarrow J^{(n)}_u E=A^{\otimes n}\otimes E,
&\hfill&
a\otimes e\longmapsto a\otimes 1 \otimes \cdots\otimes 1\otimes e.
\end{align}
which is a split monomorphism.

The jet projections are
\begin{align}
\pi^{1,0}_u=\cdot\colon A\otimes -\longrightarrow \id_{\AModB},
&\hfill&
\pi^{n,n-1}_u=\id_{A\otimes - }\colon A\otimes - \longrightarrow A\otimes -,
&\hfill&
\forall n\ge 2.
\end{align}
The jet prolongations are
\begin{align}
j^{0}_u=\id\colon \id_{\AModB} \longrightarrow \id_{\AModB},
&\hfill&
j^{n}_u=1\otimes - \colon \id_{\AModB}\longrightarrow A\otimes -,
&\hfill&
\forall n\ge 1.
\end{align}

Furthermore, for all $n\ge 1$ the $n$-jet sequence is short exact.
\begin{equation}
\begin{tikzcd}
0\ar[r]&S^n_u\ar[r,hook,"\iota^n_u"]&J^n_u\ar[r,two heads,"\pi^{n,n-1}_u"]&J^{n-1}_u\ar[r]&0
\end{tikzcd}
\end{equation}
\end{theo}
\begin{proof}
For $n=0$ and $n=1$ , the proposition has already been discussed above.

For $\ge 2$ we proceed by induction.
We first consider the special case $n=2$.
The $2$-jet is defined as the kernel of $\widetilde{\DH}_{d,E}$, cf.\ \cite[\jetsdeftwojet]{FMW}, so we are done if we prove that the suggested $l^2_u$ is the kernel map of $\widetilde{\DH}_{d,E}$.
We can compute the latter map, obtaining
\begin{align}
\widetilde{\DH}_{d,E}\colon J^{(2)}_u E=A\otimes A\otimes E\longrightarrow (\Omega^1_u\ltimes \Omega^2_u)(E),
&\hfill&
a\otimes b\otimes e\longmapsto adb\otimes_A e+ da\wedge db\otimes_A e.
\end{align}
First of all, we have $\widetilde{\DH}_{d,E}\circ l^2_{u,E}=0$, as $d_u 1=0$.
Moreover, the map $l^2_{u,E}$ is a monomorphism, as the product on the first two components is a retraction for it.
We thus have to prove that the image of $l^2_{u,E}$ surjects onto the kernel of $\widetilde{\DH}_{d,E}$.
Let now $\sum_i a_i\otimes b_i\otimes e_i\in \ker(\widetilde{\DH}_{d,E})$.
By the vanishing of the first summand of $\widetilde{\DH}_{d,E}$, $\widetilde{\DH}^{I}_{d,E}$ we get
\begin{equation}\label{eq:DHI-universal}
0
=\sum_i a_id_u b_i\otimes_A e_i
=\sum_i a_i\otimes b_ie_i-\sum_i a_ib_i\otimes e_i.
\end{equation}
The vanishing of $\widetilde{\DH}^{II}_{d,E}$ implies
\begin{align}\label{eq:DHII-universal}
\begin{split}
0
&=\sum_i d_u a_i\wedge d_u b_i\otimes_A e_i\\
&=\sum_i (1\otimes a_i-a_1\otimes 1)\otimes_A (1\otimes b_i-b_i\otimes 1)\otimes_A e_i\\
&=\sum_i 1\otimes a_i\otimes b_ie_i-\sum_i 1\otimes a_ib_i\otimes e_i-\sum_i a_i\otimes 1\otimes b_ie_i+\sum_i a_i\otimes b_i\otimes e_i\\
&=\sum_i 1\otimes (a_i\otimes b_ie_i-a_ib_i\otimes e_i)-\sum_i a_i\otimes 1\otimes b_ie_i+\sum_i a_i\otimes b_i\otimes e_i\\
&=0-\sum_i a_i\otimes 1\otimes b_ie_i+\sum_i a_i\otimes b_i\otimes e_i,
\end{split}
\end{align}
where the last equality follows from \eqref{eq:DHI-universal}.
By \eqref{eq:DHII-universal} we thus get that
\begin{equation}
\sum_i a_i\otimes b_i\otimes e_i
=\sum_i a_i\otimes 1\otimes b_ie_i
=l^2_{u,E}\left(\sum_i a_i\otimes b_i e_i\right).
\end{equation}

Let now $n>2$, the aim is to prove that $l^n_{u,E}$ is the kernel map of
\begin{equation}
\begin{tikzcd}
J^1_u J^{n-1}_u E \ar[r,"J^1_u(l^{n-1}_{u,E})"]&[30pt] J^{(2)}_u J^{n-2}_u E\ar[r,two heads,"\widetilde{\DH}_{d,J^{n-2}_d E}"]&(\Omega^1_u\ltimes \Omega^2_u)(J^{n-2}_d E),
\end{tikzcd}
\end{equation}
which by inductive hypothesis becomes as follows.
\begin{equation}\label{eq:universal_jet_higher_edh}
\begin{tikzcd}
A\otimes A\otimes E \ar[r,"J^1_u(l^{n-1}_{u,E})"]&[40pt] A\otimes A\otimes A\otimes E\ar[r,two heads,"\widetilde{\DH}_{d,A\otimes   E}"]&[40pt] (\Omega^1_u\ltimes \Omega^2_u)(A\otimes E)\cong (\Omega^1_u\oplus \Omega^2_u)\otimes E\\[\vsfd]
a\otimes b\otimes e\ar[r,mapsto]&a\otimes b\otimes 1\otimes e\ar[r,mapsto]&ad_u b\otimes e+d_u a\wedge d_u b\otimes e
\end{tikzcd}
\end{equation}
As for $n=2$, we have that $l^n_{u,E}$ is injective, and vanishes if precomposed with \eqref{eq:universal_jet_higher_edh}.
Let now $\sum_i a_i\otimes b_i\otimes e_i$ be in the kernel of \eqref{eq:universal_jet_higher_edh}.
We proceed similarly to \eqref{eq:DHI-universal}, obtaining
\begin{equation}
0
=\sum_i a_i\otimes b_i\otimes e_i-\sum_i a_ib_i\otimes 1\otimes  e_i
\end{equation}
from the vanishing of the first component of \eqref{eq:universal_jet_higher_edh}.
This shows that
\begin{equation}
\sum_i a_i\otimes b_i\otimes e_i
=\sum_i a_ib_i\otimes 1\otimes  e_i
=l^{n}_{u,E}\left(\sum_i a_ib_i\otimes  e_i\right),
\end{equation}
proving that $l^n_{u,E}$ is the desired kernel map.
Thus, $J^n_u E=A\otimes E$ for all $n$.

By definition, $\iota_{J^n_u E}$ is of the desired form, which has as retraction the map
\begin{align}
A^{\otimes n}\otimes E\to A\otimes E,
&\hfill&
a_1\otimes \cdots\otimes a_n\otimes e\longmapsto a_1\cdots a_n\otimes e.
\end{align}

The projection $\pi^{1,0}_{u,E}$ is the multiplication by definition.
For $n\ge 2$, the projections $\pi^{n,n-1}_{u,E}$ are, by definition, the composition
\begin{equation}
\pi^{n,n-1}_{u,E}(a\otimes e)
=\pi^{1,0}_{u,J^{n-1}_u E}\circ l^n_{u,E}(a\otimes e)
=\pi^{1,0}_{u,J^{n-1}_u E}(a\otimes 1\otimes e)
= a\otimes e.
\end{equation}
Thus $\pi^{n,n-1}_{u,E}=\id_{A\otimes E}$.

The prolongation $j^n_{u,E}$ for $n<2$ has already been discussed in \cite{FMW}.
For $n\ge 2$, it is by definition the factorization through $l^n_{u,E}$ of the map
\begin{align}
j^1_{u,J^{n-1}_u E}\circ j^{n-1}_{u,E}\colon E\longrightarrow J^1_u J^{n-1}_u E,
&\hfill&
e\longmapsto 1\otimes 1\otimes e=l^{n}_{u,E}(1\otimes e).
\end{align}
Thus $j^{n}_{u,E}(e)=1\otimes e$.

We are left to prove that the $n$-jet sequence is exact.
The case $n=1$ holds by definition, cf.\ \cite[\jetsssuniversalonejet]{FMW}.
For $n\ge 2$, we have $S^2_d=0$, and the exactness of the sequence follows from the fact that $\pi^{n,n-1}_{u,E}=\id_{A\otimes E}$ is an isomorphism.
\end{proof}
\begin{rmk}
If $A$ is flat in $\Mod$, we have that $\Omega^1_u$ is flat in $\ModA$, cf.\ \cite[\jetscoruniflat]{FMW}.
Thus $\Omega^n_u=(\Omega^1_u)^{\otimes_A n}$ is also  flat for all $n$.
By Proposition \ref{prop:universal_Spencer-acyclic} we know that $H^{2,n}_{\delta_u}=0$, so we can obtain Theorem \ref{theo:universal_higher_jets} as a corollary of \cite[\jetscorspencerdeltajes]{FMW}.
What Theorem \ref{theo:universal_higher_jets} really shows is that this holds in absolute generality, with no conditions imposed on $A$.
\end{rmk}
\begin{cor}
The universal sesquiholonomic jets coincide with the universal holonomic jets for $n\ge 3$.
\end{cor}
\begin{proof}
It follows from the fact that in the proof of Theorem \ref{theo:universal_higher_jets}, for $n>2$, in order to prove surjectivity of $l^n_{u,E}$ on the kernel of \eqref{eq:universal_jet_higher_edh} we only used $\widetilde{\DH}^I_d=0$.
\end{proof}

Since the universal (holonomic) jets stabilize after degree $1$, we immediately have
\begin{prop}
The infinity universal jet is $J^\infty_u=A\otimes -\colon \AModB\to \AModB$, with
\begin{align}
\pi^{\infty,0}_u=\cdot,
&\hfill&
\pi^{\infty,n}_u=\id_{A\otimes -},
&\hfill&
j^\infty_u=1\otimes -,
&\hfill&
\forall n\ge 1.
\end{align}
\end{prop}

For completeness, we mention the following result concerning differential operators with respect to the universal exterior algebra.
\begin{prop}\label{prop:universal_DOs}
Given an associative unital algebra $A$ equipped with its universal exterior algebra, the differential operators are as follows:
\begin{align}
\Diff^0_u(E,F)=\AHom(E,F),
&\hfill&
\Diff^n_u(E,F)=\Hom(E,F),
&\hfill&
\text{for }n>0.
\end{align}
\end{prop}
\begin{proof}
It follows directly from the following results: \cite[\jetsproponediffuni]{FMW}, \cite[\jetssecDOzero]{FMW}, and \cite[\jetseqDOfiltration]{FMW}.
\end{proof}
This shows once more that the universal calculus does not provide any extra differential information beyond order zero.

\bibliography{../Bibliography}
\bibliographystyle{alpha}
\end{document}